%% file: HFB_revised.tex
\documentclass[11pt, reqno]{amsart}

\usepackage{latexsym,graphics,enumerate}
\usepackage{amsmath,amsfonts,amsthm,amssymb}
\usepackage{mathrsfs}
\usepackage{xcolor}
\usepackage[normalem]{ulem}
\usepackage{ dsfont }
\usepackage{cleveref}

\newtheorem{theorem}{Theorem}[section]
\newtheorem{lemma}[theorem]{Lemma}
\newtheorem{prop}[theorem]{Proposition}
\newtheorem{cor}[theorem]{Corollary}
\theoremstyle{definition}

\theoremstyle{remark}
\newtheorem{remark}[theorem]{Remark}

\newcommand{\Lp}[2]{\left\Vert \, #1 \, \right\Vert_{#2}}

\newcommand{\llp}[1]{ \Vert \, #1 \, \Vert }
\newcommand{\rr}{\mathbb{R}}
\newcommand{\cc}{\mathbb{C}}
\newcommand{\nn}{\mathbb{N}}

\newcommand{\bd}{\partial}
\newcommand{\lapl}{\Delta}
\newcommand{\inprod}[2]{\langle #1, #2 \rangle}

\newcommand{\tr}{\operatorname{tr}}
\newcommand{\grad}{\nabla}
\newcommand{\vect}[1]{\mathbf{ #1 }}

\newcommand{\Tr}{\operatorname{Tr}}
\newcommand{\diag}{\operatorname{diag}}
\newcommand{\ch}{\operatorname{ch}}
\newcommand{\sh}{\operatorname{sh}}

\newcommand{\sv}[1]{|#1 \rangle}
\newcommand{\csv}[1]{\langle #1 |}

\newcommand{\nor}{\operatorname{Nor}}

\numberwithin{equation}{section}

\begin{document}

\title[Dynamical Hartree-Fock-Bogoliubov Approximation]{Dynamical Hartree-Fock-Bogoliubov 
Approximation of Interacting Bosons}


\author{Jacky J. Chong and Zehua Zhao}

\address{Jacky J. Chong
\newline \indent  Department of Mathematics, University of Texas,
    Austin, U.S.A. }
\email{jwchong@math.utexas.edu}

\address{Zehua Zhao
\newline \indent  Department of Mathematics, University of Maryland, College Park, U.S.A. }
\email{zzh@umd.edu}

\thanks{}
\address{}
\curraddr{}
\email{}
\thanks{}




\dedicatory{}

\begin{abstract}
We consider a many-body Bosonic system with pairwise particle interaction given by
$N^{3\beta-1}v(N^\beta x)$ where $0<\beta<1$ and $v$ a non-negative spherically-symmetric
function. Our main result is the extension of the local-in-time Fock space approximation of the exact 
dynamics of squeezed states proved in \cite{GM3} for $0<\beta<\frac{2}{3}$ to a global-in-time 
approximation for $0<\beta<1$. Our work can also be viewed as a generalization of the results in \cite{BCS} to a more general set of initial data that includes coherent states along with an improved error estimate.
The key ingredients in establishing the Fock space approximation are the work of Grillakis and
 Machedon on the the local well-posedness theory \cite{GM4}, some recent established global estimates in \cite{CGMZ}, and our quantitative results on the uniform in $N$ global well-posedness of the time-dependent Hartree-Fock-Bogoliubov (TBHFB) system.  
\end{abstract}
\maketitle

\section{Introduction}
We consider a system of $N$ non-relativistic interacting spinless Bosons in three dimensional
space whose dynamics is governed by the $N$-body linear Schr\"odinger equation\footnote{We work 
under the assumption that
$\hbar=1$ and $2m=1$, where $m$ is the mass. However, it would be interesting to incorporate $\hbar$ in the calculation to see the explicit 
dependence of $\hbar$ in our results. Moreover, as written, \eqref{Linear-Schrodinger} models a system
 of interacting particles in 
the mean-field
scaling. Cf. section 1.8 of \cite{Gol}.}
\begin{align}\label{Linear-Schrodinger}
\left(\frac{1}{i}\frac{\bd}{\bd t}+\sum^N_{j=1}-\lapl_{x_j} + \frac{1}{N}\sum_{i>j}
v_N(x_i-x_j)\right)\Psi_N(t, x_1, \ldots, x_N) = 0 
\end{align}
where $ x_i \in \rr^3$ and  $v_N(x) = N^{3\beta}v(N^\beta x)$. In particular, we work exclusively
in the setting of repulsive interaction.  

The two main goals of this paper are to study of the global-in-time behavior of quantum fluctuations about the mean-field dynamics 
of \eqref{Linear-Schrodinger} and provide a quantitative method for effectively tracking the evolution
of the many-body quantum system in state space. Unfortunately, the 
 problem of tracking the exact dynamics of Bosonic systems in state space with arbitrary initial data is still not 
 tractable with the current available tools. See \cite{LGS} for a in-depth overview of the physics of ultracold Bose gas. Nevertheless, if we restrict ourselves to a special class of 
 initial data, such as coherent states or, more generally, squeezed states, we are able to obtain some positive 
 results in the direction of understanding the exact evolution of the many-body quantum system in state space by means of studying its effective dynamics. 
 
 Another interesting question that one could study is the range of the short-range scaling parameter $\beta$ associated 
 with the interaction potential $v_N$, which was first introduced in \cite{ESY2} as a way to model the different length scales in the problem.
 More precisely, we adopt the
 convention of setting the scale of the system to order one so that the density $\rho$ of the system is of order $N$. 
 Then the interaction potential
 $v_N$ has two length scales: the range of $v_N$ and the scattering length of $v_N$, denoted by $r_N$ and $a_N$, respectively. Here, we see
that $r_N\sim \mathcal{O}(N^{-\beta})$ and $a_N\sim \mathcal{O}(N^{-1})$. When $\beta \in (0, 1)$, we are in the situation
where $a_N\ll r_N\ll 1$. Then the effective interaction potential of $v_N$ is given by $(\int v)\rho$ where $\int v$ is
 the Born approximation of the two-body scattering process. However, when $\beta =1$, also known as the Gross-Pitaevskii regime, we have $r_N\sim a_N$. Then the effective two-body
scattering is just the full two-body scattering process that means the effective potential is given by $8\pi \rho a_N \sim \mathcal{O}(1)$. 
See appendices of \cite{ESY2, Lieb} for more detail. 
 
 Physically speaking, for $\beta \in (0, \tfrac{1}{3})$, we are in the situation where $a_N\ll \rho^{-\frac{1}{3}}\ll r_N$ ($\rho^{-\frac{1}{3}}$ 
 is called the mean inter-particle distance), which model a weakly interacting dense gas. But, once $\beta \in [\frac{1}{3}, 1)$, 
we enter the self-interacting regime or sometimes called the weakly interacting diluted gas regime, that is, 
$a_N\ll r_N\ll\rho^{-\frac{1}{3}}$. A way to see that \eqref{Linear-Schrodinger} models a weakly interacting gas when $\beta \in (0, 1)$ 
is by moving to the microscopic coordinates, i.e. $(t, x)\mapsto (\tau, q):=(N^{2\beta}t, N^{\beta}x)$. Thence, 
we see that the particle system is weakly coupled with the coupling constant being proportional to $N^{\beta-1}$.  
Let us remark that in three dimensional space the case $\beta=1$ is the endpoint case in the heuristic scaling analysis. 
In fact, when $\beta=1$ the system can model a dilute gas with rare but strong collisions, which suggests a 
more refined analysis is needed since, heuristically, we can no longer treat the interactions as a mere perturbation of the non-interaction
case, or, at the very least, not a small perturbation.

In recent years, many have contributed to the studies of effective dynamics for many particle systems. In the case of $\beta=0$ with 
 repulsive Coulomb interactions, Erd\"os and Yau in \cite{EY} proved the qualitative result, via the method of BBGKY hierarchy, that 
 the one-particle marginal density $\gamma^{(1)}_{N, t}$ associated with the wave function $\Psi_{N, t}$ with asymptotically factorized 
 initial state, i.e. $\Psi_{N, 0} \rightarrow \phi^{\otimes N}$ as $N\rightarrow \infty$, converges to $\sv{\phi_t}\csv{\phi_t}$  in trace 
 norm  in the mean-field limit of $N\rightarrow \infty$ where $\phi_t$ satisfies the Hartree equation
 \begin{align}
 \frac{1}{i}\frac{\bd}{\bd t}\phi(t, x) -\lapl_x \phi(t, x) + \left(\frac{1}{|\cdot|}\ast |\phi_t|^2 \right)\phi(t, x) = 0.
 \end{align}
 Using the Fock space method introduced by Hepp in \cite{Hepp} and subsequently extended by Ginibre and Velo in \cite{GV, GV2}, 
 Rodnianski and Schlein in \cite{RS} provided a rate of convergence of the one-particle marginal associated with the many-body 
 quantum system towards the Hartree dynamics in trace norm, that is\footnote{We adopt the standard notation $A \lesssim B$ to 
 mean there exists a constant, depending on some parameters, such that $A \leq CB$.},
 \begin{align}
    \operatorname{Tr}\left|\gamma^{(1)}_{N, t}-\sv{\phi_t}\csv{\phi_t}\right| \lesssim \frac{e^{Kt}}{\sqrt{N}}.
\end{align}
The estimate was later improved to $e^{Kt}N^{-1}$ in \cite{ES, CLS} with the optimal exponent in $N$. Using a second-order correction Fock space method 
introduced by Grillakis, Machedon and Margetis in \cite{GMM, GMM2}, Kuz in \cite{Kuz} provides a rate of convergence of the 
many-body quantum system to the Hartree dynamics in the sense of Fock space marginal density\footnote{One should note that the main 
result in Rodnianski and Schlein's paper is their result on the rate of convergence of the one-particle Fock marginal towards the 
Hartree dynamics. Whereas, the significance of Kuz's paper is that she was able to show that the mean-field estimate is actually 
valid for a much longer period of time then most proceeding results had indicated.}. Consequently, Kuz shown that 
\begin{align}
    \operatorname{Tr}\left|\gamma^{(1)}_{N, t}-\sv{\phi_t}\csv{\phi_t}\right| \lesssim \frac{\sqrt{1+t}}{N^{\frac{1}{4}}}
\end{align}
which in turn establishes the validity of the approximation for time $t$ of the order $\sqrt{N}$. Similar results were derived in \cite{FKS, 
KP} but the approaches are completely different from the above methods. 

For the case $0<\beta\leq 1$, Erd\"os, Schlein and Yau in a series of papers \cite{ESY, ESY2, ESY4, ESY5}  showed qualitatively that 
the many-body dynamics with asymptotically factorized initial data converges to the cubic nonlinear Schr\"odinger dynamics when $0<\beta<1$ or 
the Gross-Pitaevskii dynamics when $\beta =1$. More precisely, they proved that $\gamma^{(1)}_{N, t}$ converges to $\sv{\phi_t}
\csv{\phi_t}$ in trace norm where $\phi_t$
satisfies
\begin{align}
    i\bd_t\phi(t, x) + \lapl_x\phi(t, x) =
    \begin{cases}
    \left(\int v \right)|\phi_t|^2\phi_t & \text{ if } 0<\beta<1\\
    8\pi a|\phi_t|^2\phi_t & \text{ if } \beta =1
    \end{cases}
\end{align}
with $a$ being the scattering length corresponding to the potential $v$. Results on the rate of convergence of Fock space marginals 
can be found in \cite{KP, BDS, Kuz}.

After identifying the mean-field dynamics, it is natural to study the quantum fluctuations about it. 
A natural setting to account for the 
fluctuations is in the Bosonic Fock space
\begin{align*}
    \mathcal{F}_s(\mathfrak{h}) = \cc\oplus\bigoplus_{n\geq 1}\operatorname{Sym}\left(\mathfrak{h}^\otimes\right)
\end{align*}
where $\mathfrak{h}=L^2(\rr^3)$. Introducing $\mathcal{F}_s$ allows us to deal with states with varying number of particles, which in 
our model are the excitation and condensate elements. Recent works on the evolution of coherent states in Fock space with quantum 
fluctuation can be found in \cite{RS, GMM, GMM2, XChen, GM1, GM2,  Kuz, Kuz2, BCS, NN, Chong}.  Hence, by accounting 
for some quantum fluctuations, one is able to track the evolution of the coherent states in Fock space norm, which in effect allows 
one to obtain $L^2$-norm approximation of the evolution of many-body quantum system with factorized initial data. We refer the 
reader to \cite{Lieb, Gol, GMM3} for a complete survey of the subject.

\section{Background and Main Result}
\subsection{Notations and Earlier Results}
In this section, we provide a brief summary of the results obtained in \cite{GM1, GM3, GM4} along with relevant 
notations and rudimentary background materials on second quantization necessary to capture the effects of 
quantum fluctuations about quasifree states\footnote{Cf. Chapter 10 in \cite{Solo} for the definition of quasifree.}. The presentation is standard in
the literature. For reference, the reader may consult \cite{berazin, Solo}. 

 Let us denote the one-particle base space by $\mathfrak{h}:= L^2(\rr^3, dx )$ endowed with the inner product 
 $\inprod{\cdot}{\cdot}_{\mathfrak{h}}$ that is linear
in the second variable and conjugate linear in the first variable. We define the symmetric (Bosonic) Fock space over $\mathfrak{h}$ to be 
the closure of
\begin{align}
\mathcal{F}_s(\mathfrak{h}) = \mathcal{F}_s := \cc\oplus \bigoplus_{n= 1}^\infty \operatorname{Sym}(\mathfrak{h}^{\otimes n})
\end{align}
with respect to the norm induced by the Fock inner product
\begin{align}
\inprod{\varphi}{\psi} = \bar\varphi_0\psi_0 + \sum^\infty_{n=1} 
\inprod{\varphi_n}{\psi_n}_{\mathfrak{h}^{\otimes n}}
\end{align}
for every $\varphi=(\varphi_0, \varphi_1, \ldots), \psi=(\psi_0, \psi_1, \ldots)
 \in \mathcal{F}_s(\mathfrak{h})$. The vacuum, denoted by $\Omega$,
  is defined to be the Fock vector $(1, 0, 0, \ldots ) \in \mathcal{F}_s$.

For every field $\phi \in \mathfrak{h}$, we can define the associated
 creation and annihilation operators on
 $\mathcal{F}_s$, denoted respectively by
 $a^\ast(\phi)$ and $a(\bar\phi)$, as follow
 \begin{subequations}
\begin{align}
(a^\ast(\phi)\psi)_n(x_1, \ldots, x_n) :=&  \frac{1}{\sqrt{n}} \sum^n_{j=1} \phi(x_j)
 \psi_{n-1}(x_1, \ldots, \hat x_j, \ldots, x_n)\\
(a(\bar\phi)\psi)_n(x_1, \ldots, x_n):=& \sqrt{n+1} \int dx\ \bar\phi(x) 
\psi_{n+1}(x, x_1, \ldots, x_n)
\end{align}
\end{subequations}
on each sector with the property that $a(\phi)\Omega = 0$. In particular, we could define the corresponding creation 
and annihilation distribution-valued 
operators denote by $a^\ast_x$ and $a_x$ as
follow
\begin{subequations}
\begin{align}
(a^\ast_x\psi)_n :=&\ \frac{1}{\sqrt{n}} \sum^n_{j=1} \delta(x-x_j) \psi_{n-1}(x_1, \ldots, \hat x_j, \ldots, x_n),\\
(a_x\psi)_n :=&\ \sqrt{n+1}\psi_{n+1}(x, x_1, \ldots, x_n).
\end{align}
\end{subequations}
Hence, we have the relations
\begin{align*}
a^\ast(\phi) = \int dx\ \{\phi(x) a^\ast_x\} \ \ \text{ and } \ \ a(\bar\phi) = \int dx\ \{\bar\phi(x) a_x\}.
\end{align*}
Let us note that the creation and annihilation operators $a(\bar\phi)$ and $a^\ast(\phi)$
 associated with the field $\phi$ are unbounded, densely defined,
closed operators, which are dual to one another. Moreover, one could formally verify the pair $(a^\ast_x, a_x)$ satisfies the canonical commutation relation 
(CCR):
$[a_x, a_y^\ast] = \delta(x-y)$, $[a_x, a_y] = [a_x^\ast, a_y^\ast] = 0$, and 
the number operator defined by
\begin{align}
 \mathcal{N}:= \int dx\ a_x^\ast a_x
\end{align}
is a diagonal operator on $\mathcal{F}_s$ that counts the number of particles in each sector; for instance, $(\mathcal{N}\psi)_n = n\psi_n$.

 For each $\phi \in \mathfrak{h}$, we associate a skew-Hermitian operator
\begin{align}
 \mathcal{A}(\phi):= a(\bar\phi) - a^\ast(\phi)
\end{align}
and define  the corresponding Weyl operator to be the unitary operator $e^{-\sqrt{N}\mathcal{A}(\phi)}$.
Then the coherent state associated with $\phi$
is given by 
\begin{align}
 \psi(\phi):=e^{-\sqrt{N}\mathcal{A}(\phi)}\Omega.
\end{align}
Using the Baker-Campbell Hausdorff formula, one can show 
\begin{align*}
\psi(\phi) = \left(\ldots, c_n \phi^{\otimes n}, \ldots\right) \ \ 
\text{ with } \ \ c_n = \left( e^{-N\llp{\phi}^2_\mathfrak{h}}N^n/n!\right).
\end{align*}
In this context, $\phi$ is called the condensate wave function and the Fock space marginal associated with
the approximation scheme $e^{-\sqrt{N}\mathcal{A}(\phi_t)}\Omega$ where $\phi_t =\phi(t)$ satisfies the cubic NLS offers
a first-order (mean-field) approximation to exact evolution of the coherent state in trace norm; see \cite{RS}.  
However, to track  the exact dynamics in Fock space, a first-order approximation is not sufficient. 
We need to introduce the pair excitation 
function $k(x, y)=k(y, x)$ and its corresponding quadratic operator $\mathcal{B}(k)$, called the squeezed operator, given by
\begin{align}
\mathcal{B}(k) := \int dx d y\ \{\bar k(x, y)a_x a_y -
k(x, y)a^\ast_x a^\ast_y\}.
\end{align}
From the pair excitation function, we concoct a new approximation scheme, a second-order correction\footnote{In the mathematical 
physics literature, $e^{\mathcal{B}}$ is called the infinite-dimensional Segal-Shale-Weil representation 
of the double cover of the group of symplectic matrices of integral operators. The elements of the 
corresponding $C^\ast$-algebra are called 
Bogoliubov transformations (cf. Chapter 4 of \cite{Folland} and Chapter 11 of \cite{DezGer}).  } 
to the mean-field approximation,   given by
\begin{align}\label{approx}
\psi_{\text{approx}} = e^{i\chi (t)}e^{-\sqrt{N}\mathcal{A}(\phi_t)}e^{-\mathcal{B}(k_t)}\Omega 
\end{align}
where $\chi(t)$ is some phase factor to be determined. With an appropriate choice of evolution
equations for $\phi$ and $k$, 
we will see that 
\eqref{approx} will indeed allows us to track the exact dynamics of the evolution of quasifree states 
in Fock space. Hence we also refer to \eqref{approx} as the quasifree approximation. 

For a fixed $N \in \nn$, we are interested in the time evolution generated by
the Fock Hamiltonian operator\footnote{We adopted the convention that $\mathcal{H} \leq 0$, a non-positive operator, which could make some of
the signs in the paper not conventional when compare to the mathematical physics literature.} associated 
with $N$ on the Fock space, which is a diagonal operator denoted by $\mathcal{H}_N$,
defined by its action on Fock vectors of the form $\psi = (0, \ldots, 0, \Psi_n(x_1, \ldots, x_n), 0, \ldots)$
\begin{align}
 (\mathcal{H}_N\psi)_n = \left(\sum^n_{j=1}\lapl_{x_j}
-\frac{1}{N}\sum_{i<j}^nv_N(x_i-x_j)\right)\Psi_n=:H_{N, n}\Psi_n.
\end{align}
For convenience, we drop the $N$ subscript in $\mathcal{H}_N$ since it should be clear from the context. 
Rewrite $\mathcal{H}$ using creation and annihilation operators yields
\begin{subequations}
\begin{align}
 \mathcal{H} :=&\ \mathcal{H}_0-N^{-1}\mathcal{V} \ \ \text{ where }\\
  \mathcal{H}_0:=&\ \int  dxdy\ \{\lapl_x \delta(x-y) a^\ast_x a_y\}  \ \ \text{ and }\\
  \mathcal{V}:=&\ \frac{1}{2}\int dxdy\ \{v_N(x-y)a^\ast_x a^\ast_y a_xa_y\}.
\end{align}
\end{subequations}
In light of the Fock Hamiltonian, we are interested in studying the
solution to the following Cauchy problem
in Fock space
\begin{equation}\label{Fock-Cauchy}
\begin{aligned}
 \vect{S}_D \psi :=&\ \left(\frac{1}{i}\frac{\bd}{\bd t} - \mathcal{H}\right)\psi=0 \ \ \text{ with }\\
 \psi(0) =&\ \psi_0=\ e^{-\sqrt{N}
\mathcal{A}(\phi_0)}e^{-\mathcal{B}(k_0)}\Omega
\end{aligned}
\end{equation}
which we shall formally write as 
\begin{align}\label{exact}
 \psi_{\text{exact}}(t)= e^{it\mathcal{H}}e^{-\sqrt{N}\mathcal{A}(\phi_0)}e^{-\mathcal{B}(k_0)}\Omega.
\end{align}
Here $\psi_0$ is usually referred to as a pure squeezed state, or more generally, a quasifree state. 

Let $\mathcal{M} = e^{-\sqrt{N}\mathcal{A}}e^{-\mathcal{B}}$. 
Following \cite{GM1, GM3}, we work with the reduced dynamics, which is also called the fluctuation dynamics. 
More specifically, since $\mathcal{M}$ is a unitary operator then the error of the quasifree approximation
can be rewritten in terms of the reduced dynamics, that is, 
\begin{align}
    \Lp{E}{\mathcal{F}}:=&\ \Lp{\psi_{\text{exact}}(t)-\psi_{\text{approx}}(t)}{\mathcal{F}}= \Lp{\psi_{\text{red}}(t)-\Omega}{\mathcal{F}}
\end{align}
where
\begin{align}
    \psi_{\text{red}}(t) = e^{\mathcal{B}(t)}e^{\sqrt{N}\mathcal{A}(t)}
    e^{it\mathcal{H}} e^{-\sqrt{N}\mathcal{A}(\phi_0)}e^{-\mathcal{B}(k_0)}\Omega
\end{align}
is called the reduced Fock vector. Therefore, to study the error, it suffices to study the evolution of the reduced dynamics. 

By direct computation, one can show that the evolution of $\psi_{\text{red}}$ is determined by the equation
\begin{align}\label{reduced-ham}
    \frac{1}{i}\frac{\bd}{\bd t} \psi_{\text{red}}  = \mathcal{H}_{\text{red}}\psi_{\text{red}}
    \ \ \text{ where } \ \ \mathcal{H}_{\text{red}} = \frac{1}{i}(\bd_t \mathcal{M}^\ast)\mathcal{M}
    + \mathcal{M}^\ast\mathcal{H}\mathcal{M}.
\end{align}
 $\mathcal{H}_\text{red}$ is called the reduced Hamiltonian. The reader should note that $\mathcal{H}_\text{red}$
is a fourth degree polynomial in $a$ and $a^\ast$. Then we see that the error $E$ satisfies
the inhomogeneous equation\footnote{Here we made a slight abused of notation by identifying $E$ with both $\psi_\text{exact}-\psi_\text{approx}$ and $\psi_\text{red}-\Omega$. 
However, since our analysis only involves studying $\psi_\text{red}-\Omega$, the reader can safely assume $E=\psi_\text{red}-\Omega$
for the remainder of the paper.}
\begin{align}
  \left( \frac{1}{i}\frac{\bd}{\bd t}-\mathcal{H}_{\text{red}}\right)E = \mathcal{H}_{\text{red}}\Omega 
\end{align}
where
\begin{align}\label{forcing}
    \mathcal{H}_{\text{red}}\Omega = (X_0, X_1, X_2, X_3, X_4, 0, 0, \ldots).
\end{align}

Thus, to estimate the error and show that it is small, for large $N$, we need to be able to control the 
forcing term $\mathcal{H}_{\text{red}}\Omega$ uniformly in $N$.
However, it has been shown in \cite{GM1, GM2} that $X_0, X_1$ and $X_2$ are heuristically large since they are proportional to 
$N, N^{\frac{1}{2}}$ and constant, respectively.  On the other hand, $X_3$ and $X_4$ are proportional to $N^{-\frac{1}{2}}$ 
and $N^{-1}$, respectively, making them heuristically small when $N$ is sufficiently large. 

Fortunately, $X_0$ does not pose any serious problem. In fact, by simply modifying our quasifree approximation by
a phase factor is sufficient to get rid of $X_0$ in the forcing term \eqref{forcing}, that is, if we consider the modified approximation
\begin{align}
 \tilde \psi_\text{red}(t) = e^{-i\int^t_0 X_0(s)\ ds}\psi_\text{red}(t)
\end{align}
then we see that the modified error $\tilde E$ satisfies the equation
\begin{align}\label{Fock-error-eq}
\vect{S}_F \tilde E :=  \left( \frac{1}{i}\frac{\bd}{\bd t}-\mathcal{H}_{\text{red}}+X_0\right)(\tilde\psi_\text{red}-\Omega)
= \mathcal{H}_\text{red}\Omega - X_0\Omega.
\end{align}
The explicit form of $X_0$, also denoted by $X_0:=-N\mu$, is given in Proposition 6.1 of \cite{GM1}. Next, since we have the two degrees of freedom $\phi$ and $k$ from our approximation scheme, then we can handle $X_1$ and $X_2$
by simply imposing the natural conditions that $\phi_t$ and $k_t$ must satisfy the abstract equations 
\begin{align}\label{abstract-eq}
X_1 = 0 \ \ \text{ and } \ \ X_2 = 0.
\end{align}
These equations were first explicitly computed in Theorem 7.1 of \cite{GM1} (cf. appendix C of \cite{BBCFS}).   

To compactly write down the evolution equations for $\phi_t$ and $k_t$, we need to introduce some notations.

It turns out that the equations for $k$ is most appropriately stated in terms of the following set of fields
\begin{subequations}
\begin{align}
 \sh(k):=&\ k + \frac{1}{3!}k\circ \bar k \circ k + \frac{1}{5!} k\circ \bar k \circ k \circ \bar k \circ k +\ldots \\
 \ch(k):=&\ \delta + \frac{1}{2!}\bar k \circ k+\frac{1}{4!}\bar k \circ k \circ \bar k \circ k+\ldots
\end{align}
\end{subequations}
where $\circ$ indicates composition, that is, the kernel of $k\circ l$ is given by
\begin{align}
 (k\circ l)(x, y) = \int dz\ k(x, z) l(z, y).  
\end{align}
Adopting the convention of \cite{GM1}, we also use the notation $u = \sh(k)$ and $c = \ch(k) = \delta +p$.

Moreover, we define the operator kernels
\begin{subequations}\label{def-k-eqs}
\begin{align}
 \Lambda(t, x, y) :=&\ \phi(t, x)\phi(t, y)+ \frac{1}{2N}\sh(2k)(t, x, y)\label{def-lambda}\\
 \Gamma(t, x, y) :=&\ \bar\phi(t, x)\phi(t, y) + \frac{1}{N}
 \left(\overline{\sh(k)} \circ \sh(k)\right)(t, x, y)\label{def-gamma}
\end{align}
\end{subequations}
which will play paramount roles in our subsequent discussion. To reconcile with some of the notations
in the literature, in particular the notations of \cite{BBCFS}, we also adopt the notation 
$\sigma :=u\circ c =\frac{1}{2}\sh(2k)$ and $\gamma := u\circ \bar u = \frac{1}{2}(\overline{\ch(2k)}-\delta)$.
It is convenient to consider the
trace density\footnote{When an integral operator $F$ is 
positive-definite, such as $\Gamma$, $\diag F$, denoted the restriction of $F$ to the diagonal, is the trace density of $F$, that is $\tr F = \int F(x, x)\ dx$.
In the physics literature, $N\diag\Gamma(x)$ is called the total-number density  and is often denoted by $n(x)$. Here we adopt the 
standard notation,
$n_c$ and $\tilde n$ denote the condensate  density and the non-condensate  density, respectively, i.e. $n(x) = n_c(x)+\tilde n(x)$.} of $\Gamma$
\begin{equation}
\begin{aligned}
\diag\Gamma(t, x) =&\ \frac{1}{N}[n_c(t, x) + \tilde n(t, x)]:=  |\phi(t, x)|^2+ 
 \frac{1}{N}\left(\bar u \circ u\right)(t, x, x) 
\end{aligned}
\end{equation}
and define the operator kernel
\begin{subequations}
 \begin{align}
 \alpha(t, x, y) :=&\ (v_N\ast\diag\Gamma)(t, x)\delta(x-y)+v_N(x-y)\Gamma(t, x, y)\\
 =&\ (v_N\ast |\phi|^2)(t, x)\delta(x-y) + v_N(x-y)\bar\phi(t, x)\phi(t, y) \nonumber\\
 &\ +\frac{1}{N}\Big\{ (v_N\ast \tilde n)(t, x)\delta(x-y)+v_N(x-y)(\bar u\circ u)(t, x, y)\Big\}\nonumber\\
 =:&\ \alpha_c(t, x, y)+N^{-1}\tilde \alpha(t, x, y) 
\end{align}
\end{subequations}
and the multiplication
\begin{align}
 (v_NF)(t, x, y) := v_N(x-y)F(t, x, y).
\end{align}
Let us also define the following Schr\"odinger-type operators 
\begin{align}\label{schrodinger-op}
 \vect{S}_\pm = \frac{1}{i}\frac{\bd}{\bd t}-\lapl_x-\lapl_y \ \ \text{ and } \ \  \vect{S} = \frac{1}{i}\frac{\bd}{\bd t}-\lapl
\end{align}
where the latter is defined for either 6+1  or 3+1 dimensions, which will be clear from the context, and their corresponding operators
with ``potentials'': for any function $s$ symmetric in $(x, y)$, i.e. $s(x, y) = s(y, x)$, and any function $p$ conjugate symmetric in $(x, y)$,
i.e. $\bar p(y, x) = p(x, y)$ or $\bar p^T = p$ where $p^T(x, y) = p(y, x)$ denotes the transpose operator, we define 
\begin{subequations}
\begin{align}
 \vect{\tilde S}(s) :=&\ \vect{S}(s)+[\alpha^T, s]_+\\
 \vect{\tilde W}(p) :=&\ \vect{S}_\pm(p)+[\alpha^T, p]
\end{align}
where $[A, B]:=A\circ B-B^\ast\circ A^\ast$ is the (skew-Hermitian) commutator and $[A, B]_+:=A\circ B+B^T\circ A^T$ is the symmetrization
of the two operators $A$ and $B$. 
\end{subequations}

Then we have the following theorem.
 \begin{theorem}[Theorem 7.1 of \cite{GM1}]
  The equation $X_1 =0$ is equivalent to
  \begin{subequations}\label{HFB1}
  \begin{equation}
  \begin{aligned}\label{HFB1-phi}
\vect{S}\phi_t(x_1)=&\ -\int dy\ \left\{(v_N\Lambda_t)(x_1, y)\bar\phi_t(y)+\frac{1}{N}\tilde\alpha^T (t, x_1, y)\phi_t(y) \right\}. 
  \end{aligned} 
  \end{equation}
  The equation $X_2=0$ is equivalent to the pair of equations
\begin{align}
\vect{\tilde S}(\sh(2k_t)) + [v_N\Lambda_t, \ch(2k_t)]_+=&\ 0, \label{compact-sh-eq}\\
\vect{\tilde W}(\overline{\ch(2k_t)}) + [v_N\Lambda_t, \overline{\sh(2k_t)}] =&\ 0 \label{compact-ch-eq}.
\end{align}
\end{subequations}
 \end{theorem}

 By imposing the system of equations \eqref{HFB1}, we have reduced \eqref{Fock-error-eq} to
 \begin{align}
  \vect{S}_F\tilde E = (0, 0, 0, X_3, X_4, 0,\ldots) 
 \end{align}
which we could apply energy method (or Strichartz) to estimate $\tilde E$ in Fock space.  If we write out $X_3$ and $X_4$ explicitly, 
up to symmetrization, normalization and some lower-order terms, 
we see that
\begin{subequations}
\begin{align}
 &X_3(t, x_1, x_2, x_3) \\
 &\sim  \int \frac{dy_1dy_2 }{\sqrt{N}}\ \overline{\ch(k_t)}(x_1, y_1)\ch(k_t)(y_2, x_2)v_N(y_1-y_2)\phi_t(y_2)\sh(k_t)(x_3, y_1) \nonumber
\end{align}
and
\begin{align}
 &X_4(t, x_1, x_2, x_3, x_4) \\
 &\sim \int \frac{dy_1 dy_2}{N}\ \overline{\ch(k_t)}(x_1, y_1)\ch(k_t)(y_2, x_2)v_N(y_1-y_2)\phi_t(y_2)\sh(k_t)(x_3, y_1). \nonumber
\end{align}
\end{subequations}
The complete list of the error terms can be found in Section 5 of \cite{GM2}.
Hence we would like to obtain uniform in $N$ estimates for $\sh(k)$ and $\ch(k)$, which will in term show that 
the Fock space error tends to zero in the particle limit $N\rightarrow \infty$. However, due to the singular nature of the
forcing term $v_N\Lambda$ in 
\eqref{compact-sh-eq} and \eqref{compact-ch-eq}, this poses some inconvenience when trying to estimate $\sh(2k)$ and $\ch(2k)$ uniformly
in $N$ for $\beta \in (0, 1)$. It was shown in \cite{GM3} 
that $k$ is best expressed in terms of auxiliary fields $\Gamma$ and $\Lambda$, which are functions defined
by \eqref{def-k-eqs} in terms of $k$ and $\phi$. In fact, it was shown that $\Gamma$ and $\Lambda$ can be viewed as
the ``generalized'' (Fock space) marginal density matrices for the pure quasifree state $\psi_0$ in \eqref{Fock-Cauchy}. 

\subsection{Hartree-Fock-Bogoliubov System} To write down the alternative system of nonlinear equations to \eqref{HFB1}, we define the kernel of 
the marginal density matrices $\mathcal{L}_{m, n}$ associated with the state $\psi_\text{approx}$, which
is sometimes called the $(m+n)$-point correlation function, as follows
\begin{equation}
\begin{aligned}
 &\mathcal{L}_{m, n}(t, y_1, \ldots, y_m; x_1,\ldots, x_n)\\
 &:= \frac{1}{N^{(m+n)/2}}\inprod{a_{y_1}\cdots a_{y_m}\psi_\text{approx}}{ a_{x_1}\cdots a_{x_n}\psi_\text{approx}}\\
 &= \frac{1}{N^{(m+n)/2}} \inprod{\mathcal{M}\Omega}{\mathcal{P}_{m,n}\mathcal{M}\Omega}
 \end{aligned}
\end{equation}
where $\mathcal{P}_{m, n} = a_{y_1}^\ast\cdots a_{y_m}^\ast a_{x_1}\cdots a_{x_n}$.
It is clear from the definition that $\mathcal{L}_{m, n}^\ast = \mathcal{L}_{n, m}$. It is also straightforward to write
down the evolution equations for $\mathcal{L}_{m, n}$ using \eqref{reduced-ham}, that is, 
\begin{subequations}
\begin{align}
 \frac{1}{i}\frac{\bd}{\bd t}\left(\mathcal{M}^\ast\mathcal{P}\mathcal{M} \right)
 = [\mathcal{H}_\text{red}, \mathcal{M}^\ast\mathcal{P}\mathcal{M}]+\mathcal{M}^\ast [\mathcal{P}, \mathcal{H}]\mathcal{M}
\end{align}
which means
\begin{align}\label{gen-marginal-evol}
 \frac{1}{i}\frac{\bd}{\bd t}\mathcal{L} = \frac{1}{N^{(m+n)/2}}
 \left(\inprod{\Omega}{[\mathcal{H}_\text{red}, \mathcal{M}^\ast\mathcal{P}\mathcal{M}]\Omega}
 +\inprod{\Omega}{\mathcal{M}^\ast[\mathcal{P}, \mathcal{H}]\mathcal{M}\Omega} \right).
\end{align}
\end{subequations}
Notice the set of  equations \eqref{gen-marginal-evol} for different pairs $(m, n)$ forms an 
infinite linear hierarchy of coupled equatons where the evolution of each point correlation function
depends on other higher-order point correlation functions. In the literature, \eqref{gen-marginal-evol} is 
refer to as the von Neuman-Landau equation. 

Since $\psi_\text{approx}$ is 
a quasifree state, then it can be shown that any $(n+m)$-point correlation functions for $n+m\geq 3$ can be expressed
in terms of the one-particle or two-particle Fock marginal densities, $\mathcal{L}_{0, 1}, \mathcal{L}_{1, 1}$ and $\mathcal{L}_{0, 2}$.    
Hence it suffices for us to study the evolution equations for $\mathcal{L}_{0, 1}, \mathcal{L}_{1, 1}$ 
and $\mathcal{L}_{0, 2}$. In fact, it was shown in \cite{GM3} that $\mathcal{L}_{0, 1} = \phi, \mathcal{L}_{1, 1} = \Gamma, \mathcal{L}_{0, 2} = \Lambda$ 
 and that their evolution equations form a closed system of nonlinear coupled PDEs given by
\begin{subequations}\label{TDHFB}
 \begin{alignat}{2}
&\left\{\frac{1}{i}\frac{\bd}{\bd t}-\lapl_{x_1}\right\}\phi(x_1) =&&\ -\int dy\ \{v_N(x_1-y)\Gamma(y, y)\} \phi(x_1) \label{phi-eq}\\
  &\quad        &&\ - \int dy\ \{v_N(x_1-y)(\Gamma(y, x_1)-\bar\phi(y)\phi(x_1))\phi(y)\} \nonumber \\
  &\quad         &&\ - \int dy\ \{v_N(x_1-y)(\Lambda(x_1, y)-\phi(y)\phi(x_1))\bar\phi(y)\} \nonumber
  \end{alignat}
  \begin{align}
&\left\{ \frac{1}{i}\frac{\bd}{\bd t}-\lapl_{x_1}+\lapl_{x_2}\right\} \overline\Gamma(x_1, x_2) \label{gamma-eq} \\
 & =  - \int dy\ \{(v_N(x_1-y)-v_N(x_2-y))\Lambda(x_1, y)\overline\Lambda(y, x_2)\}\nonumber\\
&\quad -\int dy\ \{(v_N(x_1-y)-v_N(x_2-y))(\overline\Gamma(x_1, y)\overline\Gamma(y, x_2)+\overline\Gamma(y, y)\overline\Gamma(x_1, x_2))\}\nonumber\\
&\quad  + 2\int dy\ \{(v_N(x_1-y)-v_N(x_2-y))|\phi(y)|^2\phi(x_1)\bar\phi(x_2)\}\nonumber
\end{align}
\begin{align}
&\bigg\{ \frac{1}{i}\frac{\bd}{\bd t}-\lapl_{x_1}-\lapl_{x_2}+\frac{1}{N}v_N (x_1-x_2)\bigg\} \Lambda (x_1, x_2) \label{Lambda-eq}\\
&= - \int dy\ \{(v_N(x_1-y)+v_N(x_2-y))\Gamma(y, y)\Lambda(x_1, x_2)\}\nonumber\\
&\quad -\int dy\ \{(v_N(x_1-y)+v_N(x_2-y))(\Lambda(x_1, y)\Gamma(y, x_2)+\overline\Gamma(x_1, y)\Lambda(y, x_2))\}\nonumber\\
&\quad + 2\int dy\ \{(v_N(x_1-y)+v_N(x_2-y))|\phi(y)|^2\phi(x_1)\phi(x_2)\}.\nonumber
\end{align}
\end{subequations}
Notice we have suppressed the time dependence to compactify the notation. 
We refer to \eqref{TDHFB} as the time-dependent Hartree-Fock-Bogoliubov (TDHFB) system. 

\begin{remark}
 In the literature, the name TDHFB equations typically refers to the coupled system \eqref{HFB1} where 
 the pair interaction is given by the contact interaction potential $v(x) = g_0\delta(x)$ for some coupling constant $g_0\geq 0$ (In the 
 context of diluted gases, one has $g_0 = 4\pi a$ where $a$ is some scattering length).
Furthermore, it is standard to express the equation for $\phi$ in terms of $\Phi := \sqrt{N}\phi$ then \eqref{HFB1-phi} comes
\begin{subequations}
\begin{align}
 \vect{S}\Phi(x_1) = -g_0\left[n_c(x_1)+2\tilde n(x_1)\right]\Phi(x_1)-g_0\tilde m(x_1)\overline\Phi(x_1) 
\end{align}
where $\tilde m(x_1) :=\sigma(x_1, x_1)= N\Lambda(x_1, x_1)-|\Phi(x_1)|^2$. The other two equations 
\eqref{compact-sh-eq} and \eqref{compact-ch-eq} can be neatly rewritten in terms of a matrix equaton of $\gamma$ and $\sigma$
as follows
\begin{align}
 \frac{1}{i}\frac{\bd}{\bd t}R  = [R, \Im^\ast]
\end{align}
\end{subequations}
where $R$ is the $2\times 2$ matrix of operators with matrix kernel
\begin{subequations}
\begin{align}
 R(x, y) = 
 \begin{pmatrix}
  \gamma(x, y) & \sigma(x, y)\\
  \bar\sigma(x, y) & \delta(x-y) +  \bar\gamma(x, y)
 \end{pmatrix}
\end{align}
and $\Im$ has the matrix kernel
\begin{equation}
\begin{aligned}
\Im(x, y) =&\
\begin{pmatrix}
 1 & 0\\
 0 & -1
\end{pmatrix}
 \begin{pmatrix}
 G(x, y) & M(x, y)\\
 \bar M(x, y) & \bar G(x, y)
 \end{pmatrix} 
\end{aligned}
\end{equation}
\end{subequations}
with
\begin{subequations}
\begin{align}
 G(x, y) :=&\ -\lapl_x\delta(x-y) +2g_0[n_c(x)+\tilde n(x)]\delta(x-y)\\
 M(x, y) :=&\   g_0[\Phi(x)^2+\tilde m(x)]\delta(x-y). 
\end{align}
\end{subequations}
Moreover, the physical interpretation of $(\Phi, \gamma, \sigma)$ (or equivalently, $(\phi, \Gamma, \Lambda)$) is
as follows: The one-particle wave function $\Phi$ is called the condensate wave function which describes the behavior
of the Bose-Einstein condensate (BEC). On the other hand, $\gamma$ and $\sigma$  describe the dynamics of sound waves, or quasiparticles,  
in the quasifree approximation. In particular, $\tilde n(x)  = \gamma(x, x)$ determines the density of the ``thermal cloud''
of atoms, i.e. the excitation density of the Bose gas. In the physics literature, $\tilde n(x) = \gamma(x, x)$
is called the non-condensate density and $\tilde m(x) = \sigma(x, x)$
is called the anomalous density. 
\end{remark}

\begin{remark}
 It has been shown in \cite{BBCFS} that the evolution equations for the higher Fock marginal densities can 
 be written in terms of an effective nonlinear quadratic Hamiltonian parametrized by $(\phi,\Gamma, \Lambda)$
 called the Hartree-Fock-Bogoliubov (HFB) Hamiltonian, denoted by $\mathcal{H}_\text{HFB}$; more precisely, we have
 the evolution equation
 \begin{align}\label{higher-evol}
  \frac{1}{i}\frac{\bd}{\bd t}\mathcal{L}_{m, n} = 
  \inprod{\mathcal{M}\Omega}{[\mathcal{H}_\text{HFB}, \mathcal{P}_{m, n}]\mathcal{M}\Omega}.
 \end{align}
The HFB Hamiltonian is given explicitly by
\begin{align}
 \mathcal{H}_\text{HFB} = -N\mu + \sqrt{N}\mathcal{H}_1 + \mathcal{H}_2
\end{align}
with
\begin{subequations}
\begin{align}
 \mathcal{H}_1 =&\ 2 \int dx\ \{(v_N\ast |\phi|^2)\phi(x)a^\ast_x+ (v_N\ast |\phi|^2)\bar\phi(x)a_x\} \\
 \mathcal{H}_2 =&\ -\frac{1}{2}\int dxdy\ \{2g^T(x, y)a_x^\ast a_y + m(x, y)a_x^\ast a_y^\ast +\bar m(x, y) a_xa_y\}
\end{align}
\end{subequations}
where
\begin{subequations}
 \begin{align}
  g(x, y) :=&\ -\lapl_x\delta(x-y) + (v_N\ast\diag\Gamma)(x)\delta(x-y) +(v_N\Gamma)(x, y) \label{g-function}\\
  m(x, y) :=&\ (v_N\Lambda)(x, y) = v_N(x-y)\Lambda(x, y).
 \end{align}
\end{subequations}
Also note, by \eqref{gen-marginal-evol} and \eqref{higher-evol}, we have the relation
\begin{align}
 \mathcal{H}+\mathcal{H}_\text{HFB}=\mathcal{M}\mathcal{H}_\text{red}\mathcal{M}^\ast.
\end{align}
\end{remark}

Since we are interested in studying the uniform in $N$ error bound of the quasifree approximation, it is then 
instructive to take the formal large particle limit, $N\rightarrow \infty$, to obtain the following equations
\begin{subequations} \label{limit-eqs}
 \begin{align}
\vect{S}\phi(x_1) =&\ -2g_0\{\Gamma(x_1, x_1)-|\phi(x_1)|^2\}\phi(x_1)-g_0\Lambda(x_1, x_1)\bar\phi(x_1)
  \end{align}
  \begin{equation}
  \begin{aligned}
   \vect{S}_\pm \overline\Gamma(x_1, x_2) =& -g_0\{\Lambda(x_1, x_1)\overline\Lambda(x_1, x_2)
   -\Lambda(x_1, x_2)\overline\Lambda(x_2, x_2)\}\\
   & - 2g_0\{\overline\Gamma(x_1, x_1)-\overline\Gamma(x_2, x_2)\}\overline\Gamma(x_1, x_2)\\
   & + 2g_0\{|\phi(x_1)|^2-|\phi(x_2)|^2\}\phi(x_1)\bar\phi(x_2)
  \end{aligned}
\end{equation}
  \begin{equation}
  \begin{aligned}
   \vect{S} \Lambda(x_1, x_2) =& - g_0\{\Lambda(x_1, x_1)\Gamma(x_1, x_2)+\Lambda(x_2, x_2)\overline\Gamma(x_1, x_2)\}\\
   & -2g_0\{\Gamma(x_1, x_1)
   +\Gamma(x_2, x_2)\}\Lambda(x_1, x_2)\\
   & + 2g_0\{|\phi(x_1)|^2+|\phi(x_2)|^2\} \phi(x_1)\phi(x_2)
  \end{aligned}
\end{equation}
   \end{subequations}
   for some constant $g_0$. In many ways, system \eqref{limit-eqs} serves as a guiding light in our studies
   of the uniform in $N$ estimates for the TDHFB system. In particular, as remarked in \cite{GM4}, the solution to
  the system \eqref{limit-eqs} is invariant under the scaling $\lambda\phi(\lambda^2 t, \lambda x_1, \lambda x_2), 
   \lambda^2\Gamma(\lambda^2 t, \lambda x_1, \lambda x_2), \lambda^2\Lambda(\lambda^2 t, \lambda x_1, \lambda x_2)$,
   which makes the limiting system energy critical. However, the reader should also note that the scaling argument
   does not detect the collapse to the diagonal $x_1 = x_2$ involved in these equations. More precisely, by replacing 
   $\Lambda(x_1, x_1)\bar \Lambda(x_1, x_2)$ with $|\Lambda(x_1, x_2)|^2$, the system would still scale the same way,
   but it is well known that going from $|\Lambda(x_1, x_2)|^2$ to $|\Lambda(x_1, x_1)|^2$ 
   there is a natural loss of regularity; see for instance the standard sharp trace theorem or Lemma 5.1 in \cite{GM3} for the collapsing
   estimates of the linear $\Lambda$ and $\Gamma$ equations. 
   Thus, it is arguable that the system \eqref{limit-eqs} is actually energy-supercritical, which illustrates the 
   formidability of obtaining uniform in $N$ global well-posedness and estimates 
   for the TDHFB system \eqref{TDHFB} in energy space. Hence it seems desirable to first consider
   the well-posedness of $\Gamma$ and $\Lambda$ in $H^{1+\varepsilon}$ Sobolev spaces for $\varepsilon>0$.
   But even so, the situation is not completely obvious since the conserved energy of the system
   scales uniform in $N$ like $H^1$ for $\Lambda$ which poses the classical questions regarding 
   the growth of $H^{1+\varepsilon}$ Sobolev norm of $\Lambda$. This question regarding the growth 
   of the Sobolev norm of $\Lambda$ was recently studied by the authors in a joint work with M. Grillakis and M. Machedon in \cite{CGMZ}.

A more pressing question one might have is whether the TDHFB system is even uniform in $N$ locally well-posed. 
Indeed, the uniform in $N$ local well-posedness of the TDHFB system was proven 
in \cite{GM4} for $0<\beta<1$. Let us summarize their result as follows.

We begin by defining the space of initial data. For every $\alpha>0$, we define the space
\begin{align}
 \mathcal{X}^\alpha = \{(\phi, \Gamma, \Lambda) \in H^\alpha\times H^\alpha_\text{Herm}\times H^\alpha_\text{Sym}\}
\end{align}
with $H^\alpha$ being the Sobolev space $H^\alpha(\rr^3)$, $H^\alpha_\text{Herm}$ the Sobolev space 
$H^\alpha(\rr^3\times\rr^3)$ restricted to functions $\Gamma$ such that $\Gamma(x, y) = \overline{\Gamma(y, x)}$, and 
$H^\alpha(\rr^3\times\rr^3)$ restricted to functions $\Lambda$ such that $\Lambda(x, y) = \Lambda(y, x)$. Moreover,
$\mathcal{X}^\alpha$ is endowed with the norm
\begin{align}
 \llp{(\phi, \Gamma, \Lambda)}_{\mathcal{X}^\alpha}:=&\ \Lp{\langle\grad\rangle^\alpha\phi}{L^2(\rr^3)}+\Lp{(\langle\grad_x\rangle^2\otimes 1
 +1\otimes\langle\grad_y\rangle^2)^{\frac{\alpha}{2}}\Gamma}{L^2(\rr^3\times\rr^3)} \nonumber\\
 &\ +\Lp{(\langle\grad_x\rangle^2\otimes 1
 +1\otimes\langle\grad_y\rangle^2)^{\frac{\alpha}{2}}\Lambda}{L^2(\rr^3\times\rr^3)} 
\end{align}
where $\langle D\rangle = \sqrt{1+|D|^2}$ is the standard bracket notation.

\begin{theorem}[Grillakis $\&$ Machedon '18]\label{GM-main1}
 Assume $v$ is a Schwartz function with $\hat v$ supported in the unit ball, such that $|\hat v| \leq \hat w$ with
 $w$ a Schwartz function. Let $0<\beta<1$. Fix $\alpha>\frac{1}{2}$ so that $2\alpha\beta<1$ and choose $R>0$.
 Then there exists $T=T(\beta, R)$, independent of $N$, and a corresponding spacetime function space $\mathcal{X}_T$, 
 depending only on $T$ and $\alpha$, such that for any given initial data
 \begin{align*}
  (\phi_0, \Gamma_0, \Lambda_0) \in \{(\phi, \Gamma, \Lambda) \in \mathcal{X}^\alpha \mid \llp{(\phi, \Gamma, \Lambda)}_{\mathcal{X}^\alpha}<R\},
 \end{align*}
 there exists a unique maximal mild solution to the time-dependent HFB system \eqref{TDHFB}, with initial data 
 $(\phi_0, \Gamma_0, \Lambda_0)$, satisfying 
 \begin{align}
 (\phi_t, \Gamma_t, \Lambda_t) \in C([0, T]\rightarrow \mathcal{X}^\alpha)\cap \mathcal{X}_T.
 \end{align}
\end{theorem}

Theorem \ref{GM-main1} is an improvement of their previous local well-posedness result in \cite{GM3} for the case $0<\beta<\frac{2}{3}$.
The analysis of the case $0<\beta<1$ is substantially more involved than the situation in \cite{GM3}. However,
the result of \cite{GM4} does not cover the Fock space approximation which was addressed in \cite{GM3}, at least for small time. Let us also
summarize their earlier findings.
\begin{theorem}[Grillakis $\&$ Machedon '17]\label{GM-main2}
 Assume $v$ as in Theorem \ref{GM-main1} and $0<\beta<\frac{2}{3}$. Choose $\alpha>\frac{1}{2}$ such that $2\alpha\beta<1$. Then \eqref{TDHFB}
 is uniform in $N$ locally well-posed in $\mathcal{X}^\alpha$. Moreover, suppose $(\phi_t, \Gamma_t, \Lambda_t)$ is
 a solution to the HFB system with some smooth initial conditions $(\phi_0, \Gamma_0, \Lambda_0)$ satisfying 
 the following regularity condition uniformly in $N$: for  $0\leq i \leq 1, 0\leq j\leq 2$
\begin{equation}\label{initial-data}
 \begin{aligned}
&\ \Lp{\langle \grad_x\rangle^\alpha\bd_t^i \grad_{x}^{j}\phi(t, \cdot)\big|_{t=0}}{L^2(dx)} \lesssim 1\\
&\ \Lp{\langle \grad_x\rangle^\alpha\langle \grad_y\rangle^\alpha\bd_t^i \grad_{x+y}^{j}\Gamma(t, \cdot)\big|_{t=0}}{L^2(dxdy)}\lesssim 1\\
&\  \Lp{\langle \grad_x\rangle^\alpha\langle \grad_y\rangle^\alpha\bd_t^i \grad_{x+y}^{j}\Lambda(t, \cdot)\big|_{t=0}}{L^2(dxdy)} \lesssim 1  \\
&\ \Lp{\grad^j_{x+y}\sh(2k_0)(x, y)}{L^2(dxdy)} \lesssim 1.
\end{aligned}
\end{equation}
Then there exist $T_0>0\ (T_0\sim 1)$ independent of $N$, a phase function $\chi(t)$, depending on $N$, and some constant $C=C(T_0, \alpha, \beta)$ such that we have the Fock space estimate
\begin{align}
    \Lp{e^{it\mathcal{H}}e^{-\sqrt{N}\mathcal{A}(\phi_0)}e^{-\mathcal{B}(k_0)}\Omega-e^{i\chi(t)}e^{-\sqrt{N}(\phi_t)}e^{-\mathcal{B}(k_t)}\Omega}{\mathcal{F}} \leq \frac{C}{N^{\frac{1}{6}}}
\end{align}
for all $0 \leq t \leq T_0$. 
\end{theorem}

\begin{remark}
 A crucial point to note in Theorem \ref{GM-main2} is the dependence on the
 time derivative on the initial data \eqref{initial-data}. As pointed out in Remark 2.6 of \cite{GM3}, the
 assumption of time derivative dependence of the initial data rules out the case of coherent states, i.e. $k(0, x,y) = 0$. 
\end{remark}

\subsection{Main Results}  Since the well-posedness theory of 
\eqref{TDHFB} is stated in Strichartz-type space, let us briefly review some basic notations. The 
spacetime function space endowed with the mixed norm 
\begin{align*}
 \llp{F}_{L^q(dt)L^r(dx)L^s(dy)}:=\left(\int^\infty_{-\infty} dt\ \left(\int_{\rr^d} dx\ \llp{F(t, x,\cdot)}_{L^s(\rr^d)}^r \right)^{\frac{q}{r}} \right)^{\frac{1}{q}}
\end{align*}
where the triplet $(q, r, s)$ satisfies the admissible condition $\frac{2}{q}+\frac{d}{r}+\frac{d}{s}= d$
 for $2\leq q, r, s \leq \infty$ is called a Strichartz space.

 Fix $T>0$ and let $c(t)$ be the characteristic function\footnote{The reader should be warned that 
 $c$ has already been used to denote $\ch(k)$, but we seldom use it. However, it should be clear from the context
 which function we are referring to. Nevertheless, we will make explicit the distinction when ambiguity arises.} on the interval $[0, T]$. 
 For our problem, we consider the function space $\mathcal{X}_T$ for the triplet of functions $X=(\phi, \Gamma, \Lambda)$ 
 equipped with the norm that is a sum of the following norms
 \begin{subequations}
 \begin{align}
  \vect{N}_T(\phi):=&\ \llp{\langle\grad_x\rangle^\alpha c(t)\phi}_{L^\infty(dt)L^2(dx)}+\llp{\langle\grad_x\rangle^\alpha c(t)\phi}_{L^2(dt)L^6(dx)}\\
  \vect{N}_T(\Gamma):=&\ \llp{\langle\grad_x\rangle^\alpha\langle\grad_y\rangle^\alpha c(t)\Gamma}_{L^\infty(dt)L^2(dxdy)} \\
  &\ + \llp{\langle\grad_x\rangle^\alpha\langle\grad_y\rangle^\alpha c(t)\Gamma}_{L^2(dt)L^6(dx)L^2(dy)}\nonumber \\
    &\ + \llp{\langle\grad_x\rangle^\alpha\langle\grad_y\rangle^\alpha c(t)\Gamma}_{L^2(dt)L^6(dy)L^2(dx)}\nonumber \\
    &\ + \sup_{z} \llp{\langle\grad_x\rangle^{\alpha-\frac{1}{2}}|\grad_x|^{\frac{1}{2}}c(t)\Gamma(t, x, x+z)}_{L^2(dtdx)}\nonumber
 \end{align}
and
\begin{align}
 \vect{N}_T(\Lambda):=&\ \llp{\langle\grad_x\rangle^\alpha\langle\grad_y\rangle^\alpha c(t)\Lambda}_{L^\infty(dt)L^2(dxdy)}\\
  &\ + \llp{\langle\grad_x\rangle^\alpha\langle\grad_y\rangle^\alpha c(t)\Lambda}_{L^2(dt)L^6(dx)L^2(dy)} \nonumber \\
    &\ + \llp{\langle\grad_x\rangle^\alpha\langle\grad_y\rangle^\alpha c(t)\Lambda}_{L^2(dt)L^6(dy)L^2(dx)}
    \nonumber \\
    &\ + \sup_{z} \llp{\langle\grad_x\rangle^{\alpha}c(t)\Lambda(t, x, x+z)}_{L^2(dtdx)} \nonumber \\
    &\ + \sup_{z} \llp{|\bd_t|^{\frac{1}{4}}[c(t)\Lambda(t, x, x+w)]}_{L^2(dtdx)}. \nonumber
\end{align}
\end{subequations}
Moreover, let us denote the space of triplets $X_t=(\phi_t, \Gamma_t, \Lambda_t)$ where the above norms are finite
for any $0\leq T<\infty$ by $\mathcal{X}_{\infty, \text{loc}}$. Here, we also use
\begin{align}
 \vect{N}_T(X):= \vect{N}_T(\phi)+\vect{N}_T(\Gamma)+\vect{N}_T(\Lambda).
\end{align}

Let us state the first main result of our paper. 
\begin{theorem}\label{main-theorem-1}
 Assume $v$ is a non-negative 
 Schwartz function satisfying the condition that $|\hat v| \leq  \hat w$ for some Schwartz function $w$. 
  Let $0<\beta<1$. Fix $\alpha>\frac{1}{2}$ so that $2\alpha\beta<1$ and choose $R>0$. Then for any 
  \begin{align*}
    (\phi_0, \Gamma_0, \Lambda_0) \in \{(\phi, \Gamma, \Lambda)
    \in \mathcal{X}^\alpha \mid \llp{(\phi, \Gamma, \Lambda)}_{\mathcal{X}^\alpha}<R\},
  \end{align*}
the corresponding local solution to the TDHFB system \eqref{TDHFB} given by Theorem \ref{GM-main1}
extends globally with $X_t= (\phi_t, \Gamma_t, \Lambda_t) \in C([0, \infty) 
\rightarrow \mathcal{X}^\alpha)\cap \mathcal{X}_{\infty, \text{loc}}$.  Moreover, there exist constants $C>0$, depending only on the initial data, such that the following uniform in $N$ a-priori estimate
\begin{align}
 \vect{N}_T(X) \leq CT,
\end{align}
holds for all $T>0$. 
\end{theorem}
In view of Theorem \ref{main-theorem-1}, we can obtain a global-in-time Fock space estimate for the error of 
our quasifree approximation.
\begin{theorem}\label{main-theorem-2}
 Assume $v$ as in Theorem \ref{main-theorem-1}.  Let $0<\beta<1$. Fix $\alpha>\frac{1}{2}$ so that $2\alpha\beta<1$. Let
 $\phi, k$ be solutions to \eqref{TDHFB} with smooth initial conditions $\phi_0, k_0$ satisfying
 the following regularity conditions: for $0\leq j\le 2$ 
 \begin{equation}
  \begin{aligned}
  &\llp{\langle\grad_x\rangle^\alpha\langle\grad_y\rangle^\alpha \grad_{x+y}^j\Gamma_0}_{L^2(dxdy)}+\llp{\langle\grad_x\rangle^\alpha\langle\grad_y\rangle^\alpha \grad_{x+y}^j\Lambda_0}_{L^2(dxdy)}\\
  &+ \llp{\langle\grad_x\rangle^\alpha \grad_x^j\phi_0}_{L^2(dx)} + \llp{\grad^j_{x+y}\sh(2k_0)(\cdot, \cdot)}_{L^2(dxdy)} \lesssim 1.
  \end{aligned} 
 \end{equation}
Then for some $\varepsilon'>0$, there exists a real phase function $\chi_0(t)$ depending on $N$ and constants $C=C(\beta, \varepsilon), \kappa =\kappa(\varepsilon)>0$ 
  such that
  \begin{align}
  \llp{\psi_\text{exact}(t)-\tilde\psi_\text{approx}(t)}_\mathcal{F} \leq \frac{C\exp(\kappa t)}{N^{\frac{1-\beta}{2}}}
  \end{align}
for all $t>0$. 
\end{theorem}

\begin{remark}
 Theorem \ref{main-theorem-2} should be compared with Theorem 1.1  of \cite{BCS}, which also provides a  global-in-time 
 Fock space for the error between the exact evolution and the effective evolution. 
 In fact, our exponent for $N$ is consistent with their exponent, which we believe to be sharp  for the Fock space estimate. 
 However, the nature of the two results are very different. In their work, the approximation scheme is given by
 \begin{align}\label{their-approx}
  \psi_\text{approx}(t) = e^{i\chi_0(t)}e^{-\sqrt{N}\mathcal{A}(\phi_t)} e^{-\mathcal{B}(k_N(t))} U_{2, N}(t)\Omega
 \end{align}
 where the pair excitation function $k_N(t, x, y)$ is given explicitly and $U_{2, N}(t)$ is the dynamics associated
 the a quadratic generator which depends on $k_N$. Their form of $k_N$ is inspired by the fact that 
 when $\beta=1$ the ground state of \eqref{Linear-Schrodinger} exhibits a short-range correlation structure, which means 
 the ground state cannot be a simple tensor product. However, it is also known
 that for $\beta \in (0, 1)$ the ground state 
 factorizes in the large particle limit (cf. Appendix B in \cite{ESY4}), which is not capture by \eqref{their-approx}. 
 
 In contrast,
 Theorem \ref{main-theorem-2} is much more flexible. It lets us start with an uncorrelated system  
 and study the dynamical formation of correlation structures in the system. Furthermore, another feature of our result is that by
 employing techniques from dispersive PDE theory we were able to get a better growth rate in time for the error term in comparison
 to the standard double exponential growth. 
\end{remark}

\section{Global Estimates for the Time-Dependent HFB Equations}
In this section we prove, for a sufficiently small $\varepsilon>0$, the following estimates
\begin{subequations}
\begin{align}
& \llp{\langle\grad_x\rangle^{\frac{1}{2}+\varepsilon}\langle\grad_y\rangle^{\frac{1}{2}+\varepsilon}
\Lambda(t)}_{L^2(dxdy)} \leq C \label{est1}\\
& \llp{\langle\grad_x\rangle^{\frac{1}{2}+\varepsilon}\langle\grad_y\rangle^{\frac{1}{2}+\varepsilon}
\Gamma(t)}_{L^2(dxdy)} \leq C \label{est2}\\
& \llp{\langle \grad_x\rangle^{\frac{1}{2}+\varepsilon}
\phi(t)}_{L^2(dx)} \leq C \label{est3}
\end{align}
\end{subequations}
hold uniformly in $N$ for any fixed time $t>0$. The proof of estimates (\ref{est1})-(\ref{est3}) relies on the conservation 
laws established in \cite{GM1}. For the reader's convenience, we restate the conservation laws for the
 time-dependent HFB system in the following proposition. Let us recall the total
  particle number and energy, which we denoted by $\mathcal{N}$ and
 $\mathcal{E}$  respectively, can be evaluated explicitly as follows:
 \begin{subequations}
 \begin{align}
 \mathcal{N} = N\cdot\left\{ \int dx\  |\phi(x)|^2+ \frac{1}{N}\int dxdy\ |u(x, y)|^2\right\}
 \end{align}
 and 
\begin{align}
 \mathcal{E}=&\ N\cdot\bigg\{\int dx\ |\grad\phi(x)|^2+ \frac{1}{2N}\int dxdy\ |\grad_{x,y}u(x,y)|^2  \\
&+ \frac{1}{2N}\int dxdydz\ v_N(x-y)|\phi(x)u(y, z)+\phi(y)u(x, z)|^2\nonumber\\
&+ \frac{1}{4}\int dxdy\ v_N(x-y)\left\{ 2|\Lambda(x, y)|^2+|\Gamma(x, y)|^2+\Gamma(x, x)\Gamma(y, y) \right\}\bigg\}.  \nonumber
\end{align}
\end{subequations}
For the sake of compactness of notation, we have suppressed the dependence on the time variable since it only plays 
a passive role in our studies of the equations in this section.
 \begin{prop}[Conservation Quantities]\label{conservation}
Suppose $(\phi(t), \Gamma(t), \Lambda(t))$ is a smooth solution to \eqref{TDHFB}
with $v \in L^1(\rr)\cap C^\infty(\rr)$. 
Then the total-particle number and energy for the system are conserved.
\end{prop}
\begin{remark}
The reader should be aware of the fact that we are assuming that the energy per particle is constant and independent of $N$. 
More precisely, we make the assumption that $\mathcal{N}$ and $\mathcal{E}$ are proportional to $N$ for some fixed $N$. In fact, we have that $\mathcal{N}=N$ and $\mathcal{E}\sim N$ or, equivalently, $N^{-1}\mathcal{E} \sim 1$. 
 \end{remark}

As an immediate corollary of the conservation quantities, we prove estimate (\ref{est2}) and (\ref{est3}). 
\begin{cor}\label{energy1}
Let $\phi(t)$ and $\Gamma(t)$ be smooth solutions to the time-dependent HFB equations. Then, for any $0<\varepsilon \leq \frac{1}{2}$, we have the estimates
\begin{align*}
& \llp{\langle \grad_x\rangle^{\frac{1}{2}+\varepsilon}\langle\grad_y\rangle^{\frac{1}{2}+\varepsilon}
\Gamma(t)}_{L^2(dxdy)} \lesssim 1\\
& \llp{\langle \grad_x\rangle^{\frac{1}{2}+\varepsilon}
\phi(t)}_{L^2(dx)} \lesssim 1 
\end{align*}
which hold uniformly in $N$ and independent of $t$. 
\end{cor}

\begin{proof}
It suffices to prove estimate \eqref{est2} since the prove of \eqref{est3} is similar. By Proposition \ref{conservation} and Cauchy-Schwarz inequality, we obtain the 
estimate
\begin{subequations}
\begin{align}
\Lp{\Gamma(t)}{L^2(dxdy)} \leq&\ \llp{\phi_t}_{L^2(dx)}^2+N^{-1}\llp{\overline{u_t}\circ u_t}_{L^2(dxdy)}\label{1-1}\\
\leq&\ \Lp{\phi_t}{L^2(dx)}^2 + N^{-1}\Lp{u_t}{L^2(dxdy)}^2 = 1\nonumber
\end{align}
uniformly in $t$ and independent of $N$. Likewise, we see that
\begin{align}
\Lp{\grad_x\grad_y \Gamma(t)}{L^2(dxdy)} 
\leq\ \Lp{\grad_x\phi_t}{L^2}^2 + N^{-1}\Lp{\grad_xu_t}{L^2}^2 \lesssim 1 \label{1-2}.
\end{align}
\end{subequations}
Hence interpolating (\ref{1-1}) and (\ref{1-2}) yields the desired result. 
\end{proof}

In the remainder of this section, we shall prove that estimate (\ref{est1}) holds. To this end, let us begin by making the observation that proving estimate (\ref{est1}) is equivalent to establishing the estimate
\begin{align}\label{main-est}
N^{-1}\llp{\langle \grad_x\rangle^{\frac{1}{2}+\varepsilon}\langle \grad_y\rangle^{\frac{1}{2}+\varepsilon}\sh(2k_t)}_{L^2(dxdy)} \lesssim  C
\end{align}
for some sufficiently small $\varepsilon>0$. Furthermore, to aid us in proving estimate (\ref{main-est}), we apply the operator identity
\begin{align}\label{hyptrig-id}
\sh(2k) = 2u\circ c = 2u+2u\circ p
\end{align}
and the triangle inequality to obtain a preliminary estimate
\begin{align*}
&\llp{\langle \grad_x\rangle^{\frac{1}{2}+\varepsilon}\langle \grad_y\rangle^{\frac{1}{2}+\varepsilon}\sh(2k_t)}_{L^2(dxdy)} \\
&\lesssim  \llp{\langle \grad_x\rangle^{\frac{1}{2}+\varepsilon}\langle \grad_y\rangle^{\frac{1}{2}+\varepsilon}u_t}_{L^2(dxdy)}
+\llp{\langle \grad_x\rangle^{\frac{1}{2}+\varepsilon}\langle \grad_y\rangle^{\frac{1}{2}+\varepsilon}u_t\circ p_t}_{L^2(dxdy)}\\
&=: I_1(t)+I_2(t).
\end{align*}
Hence it remains to show that $N^{-1}I_i(t) \lesssim C$ for $i=1, 2$. 

To estimate $I_2(t)$, we use the following lemma
\begin{lemma}\label{est-I2}
We have the following estimates
\begin{subequations}
\begin{align}
&N^{-1}\Lp{ |\grad_x|^{\frac{1}{2}}|\grad_y|^{\frac{1}{2}} u_t\circ p_t}{L^2(dxdy)} \lesssim 1,\label{1.3.1}\\
& N^{-1}\Lp{\grad_xu_t\circ \grad_yp_t}{L^2(dxdy)} \lesssim 1 \label{1.3.2}
\end{align}
\end{subequations}
where both are independent of time $t$. In particular, by interpolating estimates (\ref{1.3.1}) and (\ref{1.3.2}), we obtain the estimate
\begin{align}
N^{-1}\Lp{ |\grad_x|^{\frac{1}{2}+\varepsilon}|\grad_y|^{\frac{1}{2}+\varepsilon} u_t\circ p_t}{L^2(dxdy)} \lesssim 1
\end{align}
for any $0 \leq \varepsilon\leq \frac{1}{2}$. 
\end{lemma}

\begin{proof}
Using Plancherel identity and Cauchy-Schwarz inequality, we establish the estimate
\begin{align}
& \Lp{ |\grad_x|^{\frac{1}{2}}|\grad_y|^{\frac{1}{2}} u_t\circ p_t}{L^2(dxdy)} \label{1.3.3}\\
&\lesssim\  \llp{\grad_xu_t}_{L^2(dxdy)}\Lp{p_t}{L^2(dxdy)} + \Lp{u_t}{L^2(dxdy)}\Lp{\grad_y p_t}{L^2(dxdy)}. \nonumber
\end{align}
Next, taking derivatives of the kernel of the operator identity
\begin{align}\label{hyptrig-id2}
\overline{u}\circ u=p\circ p +2p 
\end{align}
yields the operator identity
\begin{align*}
\overline{\grad_xu}\circ\grad_yu=\grad_xp\circ \grad_y p +2\grad_x\grad_y p .
\end{align*}
In particular,  we have that
\begin{align*}
\llp{\grad_x u}^2_{L^2(dxdy)}=\Lp{\grad_xp\circ \grad_y p +2\grad_x\grad_y p }{\text{tr}} \geq \Lp{\grad_xp}{L^2(dxdy)}^2
\end{align*}
since both $\grad_x\grad_y(p\circ p+2p)$ and $2\grad_x\grad_y p$ are positive trace class operators. Hence combining estimate (\ref{1.3.3}) with
the conservation laws, we obtain the estimate
\begin{align*}
&N^{-1}\Lp{|\grad_x|^{\frac{1}{2}}|\grad_y|^{\frac{1}{2}} u_t \circ p_t}{L^2(dxdy)}\lesssim N^{-1}\Lp{\grad_xu_t}{L^2(dxdy)}\Lp{u_t}{L^2(dxdy)} \lesssim 1.
\end{align*}
Likewise, we have shown that
\begin{align*}
N^{-1}\Lp{ \grad_xu_t\circ \grad_yp_t}{L^2(dxdy)} \lesssim N^{-1}\llp{\grad_xu_t}_{L^2}^2 \lesssim 1.
\end{align*}
\end{proof}

Next, to estimate $I_1(t)$, we state a useful global estimate for the $\Lambda$, which is established in our recent paper \cite{CGMZ}. See Theorem 1.1 and Theorem 1.2  in \cite{CGMZ}.
\vspace{3mm}

Recall the equation for $\Lambda(t)$ is given by
\begin{align}
\left(\vect{S}+V\right)\Lambda =& -(v_N\Lambda)\circ \Gamma -\bar\Gamma\circ (v_N\Lambda)-(v_N\bar\Gamma)\circ\Lambda -\Lambda\circ(v_N\Gamma) \label{Lamb-eq}\\
&+ 2(v_N\ast |\phi|^2)(x)\phi(x)\phi(y) + 2(v_N\ast|\phi|^2)(y)\phi(y)\phi(x) =: F \nonumber
\end{align}
 where $v_N\Lambda = v_N(x-y)\Lambda(x, y)$ and 
 \begin{align*}
 V:=\frac{1}{N}v_N+(v_N\ast \diag\Gamma)(x)+(v_N\ast\diag\Gamma)(y).
 \end{align*}
 For the convenience of the reader, let us also recall the theorem. 
\begin{lemma}[Theorem 1.2 of \cite{CGMZ}]\label{main-lem}
Let $\Lambda(t)$ be a solution to the time-dependent HFB equations. Then we have the following time-independent global estimate
\begin{align}
\llp{\grad_x\grad_y \Lambda(t)}_{L^2(dxdy)} \lesssim \llp{\grad_x\grad_y\Lambda_0}_{L^2(dxdy)} + N^{c_0},
\end{align}
where $c_0$ is a positive constant.
\end{lemma}
\begin{remark}
The positive power $c_0$ is chosen to be $\frac{11}{2}$ in \cite{CGMZ}, which is not optimal. One may apply other methods to improve (lower) this power. But we will not explore this point in this paper since a positive power is sufficient for our problem. See Remark \ref{beta}. 
\end{remark}
\begin{remark}
The proof of Lemma \ref{main-lem} tightly depends on a coordinate-flexible Strichartz estimate and a bootstrap argument. (See Theorem 2.3, 2.4 and 2.5 in \cite{CGMZ}.) Other ingredients include an interaction Morawetz argument and a delicate analysis of the nonlinearity. We leave it to interested readers.
\end{remark}
 Moreover, by interpolation, we can obtain 
\begin{lemma}
There exists $\varepsilon_0>0$ such that 
\begin{align}
N^{-1}\llp{|\grad_x|^{\frac{1}{2}+\varepsilon}|\grad_y|^{\frac{1}{2}+\varepsilon} u}_{L^2(dxdy)} \lesssim C
\end{align}
for $0<\varepsilon\leq \varepsilon_0$. 
\end{lemma}
\begin{proof}
Applying Lemma \ref{main-lem} and \eqref{hyptrig-id} give us the estimate
\begin{align*}
&N^{-1}\llp{\grad_x\grad_yu(t)}_{L^2(dxdy)} \\
&\lesssim\ N^{-1}\llp{\grad_x\grad_y\sh(2k)}_{L^2(dxdy)}+N^{-1}\llp{\grad_x u\circ
\grad_y p}_{L^2(dxdy)}\\
&\lesssim\ \llp{\grad_x\phi}^2_{L^2(dx)}+\llp{\grad_x\grad_y\Lambda}_{L^2(dxdy)} + N^{-1}\llp{\grad_x u\circ
\grad_y p}_{L^2(dxdy)}\\
&\lesssim\ 1+ N^{c_0}.
\end{align*}
By interpolating the above inequality with the energy estimate
\begin{align}
N^{-1}\llp{|\grad_x|^{\frac{1}{2}}|\grad_y|^{\frac{1}{2}}u}_{L^2(dxdy)} \lesssim N^{-1}\llp{\grad_x u}_{L^2(dxdy)} 
\lesssim \frac{1}{\sqrt{N}},
\end{align}
we have shown that there exists $\varepsilon_0>0$ such that
\begin{align*}
N^{-1}\llp{|\grad_x|^{\frac{1}{2}+\varepsilon_0}|\grad_y|^{\frac{1}{2}+\varepsilon_0}\sh(k_t)}_{L^2(dxdy)} \lesssim C
\end{align*}
where $C$ is a constant independent of $N$ and $t$. In fact, we see that for all $0<\varepsilon\leq \varepsilon_0 = \frac{1}{2(2c_0+1)}$ we have the estimate
\begin{align}
N^{-1}\llp{|\grad_x|^{\frac{1}{2}+\varepsilon}|\grad_y|^{\frac{1}{2}+\varepsilon}u(t)}_{L^2(dxdy)} \lesssim N^{-\frac{1}{2}+(2c_0+1)\varepsilon}.
\end{align}

\end{proof}

\begin{remark}\label{beta}
Since we need $2\beta(\frac{1}{2}+\varepsilon)<1$, we need the assumption that $0<\varepsilon\leq \min(\frac{1-\beta}{2\beta}, \frac{1}{2(2c_0+1)})$ in our proof of the global well-posedness of the TDHFB system. Noticing that we can choose $\epsilon>0$ sufficiently small, a positive power $c_0$ is enough.
\end{remark}

\section{Global Well-posedness of the TDHFB System}
Let us introduce some notation that we will use in the proof of the uniform in $N$ global well-posedness of 
\eqref{TDHFB} for $0<\beta<1$.

\subsection{Dispersive Estimates} For the sake of completeness, we recall some of the standard a-priori estimates for the $\Gamma$ and $\Lambda$
 equations  proven in \cite{GM3, GM4}. However, we will be more keen on keeping track of the parameters used in \cite{GM3, GM4} to make some of their proofs
  more transparent, which will be necessary when we prove our global well-posedness statement.  
  
 Recall the definition of the $X^{s, b}$ space (or sometimes known as the Bourgain space).
 In this paper, we take $s=0$ and write $X^b = X^{0, b}$. 
   Let us denote $\vect{S}_\pm$ and $\vect{S}$ as in \eqref{schrodinger-op}. 
   Then, we see that the spacetime symbol associated with
 $\vect{S}_\pm$ is $\tau+|\xi|^2-|\eta|^2$ and the spacetime symbol associated 
 with $\vect{S}$ is either $\tau+|\xi|^2+|\eta|^2$ or $\tau+|\xi|^2$, depending on the context.
 We define the norms
 \begin{align*}
 \Lp{f}{X^b_\pm} =&\ \Lp{\langle \tau+|\xi|^2-|\eta|^2\rangle^b \widetilde f (\tau, \xi, \eta)}{L^2(d\tau d\xi d\eta)},\\
 \Lp{f}{X^b_S} =&\ \Lp{\langle \tau+|\xi|^2+|\eta|^2\rangle^b \widetilde f (\tau, \xi, \eta)}{L^2(d\tau d\xi d\eta)}
 \end{align*}
 when $d=6$ or that
   \begin{align*}
 \Lp{f}{X^b_S} = \Lp{\langle \tau+|\xi|^2\rangle^b \widetilde f (\tau, \xi)}{L^2(d\tau d\xi)}
 \end{align*}
when $d=3$. Here, $\widetilde f$ denotes the spacetime Fourier transform defined by
\begin{align}
 \widetilde f(\tau, \xi, \eta) = \int dtdxdy\ e^{-i(\tau t + \xi\cdot x+\eta\cdot y)} f(t, x, y).
\end{align}
See Chapter 2 of \cite{Tao} for an exposition of the topics. 

Let us prove the following time localization lemma with characteristic cutoff function which is a
generalization of the smooth time-localization of Lemma 2.11 in \cite{Tao}. 

\begin{lemma}\label{stable-time-cutoff}
 Let $c(t)$ be the characteristic function of $[0, 1]$. Suppose $0<T<1$ and $-\frac{1}{2}<b'\leq b<\frac{1}{2}$, then we have
 \begin{subequations}
 \begin{align}\label{time-cutoff}
  \llp{c(t/T)u}_{X^{b'}} \lesssim_{c, b, b'} T^{b-b'}\llp{u}_{X^{b}}.
 \end{align}
 Likewise, we have the estimate
 \begin{align}
    \llp{c(t/T)u}_{X^{b'}_\pm} \lesssim_{c, b, b'} T^{b-b'}\llp{u}_{X^{b}_\pm}.
 \end{align}
 \end{subequations}
\end{lemma}
\begin{proof}
Let us define the convolution operator $I(f)(t):= (\widehat{c(\cdot/T)}\ast f)(\tau)$, then it is clear that $I$ 
is a  $L^2$-$L^2$ bounded operator since
\begin{align*}
 \llp{I(f)}_{L^2(d\tau)} = \llp{c(\cdot/T)f}_{L^2(dt)} \leq \llp{f}_{L^2(dt)} =\llp{f}_{L^2(d\tau)}.
\end{align*}
Moreover, we see that $|\bd^k \widehat{c(\cdot/T)}(\tau)| \lesssim |\tau|^{-1-k}$ for $k= 0, 1$. Then, by Corollary of
Theorem 2 in section V.4.2 in Stein, we see have the estimate
\begin{align}
  \llp{I(f)}_{L^2(\omega(\tau)d\tau)} \lesssim \llp{f}_{L^2(\omega(\tau)d\tau)}
\end{align}
where $\omega \in A_2$; see chapter V of Stein for definition of $A_p$ spaces. In particular, it follows that
\begin{align}\label{cutoff1}
   \llp{c(t/T)u}_{X^{b}} \lesssim_{c, b} \llp{u}_{X^{b}}
\end{align}
which is mainly what we need from this lemma.

 Now to prove \eqref{time-cutoff}, we decompose $u$ into two parts
 \begin{align*}
u &=\ P_{\langle \tau+|\xi|^2+|\eta|^2\rangle \geq \frac{1}{T}}u + P_{\langle \tau+|\xi|^2+|\eta|^2\rangle < \frac{1}{T}}u\\
  &=:\ P_1(c(t/T) u) +P_2(c(t/T)u).
 \end{align*}
Then it follows
\begin{align*}
 \llp{c(t /T)P_1 u}_{X^{b'}}^2 =&\ 
 \int \llp{\widehat{c(\cdot/T)}\ast \widetilde{P_1 u}}_{L^2(\langle \tau+|\xi|^2+|\eta|^2\rangle^{2b'} d\tau)}^2 d\xi d\eta\\
 \lesssim&\ \int \llp{\widetilde{P_1 u}}_{L^2(\langle \tau+|\xi|^2+|\eta|^2\rangle^{2b'} d\tau)}^2 d\xi d\eta\\
 \lesssim&\ T^{2(b-b')}\int \llp{\tilde u}_{L^2(\langle \tau+|\xi|^2+|\eta|^2\rangle^{2b} d\tau)}^2 d\xi d\eta
\end{align*}
where the second inequality follows from the remark of Stein chapter V.6.4. which states that
$|\tau|^{2b}\in A_2$ provided $|b|<\frac{1}{2}$.

Finally, we see that
\begin{align*}
  \llp{\widehat{P_2u}(t)}_{L^2(d\xi d\eta)} \lesssim&\ 
  \llp{\int_{\langle \tau +|\xi|^2+|\eta|^2\rangle\leq 1/T} d\tau\ |\tilde{u}(\tau, \xi, \eta)|}_{L^2(d\xi d\eta)}\\
  \lesssim&\ T^{b-\frac{1}{2}}\llp{\left(\int d\tau\ \langle \tau +|\xi|^2+|\eta|^2\rangle^{2b}|
  \widetilde{P_2u}(\tau, \xi, \eta)|^2\right)^{\frac{1}{2}}}_{L^2(d\xi d\eta)}\\
=&\ T^{b-\frac{1}{2}}\llp{P_2u}_{X^b}
  \end{align*}
which means
\begin{align}\label{cutoff2}
 \llp{c(t/T)P_2u}_{X^0}^2 \leq&\  \int^T_0 dt\ \llp{P_2u(t)}_{L^2(dx dy)}^2 
 \lesssim T^{2b}\llp{P_2u}_{X^b}^2.
\end{align}
Interpolating \eqref{cutoff2} with \eqref{cutoff1} where $u$ is replaced with $P_2u$ yields the desired result. 
\end{proof}

Here are some standard dispersive estimates which we will need later in the section.
The proofs of the following two lemmas can be found in section 4 of \cite{GM4}.
\begin{lemma}\label{lemma 4.1}
 Let $\delta>\frac{1}{2}$ and $q, r \geq 2$ satisfy the admissible condition. Then we have the estimates 
 \begin{subequations}
 \begin{align}
  \llp{F}_{L^q(dt)L^r(dx)L^2(dy)} \lesssim_\delta& \llp{F}_{X^{\delta}}, \text{ and its dual estimate}\\
  \llp{F}_{X^{-\delta}}\lesssim_\delta& \llp{F}_{L^{q'}(dt)L^{r'}(dx)L^2(dy)}.
 \end{align}
 \end{subequations}
\end{lemma}

\begin{lemma}\label{lemma 4.2}
 Let $0<\delta<\frac{1}{2}$ and $q, r>2$ such that $\frac{2}{q}+\frac{3}{r}=\frac{5-4\delta}{2}$. Then 
 we have
 \begin{subequations}
 \begin{align}
  \llp{F}_{L^q(dt)L^r(dx)L^2(dy)} \lesssim_\delta& \llp{F}_{X^{\delta}}, \ \text{ and its dual estimate}\\
  \llp{F}_{X^{-\delta}} \lesssim_\delta& \llp{F}_{L^{q'}(dt)L^{r'}(dx)L^2(dy)}.
 \end{align}
\end{subequations}
\end{lemma}

\subsection{Global Well-posedness} Fix $\varepsilon>0$ small enough and consider  $c(t)=c_{I}(t)$ the characteristic function
on the interval $I=[T_0, T_1]$. Then we define  
\begin{subequations}
\begin{align}
\vect{N}_{I}(\Lambda) :=
 & \sup_z\llp{\langle \grad_x\rangle^{\frac{1}{2}+\varepsilon}c(t)\Lambda(t, x, x+z)}_{L^2(dt)L^2(dx)}\label{norm1}\\
& + \sup_z\llp{|\bd_t|^{\frac{1}{4}}[c(t)\Lambda(t, x, x+z)]}_{L^2(dt)L^2(dx)} \nonumber\\
& + \llp{\langle\grad_x\rangle^{\frac{1}{2}+\varepsilon}\langle \grad_y\rangle^{\frac{1}{2}+\varepsilon} c(t)\Lambda}_{L^\infty(dt)L^2(dxdy)}\nonumber\\
& + \llp{\langle\grad_x\rangle^{\frac{1}{2}+\varepsilon}\langle \grad_y\rangle^{\frac{1}{2}+\varepsilon} c(t)\Lambda}_{L^2(dt)L^6(dx)L^2(dy)}\nonumber\\
&+\llp{\langle\grad_x\rangle^{\frac{1}{2}+\varepsilon}\langle \grad_y\rangle^{\frac{1}{2}+\varepsilon} c(t)\Lambda}_{L^2(dt)L^6(dy)L^2(dx)}\nonumber\\
\vect{N}_{I}(\Gamma) :=&\ \sup_z\llp{|\grad_x|^{\frac{1}{2}}\langle \grad_x\rangle^\varepsilon c(t)\Gamma(t, x, x+z)}_{L^2(dt)L^2(dx)}\label{norm2}\\
&+ \llp{\langle\grad_x\rangle^{\frac{1}{2}+\varepsilon} \langle\grad_y\rangle^{\frac{1}{2}+\varepsilon} c(t)\Gamma}_{L^\infty(dt)L^2(dxdy)} \nonumber\\
& + \llp{\langle\grad_x\rangle^{\frac{1}{2}+\varepsilon}\langle \grad_y\rangle^{\frac{1}{2}+\varepsilon} c(t)\Gamma}_{L^2(dt)L^6(dx)L^2(dy)}\nonumber\\
&+\llp{\langle\grad_x\rangle^{\frac{1}{2}+\varepsilon}\langle \grad_y\rangle^{\frac{1}{2}+\varepsilon} c(t)\Gamma}_{L^2(dt)L^6(dy)L^2(dx)}\nonumber\\
\vect{N}_{I}(\phi) := &\ \llp{\langle\grad_x\rangle^{\frac{1}{2}+\varepsilon}c(t)\phi}_{L^\infty(dt)L^2(dx)}+\llp{\langle \grad_x\rangle^{\frac{1}{2}+\varepsilon}c(t)\phi}_{L^2(dt) L^6(dx)}.\label{norm3}
\end{align}
\end{subequations}
Like before, we denote the sum of the three norms by
\begin{align}
\vect{N}_{I}(X):= \vect{N}_{I}(\phi) +\vect{N}_{I}(\Gamma)+\vect{N}_{I}(\Lambda).
\end{align}
If $I=[T_0, T_1] = [0, T]$ then we recover $\vect{N}_{T}(X) = \vect{N}_{I}(X)$. Moreover, we adopt the notation
\begin{align}
\vect{N}_{I}(DX):= \vect{N}_{I}(D\phi)+\vect{N}_{I}(D\Gamma)+\vect{N}_{I}(D\Lambda)
\end{align}
for any differential operator $D$. 

The goal of this section is to prove the global well-posedness of solutions for the time-dependent HFB equations. However, it suffices to 
prove an a-priori estimate of the form 
\begin{align}\label{apriori-est}
\vect{N}_T(X) \lesssim F(T)
\end{align}
for some positive real-valued function $F$ defined on all of $[0, \infty)$.

 Let us now recall some propositions from \cite{GM4}, which are modified by Lemma \ref{stable-time-cutoff}. For the $\phi$ and $\Gamma$ equations, we have the following
 proposition.

\begin{prop}[Proposition 3.7 in \cite{GM4}]\label{GM-prop1}
Let $c(t)$ be the characteristic function on $I=[T_0, T_1]$ such that $T_1-T_0<1$. Suppose $\phi(t)$ and $\Gamma(t)$ are solutions to 
\begin{align*}
 \vect{S}_\pm \Gamma =&\ F, \ \ \Gamma(T_0) = \Gamma_0\\
 \vect{S} \phi =&\ f, \ \ \phi(T_0)= \phi_0.
\end{align*}
Then for any $\delta>0$ we have the estimates
\begin{subequations}
\begin{align}
 \vect{N}_{I}(\Gamma) \lesssim_\delta& \llp{\langle \grad_x\rangle^{\frac{1}{2}+\varepsilon}
 \langle \grad_y\rangle^{\frac{1}{2}+\varepsilon}\Gamma_0}_{L^2(dxdy)}\\
 & +\llp{\langle \grad_x\rangle^{\frac{1}{2}+\varepsilon}
 \langle \grad_y\rangle^{\frac{1}{2}+\varepsilon}c(t)F}_{X_\pm^{-\frac{1}{2}+\delta}} \nonumber\\
 \vect{N}_{I}(\phi)\lesssim_\delta& \llp{\langle \grad_x\rangle^{\frac{1}{2}+\varepsilon}
 \phi_0}_{L^2(dxdy)}+\llp{\langle \grad_x\rangle^{\frac{1}{2}+\varepsilon}c(t)f}_{X^{-\frac{1}{2}+\delta}}.
\end{align}
\end{subequations}
\end{prop}

For the $\Lambda$ equations, with the potential $\frac{1}{N}v_N(x-y)\Lambda$, we have the following proposition, which is also
the main result of \cite{GM4}.
\begin{prop}[Theorem 3.8 in \cite{GM4}]\label{GM-prop2}
Assume $c(t)$ as in Proposition \ref{GM-prop1}. Let $0<\beta<1$. Suppose $\Lambda(t)$ is a solution to 
 \begin{align}
  \left(\vect{S}+\frac{1}{N}v_N(x-y)\right)\Lambda = F, \ \ \ \Lambda(0) = \Lambda_0.
 \end{align}
 Then for all $\delta>0$ sufficiently small, the following estimate
 \begin{align}
  &\vect{N}_{I}(\Lambda)\\
  &\lesssim_\delta \llp{\langle \grad_x\rangle^{\frac{1}{2}+\varepsilon}
 \langle \grad_y\rangle^{\frac{1}{2}+\varepsilon}\Lambda_0}_{L^2(dxdy)} +\llp{\langle \grad_x\rangle^{\frac{1}{2}+\varepsilon}
 \langle \grad_y\rangle^{\frac{1}{2}+\varepsilon}c(t)F}_{X^{-\frac{1}{2}+\delta}}\nonumber \\
 &\quad +\min\left\{  \llp{\langle \grad_x\rangle^{\frac{1}{2}+\varepsilon}
 \langle \grad_y\rangle^{\varepsilon}c(t)F}_{X^{-\frac{1}{4}-\delta}}, \llp{\langle \grad_x\rangle^{\varepsilon}
 \langle \grad_y\rangle^{\frac{1}{2}+\varepsilon}c(t)F}_{X^{-\frac{1}{4}-\delta}}\right\}\nonumber
 \end{align}
holds uniformly in $N$.  

\end{prop}

We begin by proving couple lemmas to aid us in establishing \eqref{apriori-est}.
\begin{lemma}[sharp trace-type lemma]\label{sharp-trace}
Let $f \in \mathcal{S}(\rr^3\times \rr^3)$. Then the following estimate holds
\begin{align}
\llp{|\grad_x|^{\frac{1}{2}}\diag f}_{L^2(dx)}\lesssim \Lp{\grad_x\grad_y f}{L^2(dxdy)}.
\end{align}
\end{lemma}

\begin{proof}
Note that
\begin{align*}
\widehat{\diag f}(\xi) = \frac{1}{(2\pi)^{3}} \int d\eta\ \widehat f(\xi-\eta, \eta).
\end{align*}
 Then, by Plancherel and Cauchy Schwarz, we have the estimate
\begin{align*}
\llp{|\grad_x|^{\frac{1}{2}}\diag f}_{L^2(dx)}^2  \sim&\ \int d\xi\ |\xi|\left|\int d\eta\ \widehat f (\xi-\eta, \eta) \right|^2\\
\lesssim&\ \sup_{\xi} A(\xi) \int d\xi d\eta\ |\xi|^2|\eta|^2|\widehat f (\xi, \eta)|^2
\end{align*}
where
\begin{align*}
A(\xi) = \int d\eta\ \frac{|\xi|}{|\xi-\eta|^2|\eta|^2}.
\end{align*}
Hence, to complete the proof, it suffices to show $\sup_\xi A(\xi)<\infty$. 

Let us evaluate the integral $A(\xi)$ into each of the following three regions: $R_1=\{\frac{1}{2}|\xi|\leq |\eta| \leq 2|\xi|\}$, $R_2=\{\frac{1}{2}|\xi|\leq |\eta-\xi| \leq 2|\xi|\}$, and $R_3=\{|\eta|>2|\xi| \text{ or } |\eta-\xi|>2|\xi|\}$. Then, we see that
\begin{align*}
\int_{R_1} d\eta\ \frac{|\xi|}{|\xi-\eta|^2|\eta|^2} \lesssim \frac{1}{|\xi|} \int_{|\xi-\eta|\leq 3|\xi|}  \frac{d\eta}{|\xi-\eta|^2} \lesssim 1
\end{align*}
and likewise for $R_2$. For the region $R_3$, we see that $|\eta|>|\xi|$ and $2|\xi-\eta| \geq |\xi-\eta|+|\xi| \geq |\eta|$, which means
\begin{align*}
\int_{R_3} d\eta\ \frac{|\xi|}{|\xi-\eta|^2|\eta|^2} \lesssim \int_{|\eta|>|\xi|}d\eta\  \frac{|\xi|}{|\eta|^4} \lesssim 1.
\end{align*}
Thus, we have that $\sup_\xi A(\xi)<\infty$. 
\end{proof}

\begin{lemma}\label{no-deriv}
Let $(\phi(t), \Gamma(t), \Lambda(t))$ be a solution to \eqref{TDHFB} and $I=[T_0, T_1]$. Then there exists $\delta>0$ such that
we have the following estimates
\begin{align}\label{close-est}
\vect{N}_{I}(X) \lesssim (\text{data}) +|I|^\delta   \vect{N}_{I}(X).
\end{align}
\end{lemma}

\begin{proof}
It suffices to consider the proof of estimate \eqref{close-est} for $\Gamma$ and $\Lambda$ since the proof for $\phi$ is similar. Recall the equation for $\Gamma$ is given by
\begin{equation}
\begin{aligned}
\vect{S}_\pm \bar\Gamma =& -(v_N\Lambda)\circ\bar\Lambda +\Lambda \circ (v_N\bar \Lambda)-(v_N\bar\Gamma)\circ\bar \Gamma+\bar\Gamma\circ(v_N\bar \Gamma)\\
&- (v_N\ast \diag\Gamma)\cdot \bar\Gamma+\bar \Gamma \cdot (v_N\ast \diag\Gamma)\\
&+ 2 (v_N\ast |\phi|^2)\cdot \sv{\phi}\csv{\phi}-2\sv{\phi}\csv{\phi} \cdot (v_N\ast |\phi|^2)
=: F.
\end{aligned}
\end{equation}
Then, by Proposition \ref{GM-prop1}  and Lemma \ref{lemma 4.2}, we have that for any $0<\delta<\frac{\varepsilon}{2}$ 
the following estimate holds
 \begin{align*}
 \vect{N}_{I}(\Gamma) \lesssim&\  \llp{\langle \grad_x\rangle^{\frac{1}{2}+\varepsilon}\langle\grad_y\rangle^{\frac{1}{2}+\varepsilon}\Gamma(T_0, \cdot)}_{L^2(dxdy)}\\
 &+ \Lp{\langle \grad_x\rangle^{\frac{1}{2}+\varepsilon}\langle\grad_y\rangle^{\frac{1}{2}+\varepsilon}c(t)F}{X^{-\frac{1}{2}+\frac{\varepsilon-2\delta}{2}}}\\ 
 \lesssim&\ \llp{\langle \grad_x\rangle^{\frac{1}{2}+\varepsilon}\langle\grad_y\rangle^{\frac{1}{2}+\varepsilon}\Gamma(T_0, \cdot)}_{L^2(dxdy)}\\
 &+|I|^{\frac{\varepsilon-2\delta}{2}} \llp{\langle \grad_x\rangle^{\frac{1}{2}+\varepsilon}\langle\grad_y\rangle^{\frac{1}{2}+\varepsilon}c(t)F}_{L^{2}(dt)L^{\frac{6}{5}+}(dx) L^2(dy)}.
 \end{align*}
Here, the symbol $\frac{6}{5}+$ denotes a fixed number slightly bigger than $\frac{6}{5}$ with dependence on $\varepsilon$, in fact, $\frac{6}{5}+=\frac{6}{5-2\varepsilon}$. 

For the forcing term $\Gamma\cdot (v_N\ast\diag\Gamma)$, we apply Young's convolution inequality, Sobolev inequality, and Corollary \ref{energy1}
to obtain the estimate
\begin{align*}
&\llp{\langle\grad_x\rangle^{\frac{1}{2}+\varepsilon} \langle\grad_y\rangle^{\frac{1}{2}+\varepsilon}[\bar \Gamma \cdot (v_N\ast \diag\Gamma)]}_{L^{2}(I)L^{\frac{6}{5}+}(dx) L^2(dy)}\\
 &\lesssim\ \llp{\langle\grad_y\rangle^{\frac{1}{2}+\varepsilon}\bar \Gamma }_{L^{\infty}(I)L^{3+}(dx) L^2(dy)}\llp{\langle\grad_x\rangle^{\frac{1}{2}+\varepsilon} \diag\Gamma}_{L^{2}(I)L^{2}(dx)}\\
 &\quad +  \llp{\langle \grad_x\rangle^{\frac{1}{2}+\varepsilon}\langle\grad_y\rangle^{\frac{1}{2}+\varepsilon}\bar \Gamma }_{L^{\infty}(I)L^{2}(dxdy)}\llp{\diag\Gamma}_{L^{2}(I)L^{3+}(dx)}\\
&\lesssim\ \vect{N}_{I}(\Gamma),
\end{align*}
where $3+ = \frac{3}{1-\varepsilon}$. Next, for the forcing term $\Lambda\circ (v_N\bar\Lambda)$, we apply Kato-Ponce, Sobolev, and estimate \eqref{est3} from the 
previous section to obtain the estimate
\begin{align*}
 &\llp{\langle\grad_x\rangle^{\frac{1}{2}+\varepsilon}\langle\grad_y\rangle^{\frac{1}{2}+\varepsilon} [\Lambda \circ (v_N\bar \Lambda)]  }_{L^{2}(I)L^{\frac{6}{5}+}(dx) L^2(dy)}\\
 &\lesssim \int dz\ v_N(z) \llp{\langle\grad_x\rangle^{\frac{1}{2}+\varepsilon}\bar\Lambda(x, x-z)\langle\grad_y\rangle^{\frac{1}{2}+\varepsilon}\Lambda(x-z, y)}_{L^2(I)L^{\frac{6}{5}+}(dx)L^2(dy)}\\
  &\quad + \int dz\ v_N(z) \llp{\bar\Lambda(x, x-z)\langle\grad_x\rangle^{\frac{1}{2}+\varepsilon}\langle\grad_y\rangle^{\frac{1}{2}+\varepsilon}\Lambda(x-z, y)}_{L^2(I)L^{\frac{6}{5}+}(dx)L^2(dy)}\\
  &\lesssim\  \int dz\ v_N(z) \llp{\langle\grad_x\rangle^{\frac{1}{2}+\varepsilon}\bar\Lambda(x, x-z)}_{L^2(dtdx)}\llp{\langle\grad_y\rangle^{\frac{1}{2}+\varepsilon}\Lambda}_{L^\infty(I)L^{3+}(dx)L^2(dy)}\\
    &\quad +  \int dz\ v_N(z) \llp{\Lambda(x, x-z)}_{L^2(dt)L^{3+}(dx)}\llp{\langle\grad_x\rangle^{\frac{1}{2}+\varepsilon}\langle\grad_y\rangle^{\frac{1}{2}+\varepsilon}\Lambda}_{L^\infty(I)L^2(dxdy)}\\
   &\lesssim \vect{N}_{I}(\Lambda).
\end{align*}
The remaining nonlinear terms $(v_N\Gamma)\circ\Gamma$ and $(v_N\ast |\phi|^2)\cdot \sv{\phi}\csv{\phi}$ can be handled in a similar manner. Thus, we have shown
\begin{align*}
 \vect{N}_{I}(\Gamma) \lesssim C_0(T_0)+ |I|^\frac{\varepsilon-2\delta}{2}\vect{N}_{I}(X).
\end{align*}

Next, let us recall the equation for $\Lambda$ given by
\begin{align}
\left(\vect{S}+\frac{1}{N}v_N\Lambda\right) =
&  -(v_N\ast\diag \Gamma)\cdot \Lambda -\Lambda \cdot (v_N\ast \diag\Gamma)\\
&-(v_N\Lambda)\circ\Gamma-\bar\Gamma\circ(v_N\Lambda)-(v_N\bar\Gamma)\circ\Lambda-\Lambda\circ(v_N\Gamma)\nonumber\\
&+ 2 (v_N\ast |\phi|^2)\cdot \phi\otimes\phi-2\phi\otimes\phi \cdot (v_N\ast |\phi|^2)
=: G\nonumber.
\end{align}
To estimate $\Lambda$, we employ Proposition \ref{GM-prop2} and Lemma \ref{lemma 4.2} to get the estimate
 \begin{align*}
 \vect{N}_{I}(\Lambda)\lesssim&\ \llp{\langle \grad_x\rangle^{\frac{1}{2}+\varepsilon}\langle\grad_y\rangle^{\frac{1}{2}+\varepsilon}\Lambda(T_0, \cdot)}_{L^2(dxdy)}\\
 &+|I|^{\frac{\varepsilon-2\delta}{2}} \llp{\langle \grad_x\rangle^{\frac{1}{2}+\varepsilon}\langle\grad_y\rangle^{\frac{1}{2}+\varepsilon}G}_{L^{2}(I)L^{\frac{6}{5}+}(dx) L^2(dy)}\\
 &+|I|^{\frac{\varepsilon-2\delta}{2}}\min\bigg(\llp{\langle \grad_x\rangle^{\varepsilon}\langle\grad_y\rangle^{\frac{1}{2}+\varepsilon}G}_{L^{\frac{4}{3}}(I) L^2(dxdy)}, \\
 &\qquad \qquad \qquad\ \ \llp{\langle \grad_x\rangle^{\frac{1}{2}+\varepsilon}\langle\grad_y\rangle^{\varepsilon}G}_{L^{\frac{4}{3}}(I) L^2(dxdy)}\bigg).
 \end{align*}
 Hence, to complete the proof, it suffices to estimate $\langle \grad_x\rangle^{\varepsilon}\langle\grad_y\rangle^{\frac{1}{2}+\varepsilon}G$ in $L^{\frac{4}{3}}(I)L^2(dxdy)$ since 
 $ \llp{\langle \grad_x\rangle^{\frac{1}{2}+\varepsilon}\langle\grad_y\rangle^{\frac{1}{2}+\varepsilon}G}_{L^{2}(I)L^{\frac{6}{5}+}(dx) L^2(dy)}$ can be handled in a similar manner like $\Gamma$.  
 
 Let us consider the forcing term $(v_N\ast\diag \Gamma)\cdot \Lambda$. Then, by Kato-Ponce and Sobolev inequalities, we have that
 \begin{align*}
 &\llp{\langle \grad_x\rangle^{\varepsilon}\langle\grad_y\rangle^{\frac{1}{2}+\varepsilon}[(v_N\ast\diag \Gamma)\cdot \Lambda]}_{L^{\frac{4}{3}}(I) L^2(dxdy)}\\
 &\lesssim \llp{\langle \grad_x\rangle^{\varepsilon}\diag \Gamma}_{L^{2}(I) L^{3-}(dx)}\llp{\langle\grad_y\rangle^{\frac{1}{2}+\varepsilon}\Lambda}_{L^{4}(I)
  L^{6+}(dx)L^2(dy)}\\
 &\quad +\llp{\diag \Gamma}_{L^{2}(I) L^3(dx)}\llp{\langle \grad_x\rangle^{\varepsilon}\langle\grad_y\rangle^{\frac{1}{2}+\varepsilon}\Lambda}_{L^{4}(I) 
 L^6(dx)L^2(dy)}\\
  &\lesssim\llp{\langle \grad_x\rangle^{\varepsilon}|\grad_x|^{\frac{1}{2}-\varepsilon}\diag \Gamma}_{L^{2}(I) L^2(dx)}\llp{\langle \grad_x\rangle^{\frac{1}{2}+\varepsilon}\langle\grad_y\rangle^{\frac{1}{2}+\varepsilon}\Lambda}_{L^{4}(I) L^3(dx)L^2(dy)}.
 \end{align*}
 Finally, by Lemma \ref{sharp-trace} and 
Corollary \ref{energy1}, we obtain the estimate
\begin{align*}
 &\llp{\langle \grad_x\rangle^{\varepsilon}\langle\grad_y\rangle^{\frac{1}{2}+\varepsilon}[(v_N\ast\diag \Gamma)\cdot \Lambda]}_{L^{\frac{4}{3}}(I) L^2(dxdy)}\lesssim |I|^{\frac{1}{2}}\vect{N}_{I}(\Lambda).
\end{align*}
Next, let us consider the forcing term $(v_N\Lambda) \circ \Gamma$. By Kato-Ponce, Sobolev inequalities, and conservation of energy, we have that
\begin{align*}
 &\llp{\langle \grad_x\rangle^{\varepsilon}\langle\grad_y\rangle^{\frac{1}{2}+\varepsilon}[(v_N\Lambda)\circ \Gamma]}_{L^{\frac{4}{3}}(I) L^2(dxdy)}\\
 &\lesssim \int dz\ v_N(z)\llp{\langle \grad_x\rangle^{\varepsilon}\Lambda(x, x-z)}_{L^{2}(I) L^{3}(dx)}\llp{\langle\grad_y\rangle^{\frac{1}{2}+\varepsilon}\Gamma}_{L^{4}(I)
  L^{6}(dx)L^2(dy)}\\
  &\quad +\int dz\ v_N(z)\llp{\Lambda(x, x-z)}_{L^{2}(I) L^{3+}(dx)}\llp{\langle \grad_x\rangle^{\varepsilon}
  \langle\grad_y\rangle^{\frac{1}{2}+\varepsilon}\Gamma}_{L^{4}(I)L^{6-}(dx)L^2(dy)}\\
  &\lesssim |I|^{\frac{1}{4}}\vect{N}_{I}(\Lambda).
\end{align*}
The remaining forcing terms can be handled in a similar fashion. Hence it follows
\begin{align*}
&\vect{N}_{I}(\Lambda)\lesssim\ C_0(T_0)+ |I|^\frac{\varepsilon-2\delta}{2}\vect{N}_{I}(X)
\end{align*}
which completes the proof of the lemma.
\end{proof}

Next we aim to obtain a-priori estimates of the form (\ref{apriori-est}).  
\begin{prop}\label{4.3}
Let $T>0$. Assume $(\phi(t), \Gamma(t), \Lambda(t))$ is a solution to the time-dependent HFB system, then we have the following a-priori estimate
\begin{align}\label{apriori}
\vect{N}_T(X) \lesssim (\text{data})+T.
\end{align}
\end{prop}
\begin{proof}
Proposition \ref{no-deriv} implies that for an interval with length similar to $1$, 
\begin{equation}
    \vect{N}_{I}(X)\lesssim_{(\textmd{data})} 1.
\end{equation}
We may split the interval $[0,T]$ into $M$ intervals $I_i$ where $M \sim T$ and the length of each interval $I_i$ is less than $\frac{1}{2}$. Then we sum the norms $ \vect{N}_{I_i}(X)$ up which explains this proposition.
\end{proof}
Now Theorem \ref{main-theorem-1} follows. In particular, we obtained the linear-in-time control for $\vect{N}_T(X)$.
\begin{lemma}\label{low-order-est} 
Let $(\phi(t), \Gamma(t), \Lambda(t))$ be a solution to \eqref{TDHFB} and $I=[T_0, T_1]$. Then there exists $\delta>0$ such that we have the following estimates
\begin{subequations}
\begin{align}\label{grad-est}
 \vect{N}_{I}(\grad_{x+y}X) \lesssim&\ C_1(T_0) + |I|^\delta \vect{N}_{I}(X) \vect{N}_{I}(\grad_{x+y}X) 
\end{align}
where
\begin{align}
C_1(T) :=&\ \llp{(\grad_x\phi(T,\cdot), \grad_{x+y}\Gamma(T, \cdot), \grad_{x+y}\Lambda(T, \cdot))}_{\mathcal{X}^\alpha} .
\end{align}
\end{subequations}
Here, we use the notation $\grad_{x+y}:=\grad_x+\grad_y$. 
\end{lemma}

\begin{proof}
Taking the $\grad_{x+y}$ derivative of (\ref{Lamb-eq}) yields
\begin{align*}
\left( \vect{S}+N^{-1}v_N\right)\grad_{x+y} \Lambda =& -(v_N\Lambda)\circ \grad_{x+y}\Gamma
-(v_N\ast \diag \Gamma)\cdot\grad_{x+y}\Lambda\\
&+\text{similar terms} =: F.
\end{align*}
Note that we used the fact that $\grad_{x+y}$ commutes with $N^{-1}v_N(x-y)$, i.e. $[\grad_{x+y}, N^{-1}v_N(x-y)] = 0$.
By the same argument as in Lemma \ref{no-deriv}, we have the estimate
\begin{align*}
&\vect{N}_{I}(\grad_{x+y}\Lambda)\lesssim C_1(T_0)+|I|^\frac{\varepsilon-2\delta}{2} \llp{\langle\grad_x\rangle^{\frac{1}{2}+\varepsilon}\langle\grad_y\rangle^{\frac{1}{2}+\varepsilon}F}_{L^{2}(I)L^{\frac{6}{5}+}(dx)L^2(dy)}.
\end{align*}
We shall look at two generic cases, as stated above, to deduce \eqref{grad-est}. In the first case, we estimate the term $(v_N\Lambda)\circ \grad_{x+y}\Gamma$, which goes  as follows
\begin{align*}
&\llp{\langle\grad_x\rangle^{\frac{1}{2}+\varepsilon}\langle\grad_y\rangle^{\frac{1}{2}+\varepsilon}(v_N \Lambda)\circ \grad_{x+y}\Gamma}_{L^{2}(I)L^{\frac{6}{5}+}(dx)L^2(dy)} \\
&\lesssim
\int dz\ v_N(z) \llp{\langle\grad_x\rangle^{\frac{1}{2}+\varepsilon} \Lambda(x, x-z) \langle\grad_y\rangle^{\frac{1}{2}+\varepsilon} \grad_{x+y}\Gamma(x-z, y)}_{L^{2}(I)L^{\frac{6}{5}+}(dx)L^2(dy)}\\
&\quad +\int dz\ v_N(z) \llp{\Lambda(x, x-z) \langle\grad_x\rangle^{\frac{1}{2}+\varepsilon}\langle\grad_y\rangle^{\frac{1}{2}+\varepsilon} \grad_{x+y} \Gamma(x-z, y)}_{L^{2}(I)L^{\frac{6}{5}+}(dx)L^2(dy)}\\
&\lesssim \int dz\ v_N(z) \llp{\langle\grad_x\rangle^{\frac{1}{2}+\varepsilon}\Lambda(x, x-z)}_{L^2(I)L^2(dx)}\llp{\langle\grad_x\rangle^{\frac{1}{2}+\varepsilon} \langle\grad_y\rangle^{\frac{1}{2}+\varepsilon}\grad_{x+y}\Gamma}_{L^{\infty}(I)L^2(dxdy)}\\
&\quad +\int dz\ v_N(z) \llp{\Lambda(x, x-z)}_{L^2(I)L^{3+}(dx)}\llp{\langle\grad_x\rangle^{\frac{1}{2}+\varepsilon}\langle\grad_y\rangle^{\frac{1}{2}+\varepsilon}\grad_{x+y}\Gamma}_{L^\infty(I)L^2(dxdy)}\\
&\lesssim \vect{N}_{I}(\Lambda)\vect{N}_{I}(\grad_{x+y}\Gamma).
\end{align*}
In the second case, we estimate the term $(v_N\ast\diag\Gamma)\cdot \grad_{x+y}\Lambda$ as follows
\begin{align*}
&\ \llp{\langle\grad_x\rangle^{\frac{1}{2}+\varepsilon}\langle\grad_y\rangle^{\frac{1}{2}+\varepsilon}[(v_N\ast \diag\Gamma)\cdot\grad_{x+y}\Lambda]}_{L^{2}(I)
L^{\frac{6}{5}+}(dx)L^2(dy)} \\
&\lesssim\ \llp{(v_N\ast \langle\grad_x\rangle^{\frac{1}{2}+\varepsilon}\diag\Gamma)\cdot\langle\grad_y\rangle^{\frac{1}{2}+\varepsilon}\grad_{x+y}\Lambda}_{L^{2}(I)
L^{\frac{6}{5}+}(dx)L^2(dy)} \\
&\quad+\llp{(v_N\ast \langle\grad_x\rangle^{\frac{1}{2}+\varepsilon}\diag\Gamma)\cdot\langle\grad_y\rangle^{\frac{1}{2}+\varepsilon}\grad_{x+y}\Lambda}_{L^{2}(I)
L^{\frac{6}{5}+}(dx)L^2(dy)} \\
&\lesssim\ \vect{N}_{I}(\Gamma)\vect{N}_{I}(\grad_{x+y}\Lambda).
\end{align*}
Hence, combining the above estimates yields
\begin{align*}
\vect{N}_{I}(\grad_{x+y}\Lambda) \lesssim&\  C_1(T_0) + |I|^{\frac{\varepsilon-2\delta}{2}}\{ \vect{N}_{I}(\Lambda)\vect{N}_{I}(\grad_{x+y}\Gamma)\\
&\ +\vect{N}_{I}(\grad_{x+y}\Lambda)\vect{N}_{I}(\Gamma)+\vect{N}_{I}(\grad_{x+y}\phi)\vect{N}_{I}(\phi)\}\\
\lesssim&\ C_1(T_0) +|I|^\frac{\varepsilon-2\delta}{2}\vect{N}_{I}(X)\vect{N}_{I}(\grad_{x+y}X).
\end{align*}
Similarly, we can show
\begin{align*}
\vect{N}_{I}(\grad_{x+y}\Gamma) \lesssim C_1(T_0) +|I|^\frac{\varepsilon-2\delta}{2}\vect{N}_{I}(X)\vect{N}_{I}(\grad_{x+y}X)
\end{align*}
and
\begin{align*}
\vect{ N}_{I}(\grad_{x+y}\phi) \lesssim C_1(T_0) +|I|^\frac{\varepsilon-2\delta}{2}\vect{N}_{I}(X)\vect{N}_{I}(\grad_{x+y}X).
\end{align*}
Therefore, summing up the three inequalities yields (\ref{grad-est}).
\end{proof}

Using the above lemma we could again prove an a-priori growth estimates for the norm of $\grad_{x+y}X$.
\begin{prop}\label{4.6}
Let $T>0$. Suppose ($\phi(t), \Gamma(t)$, $\Lambda(t)$) is a solution to the TDHFB system, then we have the following uniform 
in $N$ a-priori estimate
\begin{subequations}
\begin{align}
\vect{N}_T(\grad_{x+y}X) \lesssim&\ \exp\left(\kappa T \right) \label{grad-apriori}
\end{align}
\end{subequations}
for some $\kappa>0$, which are independent of $T$. 
\end{prop}

\begin{proof}
Proposition \ref{no-deriv} and Proposition \ref{low-order-est} imply that for an interval $I$ with length similar to $1$, 
\begin{equation}
    \vect{N}_{I}(\grad_{x+y}X)\lesssim_{(\textmd{data})} 1.
\end{equation}
Here we note that for this case, the $C_1(T)$ is not conserved, which is different from Proposition \ref{4.3}. Next we split the interval $[0,T]$ into $M$' intervals $I_i=[T_i,T_{i+1}]$ where $M \sim T$ and the length of each interval $I_i$ is less than $\frac{1}{2}$. On each interval,
\begin{equation}
   \vect{N}_{I_i}(\grad_{x+y} X)\lesssim C_1(T_i). 
\end{equation}
Summing them up, we have
\begin{equation}
    \vect{N}_{T}(\grad_{x+y} X) \lesssim \sum_i \vect{N}_{I_i}(\grad_{x+y} X)\lesssim C_1(T_0)+\sum_i C_1(T_i).
\end{equation}
By switching to the continuous $T$-variable , we obtain the estimate 
\begin{align*}
\vect{N}_{T}(\grad_{x+y}X)\lesssim&\  C_1(T_0)+\int^T_0 d\tau\ \vect{N}_\tau(\grad_{x+y}X).
\end{align*}
Finally, applying Gronwall's inequality yields
\begin{align*}
\vect{N}_T(\grad_{x+y}X) \lesssim C_1(T_0)\exp \left( \alpha T\right),
\end{align*}
which explains this proposition. 
\end{proof}

Let us conclude this section with some a-priori estimates for the higher order derivatives of $(\phi, \Gamma, \Lambda)$ which we will later use to estimate $\sh(2k)$. 
\begin{lemma} \label{high-order-est}
Let $(\phi(t), \Gamma(t), \Lambda(t))$ be a solution to \eqref{TDHFB} and $I=[T_0, T_1]$. Then there exists $\delta>0$ such that  for $j\ge 2$ we have the estimates
\begin{subequations}
\begin{align}
\vect{N}_{I}(\grad_{x+y}^j X)
&\lesssim_j\ C_j(T_0) + |I|^\delta \sum_{i=0}^{\lceil j/2 \rceil}\vect{N}_{I}(\grad_{x+y}^iX) \vect{N}_{I}(\grad_{x+y}^{j-i} X)
\end{align}
\end{subequations}
where 
\begin{align}
C_j(T) =&\ \llp{(\grad_x^j\phi(T,\cdot), \grad_{x+y}^j\Gamma(T, \cdot), \grad_{x+y}^j\Lambda(T, \cdot))}_{\mathcal{X}^\alpha}.
\end{align}
In particular, for $|I|$ sufficiently small, we have that
\begin{align}
\vect{N}_{I}(\grad_{x+y}^j X) \lesssim C_j(T_0).
\end{align}

\end{lemma}

\begin{proof}
\noindent The proof is similar to the proof of Lemma \ref{low-order-est}. 
\end{proof}

\begin{prop}\label{4.9}
Let $T>0$. Suppose ($\phi(t), \Gamma(t), \Lambda(t)$) is a solution to the TDHFB system.
Then there exists a constant $\kappa>0$, depending on $j$ and independent of $T$, such that we have the following uniform in $N$ a-priori estimate
\begin{align}\label{gradsq-apriori}
\vect{N}_T(\grad_{x+y}^jX) \lesssim_j&\ \exp\left(\kappa T \right).
\end{align}
\end{prop}

\begin{proof}
The proof is an induction argument. It is similar to the proof of Lemma \ref{4.6} so we omit it. 
\end{proof}

\section{Estimates for $\sh(2k)$ and $\ch(2k)$}
In order to obtain Fock space estimates for the error terms in our quasifree-approximation, we need to establish  
estimates for $\sh(2k)$ and $\ch(2k)$. 

Recall the equation for $\sh(2k)$ is given by
\begin{subequations}
\begin{equation}\label{sh2-eq}
\begin{aligned}
\vect{S}(\sh(2k)) =& -2v_N(x-y)\Lambda-(v_N\Lambda)\circ p_2-\bar p_2\circ (v_N\Lambda)\\
&-((v_N\ast\diag\Gamma)(x) + (v_N\ast\diag\Gamma)(y))\sh(2k)\\
&-(v_N\Gamma)\circ\sh(2k)-\sh(2k)\circ(v_N\Gamma).
\end{aligned}
\end{equation}
Since $\ch(2k) = \delta(x-y) + p_2$ then we could simply write
an equation for $p_2$ instead, that is
\begin{equation}\label{ch2-eq}
\begin{aligned}
\vect{S}_\pm(\bar p_2) =& -(v_N\Lambda)\circ \overline{\sh(2k)}+\sh(2k)\circ (v_N\Lambda)\\
&-((v_N\ast\diag\Gamma)(x) - (v_N\ast\diag\Gamma)(y))\bar p_2\\
&-(v_N\Gamma)\circ \bar p_2+\bar p_2\circ(v_N\Gamma).
\end{aligned}
\end{equation}
\end{subequations}

\begin{remark}
The main difficulty one faces when handling \eqref{sh2-eq} and \eqref{ch2-eq} is in controlling the singular 
potential term $v_N\Lambda$ of the $\sh(2k)$ equation. In a former version of the paper, we required a bounded on the 
time-derivative of the initial which greatly limited the kind of initial data that we could study. 
In particular, the requirement of bounded one time-derivative rules out the possibility of studying 
the case $k(0)\equiv 0$, i.e. the coherent state case (c.f. Remark 2.6 in \cite{GM3}). 
\end{remark}

The following theorem is the main result of this section.
\begin{theorem}\label{sh2-thm}
 Let $\sh(2k)$ satisfies \eqref{sh2-eq} with initial conditions given by the Nonlinear Theorem.  
 Then for $j=0, 1, 2$ we have the estimates
 \begin{subequations}
 \begin{align}
  \llp{\grad_{x+y}^j\sh(2k)(t, \cdot, \cdot)}_{L^2(dxdy)}\lesssim&\ \exp(\kappa T)\label{endpoint1}\\
    \llp{\grad_{x+y}^j\sh(2k)(t, \cdot, \cdot)}_{L^2([0, T])L^6(dx)L^2(dy)}\lesssim&\ \exp(\kappa T)\label{endpoint2}\\
  \sup_x\llp{\sh(2k)(x,\cdot)}_{L^2(dz)} \lesssim&\ \exp(\kappa T)\label{L^2L^infty}
 \end{align}
 \end{subequations}
for some constant $\kappa>0$ and uniform in $t$ on the interval $[0, T]$. Similar estimates hold for $p_2$. 
\end{theorem}

As an immediate consequence, we have the following corollary.
\begin{cor}\label{sh-corollary}
Suppose $(\phi(t), k(t))$ satisfies \eqref{HFB1}. Then there is some constant $\kappa>0$ such that the following estimates
 \begin{subequations}
\begin{align}
 \llp{u(t, \cdot, \cdot)}_{L^2(dxdy)} \lesssim&\ \exp\left( \kappa T\right)\\
  \llp{p(t, \cdot, \cdot)}_{L^2(dxdy)} \lesssim&\ \exp\left( \kappa T\right)\\
 \sup_x\llp{u(t, x, \cdot)}_{L^2(dy)} \lesssim&\ \exp\left( \kappa T\right)\\
  \sup_x\llp{p(t, x, \cdot)}_{L^2(dy)} \lesssim&\ \exp\left( \kappa T\right)
\end{align}
 \end{subequations}
 hold uniformly in $N$ on the interval $[0, T]$. 
\end{cor}

\begin{proof}
 Let us recall that $\ch(k) = \delta + p$ is an invertible operator since $p$ is bounded (semi)positive-definite operator. In fact, we have
 \begin{align}
  \llp{f}_{L^2(dx)}^2 \lesssim \inprod{f}{\ch(k)f} \ \ \implies \ \ \llp{\ch(k)^{-1}}_\text{op} \lesssim 1.
 \end{align}
 Hence, using the identity \eqref{hyptrig-id} and Theorem \ref{sh2-thm}, it follows
 \begin{align*}
  \llp{u(t, \cdot, \cdot)}_{L^2(dxdy)} \lesssim&\ \llp{\ch(k)^{-1}}_\text{op} \llp{\sh(2k)(t,\cdot, \cdot)}_{L^2(dxdy)}\\
  \lesssim&\ \exp\left(\kappa T \right).
  \end{align*}
Likewise, using identity \eqref{hyptrig-id2}, we see that
\begin{align*}
 \llp{p(t, \cdot, \cdot)}_{L^2(dxdy)} \lesssim \llp{u(t,\cdot, \cdot)}_{L^2(dxdy)} \lesssim  \exp\left(\kappa T \right).
\end{align*}

To prove the remaining inequalities, we will use a duality argument. Observe we have that
\begin{align*}
 \llp{u}_{L^\infty(dx)L^2(dy)}  =& \sup_{\llp{f}_{L^1(dx)L^2(dy)}=1}\left|\int dxdy\ u(x, y)f(x, y)\right|\\
 \leq&\ \sup_{\llp{f}_{L^1(dx)L^2(dy)}=1}\left|\int dxdzdy\ \sh(2k)(x,z)\ch(k)^{-1}(z, y)f(x, y)\right|\\
 \leq&\ \sup_{\llp{f}_{L^1(dx)L^2(dy)}=1}\int dx\ \llp{\sh(2k)(x, \cdot)}_{L^2(dz)}\llp{\ch(k)^{-1}}_\text{op}\llp{f(x, \cdot)}_{L^2(dy)}\\
 \leq&\ \sup_x\llp{\sh(2k)(x, \cdot)}_{L^2(dz)}.
\end{align*}
Hence the result follows. The estimate for $p$ follows from the identity $p_2\circ(\delta+\frac{1}{2}p)^{-1} = 4p$. 
\end{proof}

We also have the following corollary of the above corollary.
\begin{cor}\label{sh-strichartz-corollary}
 Let $(\phi(t), k(t))$ solves \eqref{HFB1}. Then we have the following uniform in $N$  endpoint Strichartz estimates
 \begin{subequations}
 \begin{align}
  \llp{u}_{L^2([0, T])L^6(dx)L^2(dy)}\lesssim&\ \exp(\kappa T),\\
  \llp{p}_{L^2([0, T])L^6(dx)L^2(dy)}\lesssim&\ \exp(\kappa T).
 \end{align}
 \end{subequations}
\end{cor}

The proof of Corollary \ref{sh-strichartz-corollary} will be postponed till the end of the section since the crux of the argument follows
from the proof of Theorem \ref{sh2-thm}.

\begin{remark}
 The reader should note that by interpolating \eqref{endpoint1} and \eqref{endpoint2} we recover all
 the Strichartz estimates of the form $L^qL^rL^2$, where $(q, r)$ is an admissible pair, for $\sh(2k)$. Likewise, we have all the Strichartz 
 estimates for $u$ and $p$. 
\end{remark}

\begin{prop}\label{sh-prop-1}
Let $s_a^0$ be the solution to 
\begin{equation}\label{toy}
\begin{aligned}
 \vect{S}(s_a^0) =&\ -2v_N(x-y)\Lambda\\
 s_a^0(0, x, y) =&\ \sh(2k)(0, x, y).
\end{aligned}
\end{equation}
Then we have the estimate
\begin{align*}
 \Lp{s_a^0}{L^\infty([0, T])L^2(dxdy)} + \Lp{s_a^0}{L^2([0, T])L^6(dx)L^2(dy)} \lesssim 1+T^{2}
\end{align*}
for some constant $\kappa$ independent of $T$. Moreover, we also have 
\begin{align}
 \llp{\grad_{x+y}^j s_a^0}_{L^2(dxdy)} + \llp{\grad^j_{x+y}s_a^0}_{L^2([0, T])L^6(dx)L^2(dy)}\lesssim \exp\left(\kappa T \right)
\end{align}
for $j=1, 2$. 
\end{prop}

\begin{lemma}\label{lemma 5.2}
 Let $b>0$ and $c(t)$ is the characteristic function on $[0, T]$ where $T>1$. Suppose we have
 \begin{align}
  \Lambda = \int^\infty_{-\infty}ds\ c(t-s)e^{i(t-s)\lapl_{x, y}} F(s)
 \end{align}
then it follows
\begin{align}
 \llp{\Lambda}_{X^b} \lesssim_b T\llp{F}_{X^{b-1}}.
\end{align}
\end{lemma}

\begin{proof}
 Let $\mathcal{F}$ denote the spacetime Fourier transform, then we see that
 \begin{align*}
  &\mathcal{F}\left(  \int^\infty_{-\infty}ds\ c(t-s)e^{i(t-s)\lapl_{x, y}} F(s)\right)\\
  & = \int^\infty_{-\infty} dt\ e^{-it\tau} \int^\infty_{-\infty} ds\ c(t-s)e^{-i(t-s)(|\xi|^2+|\eta|^2)} \widehat F(s)\\
  & = \hat c(\tau+|\xi|^2+|\eta|^2) \widetilde F(\tau, \xi, \eta).
 \end{align*}
In particular, we have the estimate
\begin{align*}
 |\hat c(\tau+|\xi|^2+|\eta|^2)| \lesssim 
 \left|\frac{\sin(\frac{1}{2}T(\tau+|\xi|^2+|\eta|^2))}{\tau+|\xi|^2+|\eta|^2} \right|\lesssim T\langle \tau+|\xi|^2+|\eta|^2\rangle^{-1}.
\end{align*}
Then the result follows.
\end{proof}

To prove Proposition \ref{sh-prop-1}, it suffices to show the following proposition.
\begin{prop}\label{sh-prop-2}
Suppose $c(t)$ is the characteristic function of $I=[0, T]$ and let
\begin{align}\label{forcing-term}
 E(t, x, y)= \int^\infty_{-\infty} c(t-s)e^{i(t-s)\lapl_{x, y}}(c(s)v_N(x-y)\Lambda(s))\ ds.
\end{align}
Then we have the estimate
\begin{equation}\label{main-sh-est}
\begin{aligned}
  &\Lp{E}{L^\infty(I)L^2(dxdy)}+\Lp{E}{L^2(I)L^6(dx)L^2(dy)}\\
  &\lesssim\ T\sup_z\Lp{(|\bd_t|^{\frac{1}{4}+\varepsilon}
  +|\grad_x|^{\frac{1}{2}+\varepsilon})\Lambda(t,x, x+z)}{L^2(I)L^2(dx)}.
\end{aligned}
\end{equation}
\end{prop}
It is helpful to rewrite 
\begin{align}
 v_N(x-y)\Lambda(s) &= \int dz\ v_N(z)\delta(x-y-z) \Lambda\left(s, \frac{x+y+z}{2}, \frac{x+y-z}{2} \right) \nonumber\\
 &=: \int dz\ v_N(z)\delta(x-y-z)f_z(s, x+y).
\end{align}
Then let us define
\begin{align}\label{Ez}
 E_z(t, x, y) =  \int^\infty_{-\infty}ds\ c(t-s)e^{i(t-s)\lapl}(c(s)\delta(x-y-z)f_z(s, x+y))
\end{align}
and write
\begin{align}
 E(t, x, y) &=  \int dz\ v_N(z) E_z(t, x, y).
\end{align}
The following lemmas provide uniform in $z$ estimates for $E_z$.
\begin{lemma}\label{piece1}
 Fix $z$ and $\delta>0$. Then we have the estimates
 \begin{subequations}
 \begin{align}
  \Lp{|\bd_t|^{\frac{1}{4}+\varepsilon}P_{|\xi-\eta| \gg |\tau+|\xi+\eta|^2|^{\frac{1}{2}}} E_z}{X^{\frac{1}{4}-\delta}}
  \lesssim_\delta&\ T\Lp{|\bd_t|^{\frac{1}{4}+\varepsilon}f_z}{L^2(I)L^2(dx)}\label{sh_est1}\\
    \Lp{|\grad_x+\grad_y|^{\frac{1}{2}+\varepsilon}P_{|\xi-\eta| \gg |\tau+|\xi+\eta|^2|^{\frac{1}{2}}} E_z}{X^{\frac{1}{4}-\delta}}
  \lesssim_\delta&\ T\Lp{|\grad_x|^{\frac{1}{2}+\varepsilon}f_z}{L^2(I)L^2(dx)}. \label{sh_est2}
 \end{align}
 \end{subequations}
\end{lemma}

\begin{proof}
 Applying Lemma \ref{lemma 5.2} to $|\bd_t|^{\frac{1}{4}+\varepsilon}P E_z$ yields
 \begin{align*}
  \Lp{|\bd_t|^{\frac{1}{4}+\varepsilon}P_{1'} E_z}{X^{\frac{1}{4}-\delta}}\lesssim  T
  \Lp{|\bd_t|^{\frac{1}{4}+\varepsilon}(P_{1'} c\delta(x-y-z)f_z(s, x+y))}{X^{-\frac{3}{4}-\delta}}
 \end{align*}
 where $P_{1'}$ is a short hand for the projection operator $P_{|\xi-\eta| \gg |\tau+|\xi+\eta|^2|^{\frac{1}{2}}}$.

Directly computing the spacetime Fourier transform of the RHS gives us
\begin{align*}
 &\Lp{|\bd_t|^{\frac{1}{4}+\varepsilon}(P_{1'}c\delta(x-y-z)f_z(s, x+y))}{X^{-\frac{3}{4}-\delta}}^2\\
 & \lesssim \int_{|\xi-\eta|\gg |\tau+|\xi+\eta|^2|^{\frac{1}{2}}} d\tau d\xi d\eta\ 
 \frac{|\tau|^{\frac{1}{2}+2\varepsilon}|\widetilde{cf_z}(\tau, \xi+\eta)|^2}{\langle \tau
 +|\xi+\eta|^2+|\xi-\eta|^2 \rangle^{\frac{3}{2}+2\delta}}\\
 &\lesssim \int d\tau d(\xi+\eta)\ |\tau|^{\frac{1}{2}+2\varepsilon}|\widetilde{cf_z}(\tau, \xi+\eta)|^2 \int^\infty_{2a} \frac{\rho^2 d\rho}{\langle \tau
 +|\xi+\eta|^2+\rho^2 \rangle^{\frac{3}{2}+2\delta}}
\end{align*}
where $a := |\tau+|\xi+\eta|^2|^{\frac{1}{2}}$. Hence it remains to evaluate the integral in $\rho$. There are two cases:
either $\tau+|\xi+\eta|^2>0$ or $\tau+|\xi+\eta|^2<0$. Both cases are essentially the same. In the latter case,
we have 
\begin{align*}
 \int^\infty_{2a} \frac{\rho^2 d\rho}{(1+|a^2-\rho^2|)^{\frac{3}{2}+2\delta}} \lesssim  \int^\infty_{a} \frac{\rho^2 d\rho}{(1+\rho^2)^{\frac{3}{2}+2\delta}} \lesssim_\delta 1.
\end{align*}
The proof of \eqref{sh_est2} is similar. 
\end{proof}
In the region $|\xi-\eta|\lesssim a$, we further subdivide into the cases $\tau+|\xi+\eta|^2>0$ and
$\tau+|\xi+\eta|^2<0$. 
\begin{lemma}\label{piece2}
Fix $z$ and $\delta>0$. Then we have the estimate
\begin{subequations}
 \begin{align}
  \Lp{|\bd_t|^{\frac{1}{4}+\varepsilon}P_{\{|\xi-\eta| \lesssim a\}\cap\{\tau+|\xi+\eta|^2>0\}} E_z}{X^{\frac{1}{4}-\delta}}
  \lesssim&\ T\Lp{|\bd_t|^{\frac{1}{4}+\varepsilon}(cf_z)}{L^2(dtdx)}.\\
    \Lp{|\grad_x+\grad_y|^{\frac{1}{2}+\varepsilon}P_{\{|\xi-\eta| \lesssim a\}\cap\{\tau+|\xi+\eta|^2>0\}} E_z}{X^{\frac{1}{4}-\delta}}
  \lesssim&\ T\Lp{|\grad_x|^{\frac{1}{2}+\varepsilon}(cf_z)}{L^2(dtdx)}.
 \end{align}
 \end{subequations}
\end{lemma}

\begin{proof}
 The proof is similar to the proof of Lemma \ref{piece1}. It suffices to just bound the integral
 \begin{align*}
  \int^{2a}_0 \frac{\rho^2 d\rho}{(1+a^2+\rho^2)^{\frac{3}{2}+2\delta}} \lesssim \int^{2a}_0 \frac{\rho^2 d\rho}{(1+a^2)^{\frac{3}{2}}} \lesssim 1.
 \end{align*}
 Hence we have
 \begin{align*}
  \Lp{|\bd_t|^{\frac{1}{4}+\varepsilon}P_{1''} E_z}{X^{\frac{1}{4}-\delta}} \lesssim\ T\Lp{|\bd_t|^{\frac{1}{4}+\varepsilon}(cf_z)}{L^2(dtdx)}
 \end{align*}
where $P_{1''} := P_{\{|\xi-\eta| \lesssim a\}\cap\{\tau+|\xi+\eta|^2>0\}}$. 
\end{proof}

\begin{lemma}\label{piece3}
Fix $z$ and $0<\delta<\frac{1}{2}$. Then we have the estimate
 \begin{align}
  \Lp{P_{\{|\xi-\eta| \lesssim a\}\cap\{\tau+|\xi+\eta|^2<0\}} E_z}{X^{\frac{1}{2}+\delta}}
  \lesssim_\delta T\Lp{(|\bd_t|^{\frac{1}{4}+\varepsilon}+|\grad_x|^{\frac{1}{2}+\varepsilon})(cf_z)}{L^2(dtdx)}.
 \end{align}
\end{lemma}

\begin{proof}
 Again, applying Lemma \ref{lemma 5.2} yields
 \begin{align*}
  \Lp{P_2 E_z}{X^{\frac{1}{2}+\delta}}\lesssim\  
  T\Lp{P_2c\delta(x-y-z)f_z(s, x+y)}{X^{-\frac{1}{2}+\delta}}
 \end{align*}
where $P_2$ denotes the projection operator $P_{\{|\xi-\eta| \lesssim a\}\cap\{\tau+|\xi+\eta|^2<0\}}$.

Take the spacetime Fourier transform of the RHS gives us
\begin{align*}
 &\Lp{P_2c\delta(x-y-z)f_z(s, x+y)}{X^{-\frac{1}{2}+\delta}}^2 \\
 & \lesssim \int_{|\xi-\eta|\lesssim a} d\tau d\xi d\eta\ \frac{|\widetilde{cf_z}(\tau, \xi+\eta)|^2}{\langle \tau
 +|\xi+\eta|^2+|\xi-\eta|^2 \rangle^{1-2\delta}}\\
 &\lesssim \int d\tau d(\xi+\eta)\ |\widetilde{cf_z}(\tau, \xi+\eta)|^2 \int^{2a}_0 \frac{\rho^2 d\rho}{(1+|a^2-\rho^2|)^{1-2\delta}}.
\end{align*}
Hence, it suffices to estimate the integral in $\rho$. Observe we have 
\begin{align*}
\int^{2a}_0 \frac{\rho^2 d\rho}{(1+|a^2-\rho^2|)^{1-2\delta}} \lesssim \int^a_0 \frac{a^2 du}{(1+au)^{1-2\delta}} \lesssim a^{1+4\delta}
\end{align*}
which completes our proof of the lemma. 
\end{proof}

\begin{proof}[Proof of Proposition \ref{sh-prop-2}]

Write $E = P_1E + P_2E$ where $P_1 = P_{1'} +P_{1''}$. Then, by Minkowski and Sobolev inequalites, we have
\begin{align*}
 \Lp{cE}{L^\infty(dt)L^2(dxdy)}\lesssim&\ (\int v)\sup_z\llp{|\bd_t|^{\frac{1}{4}+\varepsilon}c(t)P_1E_z}_{L^{4-\varepsilon}(dt)L^2(dxdy)}\\
 &+(\int v)\sup_z\llp{c(t)P_2 E_z}_{L^\infty(dt) L^2(dxdy)}.
\end{align*}
Now, by Lemma \ref{lemma 4.1}-\ref{lemma 4.2} (c.f. Lemma 2.9 in \cite{Tao}), we have the estimate
\begin{align*}
  \Lp{cE}{L^\infty(dt)L^2(dxdy)}\lesssim&\ \sup_z\llp{|\bd_t|^{\frac{1}{4}+\varepsilon}cP_1E_z}_{X^{\frac{1}{4}-\delta}}
  + \sup_z\llp{cP_2 E_z}_{X^{\frac{1}{2}+\delta}}
\end{align*}
where $\delta = \frac{\varepsilon}{4(4-\varepsilon)}$. Finally, combining Lemma \ref{piece1}-\ref{piece3}, we obtain 
the desired estimate
\begin{align*}
 \Lp{cE}{L^\infty(dt)L^2(dxdy)} \lesssim T\sup_z\Lp{|i\bd_t-\lapl_x|^{\frac{1}{4}+\varepsilon}c\Lambda(t,x, x+z)}{L^2(dtdx)}.
\end{align*}

For the other endpoint, we use ``Sobolev at an angle'' (c.f. Lemma 8.3 in \cite{GM4}) then it follows 
\begin{align*}
 \Lp{cE}{L^2(dt)L^6(dx)L^2(dy)} \lesssim&\ \llp{\langle\grad_x+\grad_y\rangle^{\frac{1}{2}}cP_1E}_{L^2(dt)L^3(dx)L^2(dy)}\\
 &\ + \llp{cP_2E}_{L^2(dt)L^6(dx)L^2(dy)}\\
 \lesssim&\ \llp{\langle\grad_x+\grad_y\rangle^{\frac{1}{2}}cP_1E}_{X^{\frac{1}{4}-\delta}}+\llp{cP_2E}_{X^{\frac{1}{2}+\delta}}.
\end{align*}
Then estimate \eqref{main-sh-est} follows.  
\end{proof}

Next, we include some potential terms from \eqref{sh2-eq} to \eqref{toy}. Let us start with the following lemma.

\begin{lemma}\label{bounded-pot-lem}
Consider the ``potential" operator with kernel
\begin{equation}
\begin{aligned}
V(u)(x, y) :=&\ ((v_N\ast\diag\Gamma)(x) + (v_N\ast\diag\Gamma)(y))u(x, y)\\
&\ +(v_N\Gamma)\circ u(x, y)+u\circ(v_N\Gamma)(x, y). 
\end{aligned}
\end{equation}
Let us also define 
\begin{subequations}
\begin{align}
(\bd_x^j V)(u):=&\  ((v_N\ast \bd^j\diag\Gamma)(x)+(v_N\ast \diag\Gamma)(y))u\\
&\ - v_N\Pi^j(\Gamma)\circ u-u\circ v_N\Gamma,\nonumber\\
(\bd_y^j V)(u):=&\  ((v_N\ast \diag\Gamma)(x)+(v_N\ast \bd^j\diag\Gamma)(y))u\\
&\ - v_N\Gamma\circ u-u\circ v_N\Pi^j(\Gamma)\nonumber
\end{align}
where $\Pi^j(\Gamma)(x, y):=(\grad_{x+y}^{j} \Gamma)(x, y).$ 
\end{subequations}
Then, for $j\ge 0$ and interval $I=[T_0, T_1]$, we have the following uniform in $N$ estimates
\begin{align}
 \llp{(\bd^j V)(u)}_{L^1(I)L^2(dxdy)} \lesssim&\ \vect{N}_I(\grad_{x+y}^j\Gamma)\llp{u}_{L^2(I)L^6(dx)L^2(dy)},\label{pot-bound-est}\\
 \llp{(\bd^j V)(u)}_{L^2(I)L^\frac{6}{5}(dx)L^2(dy)} \lesssim&\ \vect{N}_I(\grad_{x+y}^j\Gamma)\llp{u}_{L^\infty(I)L^2(dxdy)}. \label{pot-bound-est2}
\end{align}
\end{lemma}

\begin{proof}
It suffices to consider the case $j=1$. Moreover,  it also suffices to handle the terms 
$(v_N\ast \bd^j \diag \Gamma(x)) u$ and $v_N\Pi^j(\Gamma)\circ u$. 
For the first term, we apply H\"older, Young, and Sobolev inequalities to get
\begin{align*}
 &\int_I dt\ \llp{(v_N\ast \bd^j\diag\Gamma)(x)u}_{L^2(dxdy)}\\
 &\lesssim\ \int_I  dt\ \llp{\grad_x^{\frac{1}{2}}\diag(\grad_{x+y}^j\Gamma)}_{L^2(dx)}\llp{u}_{L^6(dx)L^2(dy)}\\
 &\lesssim\ \vect{N}_I(\grad_{x+y}^j \Gamma)\llp{u}_{L^2(I)L^6(dx)L^2(dy)}
\end{align*}
For the other term, we have
\begin{align*}
 & \int_I dt\ \llp{(v_N\Pi^j(\Gamma))\circ u}_{L^2(dxdy)}\\
 &\lesssim\int_I dt\ \int dz\ v_N(z) \llp{\Pi^j(\Gamma)(x, x-z)}_{L^3(dx)}\llp{u}_{L^6(dx)L^2(dy)} \\
  &\lesssim \int_I  dt\ \int dz\ v_N(z) \llp{\grad^{\frac{1}{2}}_x\Pi^j(\Gamma)(x, x-z)}_{L^2(dx)}\llp{u}_{L^6(dx)L^2(dy)}\\&\lesssim  \vect{N}_{I}(\grad_{x+y}^j \Gamma)\llp{u}_{L^2(I)L^6(dx)L^2(dy)}
\end{align*}
 Thus, \eqref{pot-bound-est} follows immediately. 
 
 Similarly, by  H\"older, Young, and Sobolev inequalities, we see that
\begin{align*}
 &\Lp{ (v_N\ast\bd^j \diag\Gamma)(x) u}{L^2(I)L^\frac{6}{5}(dx)L^2(dy)} \\
 &\lesssim\  \Lp{ (v_N\ast \diag(\grad_{x+y}^j\Gamma)(x)}{L^2(I)L^3(dx)}\llp{u}_{L^\infty(I)L^2(dxdy)}\\
 &\lesssim\ \vect{N}_{I}(\grad_{x+y}^j \Gamma) \llp{u}_{L^\infty(I)L^2(dxdy)}.
 \end{align*}
For the other term, we see that
\begin{align*}
 &\llp{(v_N\Pi^j(\Gamma))\circ u}_{L^2(I)L^{\frac{6}{5}}(dx)L^2(dy)}\\
 &\lesssim \int dz\ v_N(z)\llp{\Pi^j(\Gamma)(x, x-z)u(x-z, y)}_{L^2(I)L^{\frac{6}{5}}(dx)L^2(dy)}\\
 &\lesssim \vect{N}_I(\grad_{x+y}^j \Gamma) \llp{u}_{L^\infty(I)L^2(dxdy)}.
\end{align*}
Then \eqref{pot-bound-est2} follows. 
\end{proof}

\begin{prop}\label{sh_prop_3}
 Let $s_a$ be a solution to 
 \begin{equation}\label{toy2}
\begin{aligned}
 \vect{\tilde S}(s_a) =&\ -2v_N(x-y)\Lambda\\
 s_a(0, x, y) =&\ \sh(2k)(0, x, y)
\end{aligned}
\end{equation}
where $\vect{\tilde S} = \vect{S}+V$ and $V$ as in the above lemma.
Then the following estimate holds
\begin{align}\label{sh_energy_est}
 \Lp{s_a}{L^\infty([0, T])L^2(dxdy)}\lesssim 1+ T^{3}.
\end{align}
Moreover, for every $j\ge 1$, there exists $\kappa>0$ such that we have the etimate 
\begin{align}
 \llp{\grad_{x+y}^j s_a}_{L^2(dxdy)} \lesssim_j \exp\left(\kappa T \right).
\end{align}
\end{prop}

\begin{proof}
 The idea is to write $s_a = s_a^0+s_a^1$ where $s_a^0$ is a solution to \eqref{toy}. Then we see that $s_a^1$ satisfies
 the Cauchy problem 
 \begin{equation}\label{toy3}
 \begin{aligned}
  &\vect{\tilde S}(s_a^1) = -V(s_a^0)\\
  & s_a^1(0, x, y) = 0.
 \end{aligned}
\end{equation}
To obtain energy estimate for $s_a^1$, we employ the identity
\begin{align}
 \vect{S}_\pm(s_a^1\circ \overline{s_a^1})+V(s_a^1)\circ \overline{s_a^1} - s_a^1\circ \overline{V(s_a^1)} = -V(s_a^0)\circ \overline{s_a^1}+s_a^1\circ \overline{V(s_a^0)}
\end{align}
which means
\begin{align*}
\frac{d}{dt}\llp{s_a^1(t)}_{L^2(dxdy)}^2  \leq \left|\Tr\left(  \vect{S}_\pm(s_a^1\circ \overline{s_a^1})+V(s_a^1)\circ \overline{s_a^1} - s_a^1\circ \overline{V(s_a^1)}\right)\right|.
\end{align*}
Then it follows that
\begin{align*}
\sup_{t\in [0, T]}\llp{s_a^1(t)}_{L^2(dxdy)} \lesssim \int^T_0 dt\ \llp{V(s_a^0(t))}_{L^2(dxdy)}.
\end{align*}
By Lemma \ref{bounded-pot-lem}, Proposition \ref{4.3}, and Proposition \ref{sh-prop-1}, we obtain \eqref{sh_energy_est}.

The case of $\grad_{x+y}^j s_a$ is similar. We begin by writing the equation
\begin{align}
 \vect{\tilde S}(\grad_{x+y}^j s_a^1) = [V, \grad_{x+y}^j]s_a^1 -\grad_{x+y}^j(V(s_a^0)).
\end{align}
Again, by energy method, we have the estimate
\begin{align*}
\llp{\grad_{x+y}^j s_a^1}_{L^\infty([0, T])L^2(dxdy)}\lesssim&\ \int^T_0 dt\ 
\Lp{[V, \grad_{x+y}^j]s_a^1}{L^2(dxdy)}\\
&+\int^T_0 ds\ \Lp{\grad_{x+y}^jV(s_a^0)}{L^2(dxdy)}.
\end{align*}
The forcing term $\grad_{x+y}^j(V(s_a^0))$ is handled in the same manner as above. For the commutator term, we have 
the estimate
\begin{align*}
 \int^T_0 dt\ 
\llp{[V, \grad_{x+y}^j]s_a^1}_{L^2(dxdy)} \lesssim&\  \int^T_0 dt\ 
\Lp{(\bd^j V) (s_a^1)}{L^2(dxdy)}.
\end{align*}
Hence, by Lemma \ref{bounded-pot-lem} and \eqref{sh_strichartz} in the following corollary, we obtain the desired result.
\end{proof}

\begin{cor}
 Suppose $s_a$ is a solution to \eqref{toy2}. Then the following estimate holds
 \begin{align}\label{sh_strichartz}
  \llp{s_a}_{L^2([0, T])L^6(dx)L^2(dy)} \lesssim 1+T^{4}.
 \end{align}
 Moreover, for $j\ge 1$ we also have 
\begin{align}
 \llp{\grad_{x+y}^j s_a}_{L^2([0,T])L^6(dx)L^2(dy)} \lesssim_j \exp\left(\kappa T \right).
\end{align}
\end{cor}
\begin{proof}
Let us once again reduce the problem to \eqref{toy3}. By the endpoint Strichartz estimates, we have
\begin{align*}
 &\Lp{s_a^1}{L^2([0, T])L^6(dx)L^2(dy)} \\
 &\lesssim\ (\text{data}) + \Lp{V(s_a^0)}{L^2([0, T])L^{\frac{6}{5}}(dx)L^2(dy)}+ \Lp{V(s_a^1)}{L^2([0, T])L^{\frac{6}{5}}(dx)L^2(dy)}.
\end{align*}
By Lemma \ref{bounded-pot-lem}, Proposition \ref{sh-prop-1}, and \eqref{sh_energy_est}, we establish estimate
\eqref{sh_strichartz}.
The proof for the higher derivatives are similar. 
\end{proof}

\begin{proof}[Proof of Theorem \ref{sh2-thm}]
 Let us write $\sh(2k):= s_2 = s_a + s_e$ where $s_a$ satisfies \eqref{toy2}. Then we have the coupled system 
 \begin{subequations}
 \begin{align}
    \vect{\tilde S}(s_e) &=\ -(v_N\Lambda)\circ p_2 -\bar p_2 \circ (v_N\Lambda)\\
    \vect{\tilde S}_\pm(\bar p_2) &=\ -(v_N\Lambda)\circ \bar s_a + s_a \circ (v_N\Lambda) \label{p2-eq}\\
    &\quad\  -(v_N\Lambda)\circ \bar s_e + s_e \circ (v_N\Lambda) \nonumber\\
    &=:\ M -(v_N\Lambda)\circ \bar s_e + s_e \circ (v_N\Lambda) \nonumber
 \end{align}
 \end{subequations}
 where $\vect{\tilde S}_\pm = \vect{S}_\pm + V_\pm$ and 
 \begin{align*}
 V_\pm(p) = ((v_N\ast\diag\Gamma)(x) - (v_N\ast\diag\Gamma)(y))p+(v_N\Gamma)\circ p-p\circ(v_N\Gamma).
 \end{align*}
Let $I=[0, T]$.  By Strichartz estimates, we have
\begin{align*}
\llp{p_2}_{L^\infty(I)L^2(dxdy)} \lesssim&\ (\text{data})+ \llp{V_\pm (p_2)}_{L^2(I)L^\frac{6}{5}(dx)L^2(dy)}+
 \llp{M}_{L^2(I)L^\frac{6}{5}(dx)L^2(dy)}\\
 &\ +\left(\sup_z \int_I dt\ \llp{\Lambda(t, x-z, x-z)}_{L^3(dx)}^2\llp{s_e(t)}_{L^2(dxdy)}^2\right)^\frac{1}{2}.
\end{align*}
Applying Lemma \ref{bounded-pot-lem}, we have that
\begin{align}
\llp{p_2}_{L^\infty(I)L^2(dxdy)} \lesssim&\ (\text{data}) +T^4\\
&\ + \vect{N}_T(X)\left(\llp{p_2}_{L^\infty(I)L^2(dxdy)}+
\llp{s_e}_{L^\infty(I)L^2(dxdy)}\right)\nonumber
\end{align}
since $\llp{M}_{L^{2}([0, T])L^{\frac{6}{5}}(dx)L^2(dy)} \lesssim 1+T^{4}$. Likewise, we see that
\begin{align*}
 &\llp{s_e}_{L^\infty(I)L^2(dxdy)}\\
 &\lesssim (\text{data}) + \llp{V (s_e)}_{L^2(I)L^\frac{6}{5}(dx)L^2(dy)}\\
 &\quad +\left(\sup_z \int_I dt\ \llp{\Lambda(t, x-z, x-z)}_{L^3(dx)}^2\llp{p_2(t)}_{L^2(dxdy)}^2\right)^\frac{1}{2}\\
 &\lesssim (\text{data})+\vect{N}_T(X)\left(\llp{p_2}_{L^\infty(I)L^2(dxdy)}+
\llp{s_e}_{L^\infty(I)L^2(dxdy)}\right)
\end{align*}

Next, define
\begin{align*}
 E(t) = \llp{s_e}_{L^\infty([0, t])L^2(dxdy)}+\llp{p_2}_{L^\infty([0, t])L^2(dxdy)}.
\end{align*}

Putting the estimates for $p_2$ and $s_e$ together, we have
\begin{equation}
    E(T)\lesssim (\text{data})+T^{\alpha}+N_{T}(X)E(T),
\end{equation}
where $\alpha = 4$. Then we will apply the above estimate on small intervals together with Proposition \ref{4.3}. For convenience, we use $C_1>0$ to represent the implicit constant in the above estimate and we use $C_2>0$ to represent the implicit constant in the estimate of Proposition \ref{4.3}. Here we divide the time interval $[0,T]$ into $l$ consequent subintervals $\{I_i\}_{i}$ where $I_i=[T_i,T_{i+1}]$ satisfying $N_{I_i}(X)\leq \frac{1}{2C_1}$ and $l \sim 2 C_1 C_2 T$. \vspace{3mm}

For a certain interval (without loss of generality, we consider the first interval as an example), we have
  \begin{equation}
     E(T_2)\leq C_1 E(T_1)+C_1+C_1 \cdot \frac{1}{2C_1}E(T_2).
 \end{equation}
 This implies,
   \begin{equation}
     E(T_2)\leq 2C_1 E(T_1)+2C_1.
 \end{equation}
Then we can iterate the arguments on all of the intervals, noticing the number of intervals are around $l \sim 2 C_1 C_2 T$, we have, for some constant $\kappa$,
\begin{align*}
    E(T)=E(T_l)&\leq 2C_1 E(T_{l-1})+2C_1  \leq  2C_1(2C_1E(T_{l-2})+2C_1) +2C_1\\
    & \leq ...\leq (2C_1)^{2C_1C_2T} E(T_1)+\frac{(2C_1)^{2C_1C_2T+1}-1}{2C_1-1}      \lesssim e^{\kappa T}.
\end{align*}
The proof of \eqref{endpoint2} follows from \eqref{endpoint1}.
\end{proof}

Let us complete this section with the proof of Corollary \ref{sh-strichartz-corollary}. We start by proving the following lemma. 

\begin{lemma}\label{w(ch)-est}
Suppose $(\phi(t), k(t))$ solves \eqref{HFB1}. Then there exists some $\kappa>0$ such that following estimates
\begin{subequations}
 \begin{align}
 \llp{\vect{\tilde W}(\ch(2k_t))}_{L^2(dxdy)} \lesssim& \exp(\kappa T)\\
  \llp{\vect{\tilde W}(\ch(k_t))}_{L^2(dxdy)} \lesssim& \exp(\kappa T)
 \end{align}
 \end{subequations}
holds for all $t \in [0, T]$ and independent of $N$. 
\end{lemma}

\begin{proof}
 By \eqref{compact-ch-eq}, Theorem \ref{sh2-thm}, and Proposition \ref{4.9}, we have the estimate
 \begin{align*}
  \llp{\vect{\tilde W}(\ch(2k))}_{L^2(dxdy)} \lesssim \vect{N}_T(\grad_{x+y}^2\Lambda)\llp{\sh(2k)}_{L^\infty([0, T])L^2(dxdy)}
  \lesssim \exp(\kappa T).
 \end{align*}
 
To estimate $\vect{\tilde W}(\ch(k))$, we invoke a variant of identity (30) in \cite{GMM2}, which is an immediate consequence of holomorphic functional calculus, 
\begin{align}
  \vect{\tilde W}(\ch(k)) = \frac{1}{4\pi i} \int_\gamma (u\circ\bar u-z)^{-1} \vect{\tilde W}(\ch(2k))(u\circ\bar u-z)^{-1} \sqrt{1+z}\ dz
\end{align}
where the contour $\gamma$ encloses the spectrum of $u\circ \bar u$  and contains inside the ball of radius $C\llp{u}^2_{L^2(dxdy)}$ for some $C>0$, 
i.e. $|z| \leq C\llp{u}^2_{L^2(dxdy)}$. 
Since $(u\circ\bar u-z)^{-1}$ has bounded operator norm, then it follows from Corollary \ref{sh-corollary}
\begin{align*}
  \llp{ \vect{\tilde W}(\ch(k))}_{L^2(dxdy)} \lesssim&\ 
  |z|^{\frac{3}{2}}\llp{(u\circ \bar u-z)^{-1}}_\text{op}^2\llp{ \vect{\tilde W}(\ch(2k))}_{L^2(dxdy)} \\
  \lesssim&\ \exp\left(\alpha T \right)
\end{align*}
for some constant $\alpha>0$. 
\end{proof}

\begin{proof}[Proof of Corollary \ref{sh-strichartz-corollary}]
 Rewrite \eqref{p2-eq} with $p_2= 2\bar u \circ u = 2p \circ p+4p$ in the differential operator yields
 \begin{equation}
 \begin{aligned}
  \vect{\tilde W}(p) =&\ -\frac{1}{2}\vect{\tilde W}(p\circ p) -\frac{1}{4}[v_N\Lambda, \overline{\sh(2k)}] \\
  =&\ -\frac{1}{2}\vect{\tilde W}(p)\circ p-\frac{1}{2}p\circ \vect{\tilde W}(p)-\frac{1}{4}[v_N\Lambda, \overline{\sh(2k)}] .
 \end{aligned}
 \end{equation}
 Then, by the endpoint Strichartz estimates, we see that
 \begin{align*}
  &\llp{p}_{L^2([0, T])L^6(dx)L^2(dy)}\\
  &\lesssim \llp{V_\pm(p)}_{L^1([0, T])L^2(dxdy)}+\llp{\vect{\tilde W}(p)\circ p}_{L^1([0, T])L^2(dxdy)}\\
  &\quad +\llp{v_N\Lambda\circ \overline{\sh(2k)}}_{L^2([0, T])L^{\frac{6}{5}}(dx)L^2(dy)}\\
  &\lesssim T\vect{\dot N}_T(\Gamma)\llp{p}_{L^\infty([0, T])L^2(dxdy)}+T\llp{\vect{\tilde W}(p)}_{L^\infty([0, T])L^2(dxdy)}\llp{p}_{L^\infty([0, T])L^2(dxdy)}\\
  &\quad +\vect{N}_T(\Lambda)\llp{\sh(2k)}_{L^\infty([0, T])L^2(dxdy)}.
 \end{align*}
Hence the desired estimate follows from Proposition \ref{4.3}, Theorem \ref{sh2-thm}, Corollary \ref{sh-corollary} and  Lemma \ref{w(ch)-est}.

Next, since $\sh(2k) = 2u+2\bar p\circ u$, then it follows
\begin{align*}
 \llp{u}_{L^2([0, T])L^6(dx)L^2(dy)}\lesssim&\  \llp{\sh(2k)}_{L^2([0, T])L^6(dx)L^2(dy)}\\
 &\ + \llp{p}_{L^2([0, T])L^6(dx)L^2(dy)}\llp{u}_{L^\infty([0, T])L^2(dxdy)}
\end{align*}
which again yields the desired estimate. 
\end{proof}

\section{Fock Space Estimates}
 The purpose of this section is to obtain global estimates for the error 
\begin{align}\label{error}
\Lp{\psi_\text{exact}(t)-e^{i\int X_0(t)\ dt}\psi_\text{appr}(t)}{\mathcal{F}}.
\end{align}
Using the identity
\begin{align}
\Lp{\psi_\text{exact}(t)-e^{i\int X_0(t)\ dt}\psi_\text{appr}(t)}{\mathcal{F}} = \Lp{e^{-i\int X_0(t)\ dt}\psi_\text{red}(t)-\Omega}{\mathcal{F}},
\end{align}
it suffices to study the error of the reduced dynamics which we denote by $E:=e^{-i\int X_0(t)\ dt}\psi_\text{red}(t)-\Omega$. 

Recall $E$ satisfies the equation
\begin{align}\label{inhomo2}
\left(\frac{1}{i}\frac{\bd}{\bd t}-\mathcal{H}_\text{red} + X_0 \right)E = (0, 0, 0, X_3, X_4, 0, \ldots) =:\widetilde X
\end{align}
with initial condition $E(0, \cdot) = 0$. For convenience, we adopt the notation
\begin{align}
\vect{S}_F:= \frac{1}{i}\frac{\bd}{\bd t}-\mathcal{H}_\text{red} + X_0 =: \vect{S}_D-\mathcal{P}
\end{align}
where
\begin{align}
\vect{S}_D :=&\  \frac{1}{i}\frac{\bd}{\bd t}- \int a_x^\ast \lapl  a_x- \frac{1}{2N}\int dxdy\ \left\{ v_N(x-y)a_x^\ast a_y^\ast a_x a_y\right\}.
\end{align}
Here we see that $\mathcal{P}$ accounts for all the remaining terms, that is, $\mathcal{P}= \mathcal{H}_\text{red}-X_0-\mathcal{H}$. Note that $\mathcal{P}$ is
not a diagonal Fock space operator. 
\subsection{Explicit Form of $\mathcal{P}$} Let us write down
$\mathcal{H}_\text{red}$ and $\mathcal{P}$ in order to
study \eqref{inhomo2}. The reduced Hamiltonian was first calculated in section 
5 of \cite{GM2} and was later rewritten using Wick's Theorem in section 7 of \cite{GM3}.
Let first us recall the conjugation relations
\begin{subequations}
\begin{align}
 e^{\mathcal{B}(k)}a_xe^{-\mathcal{B}(k)} =& \int dy\ \{\overline{\ch(k)}(x, y) a_y + \sh(k)(x, y)a^\ast_y\} \nonumber\\
 =&\ a(\overline{\ch(k)}(x, \cdot)) + a^\ast(\sh(k)(x, \cdot))=: b_x\\
 e^{\mathcal{B}(k)}a^\ast_x e^{-\mathcal{B}(k)} =& \int dy\ \{\overline{\sh(k)}(x, y)a_y + \ch(k)(x, y) a^\ast_y\} \nonumber\\
 =&\ a(\overline{\sh(k)}(x, \cdot)) + a^\ast(\ch(k)(x, \cdot))=: b^\ast_x. 
\end{align}
\end{subequations}
We will also adopt the notation $a^\ast_x(u):= a^\ast(\sh(k)(x, \cdot))$ and $a_x(\bar u):= a(\overline{\sh(k)}(x, \cdot))$.  Then the reduced Hamiltonian is given by
\begin{subequations}
\begin{align}
 \mathcal{H}_\text{red} =& -N\mu(t) \\
 & + N^\frac{1}{2} \int dx\ \{\tilde h(\phi(t, x))b^\ast_x + \overline{\tilde{h}}(\phi(t, x))b_x\}\label{linear-reduced}\\
 & -\mathcal{H}_g- \nor\left\{\mathcal{I}\begin{pmatrix}w^T & \bar f\\ - f & -w\end{pmatrix}\right\} \label{normal-2nd}\\
 & - \frac{1}{\sqrt{N}} \nor\left\{\int dxdy\  v_N(x-y)(\bar\phi(y)b^\ast_x b_x b_y + \phi(y)b^\ast_y b^\ast_x b_x) \right\}\label{normal-3rd}\\
 & - \frac{1}{2N}\nor\left\{\int dxdy\ v_N(x-y) b^\ast_x b^\ast_y b_x b_y \right\} \label{normal-4th}
\end{align}
\end{subequations}
where $\nor\{\cdot\}$ refers to the normal ordering the the operators with respect to $(a, a^\ast)$. 
In the linear term \eqref{linear-reduced},  $\tilde h$ corresponds to the modified Hartree operator given by
\begin{align*}
 \tilde h(\phi(t, x)):=\vect{S}\phi_t(x_1)+\int dy \left\{(v_N\Lambda_t)(x_1, y)\bar\phi_t(y)+\frac{1}{N}\tilde\alpha^T (t, x_1, y)\phi_t(y) \right\}.
\end{align*}
For the quadratic term \eqref{normal-2nd}, $\mathcal{I}(\cdot)$ is the Lie algebra isomorphism between symplectic matrices of integral operators
and and quadratic polynomials in $(a, a^\ast)$ used in \cite{GMM, GMM2, GM2}, which is given by
\begin{align}
 \mathcal{I}\begin{pmatrix} w^T & \bar{f}\\ -f & - w\end{pmatrix}= 
 -\frac{1}{2} \int dxdy\left\{ w(y, x)a_xa^\ast_y + w(x, y)a^\ast_x a_y 
 -f(x, y) a^\ast_x a^\ast_y -\bar{f}(x, y)a_x a_y \right\} 
\end{align}
where
\begin{subequations}
\begin{align}
 f :=& \left(\vect{\tilde S}(\sh(k))+\overline{\ch(k)}\circ (v_N\Lambda) \right)\circ \ch(k)
 -\left(\vect{\tilde W}(\overline{\ch(k)})-\sh(k)\circ (\overline{v_N\Lambda}) \right)\circ \sh(k)\\
 w :=& \left(\vect{\tilde W}(\overline{\ch(k)})-\sh(k)\circ (\overline{v_N\Lambda}) \right)\circ \overline{\ch(k)}-
  \left(\vect{\tilde S}(\sh(k))+\overline{\ch(k)}\circ (v_N\Lambda) \right)\circ \overline{\sh(k)}
\end{align}
\end{subequations}
and
\begin{align}
 \mathcal{H}_g = \int dxdy\ g(x, y) a^\ast_x a_y 
\end{align}
where $g$ is as in \eqref{g-function}. In particular, equations \eqref{abstract-eq} correspond to $\tilde h(\phi(t, x)) = 0$ 
and $f =0$; see Theorem 7.1 in \cite{GM1}. Hence it follows that
\begin{subequations}\label{off-diagonal-fock-op}
\begin{align}
 \mathcal{P}  =&\ -\mathcal{H}_g-\mathcal{H}_0+\nor\left\{\mathcal{I}\begin{pmatrix}w^T & 0\\ 0 & - w\end{pmatrix}\right\} 
 \label{quadratic-terms}\\
 & - \frac{1}{\sqrt{N}} \nor\left\{\int dxdy\  v_N(x-y)(\bar\phi(y)b^\ast_x b_x b_y + \phi(y)b^\ast_y b^\ast_x b_x) \right\} \label{cubic-terms}\\
 & - \frac{1}{2N}\nor\left\{\int dxdy\ v_N(x-y)(b^\ast_x b^\ast_y b_x b_y-a^\ast_xa^\ast_y a_x a_y)  \right\} \label{quartic-terms}.
\end{align}
\end{subequations}

It is convenient to list out the terms in the normal ordering. Let us start the with quadratic terms $\mathcal{Q}:=-\eqref{quadratic-terms}$. 
Using the fact that $f=0$, we can rewrite $\mathcal{Q}$ as follows
\begin{subequations}
\begin{align}
 \mathcal{Q} =&\ \int dxdy\ \{(v_N\ast\rho_\Gamma)(t, x)\delta(x-y)-(v_N\Gamma)(t, x, y)\}a^\ast_x a_y \label{quadratic-part1}\\
 &\  \int dxdy\ w(x, y) a^\ast_x a_y \label{quadratic-part2}
\end{align}
\end{subequations}
where
\begin{align}
 w = \left(\vect{\tilde W}(\overline{\ch(k)}) - \sh(k)\circ \overline{(v_N\Lambda)}\right)\circ \overline{\ch(k)}^{-1}.
\end{align}
Next, we write down the cubic terms $\mathcal{C}:=-\eqref{cubic-terms}$ into $ \mathcal{C}=\mathcal{C}_1+\mathcal{C}_1^\ast$ where
\begin{subequations}
\begin{align}
 \mathcal{C}_1 &= \frac{1}{\sqrt{N}} \int dxdy\ v_N(x-y)\bar\phi(y) a_x(\bar u)a_x(\bar c) a_y(\bar c) \label{T1}\\
      &\quad + \frac{1}{\sqrt{N}} \int dxdy\ v_N(x-y)\bar\phi(y) a^\ast_x(c)a_x(\bar c) a_y(\bar c) \label{T3}\\
      &\quad +\frac{1}{\sqrt{N}} \int dxdy\ v_N(x-y)\bar\phi(y) a^\ast_x(u)a^\ast_x(c) a_y(\bar c) \label{T4}\\
 &\quad + \frac{1}{\sqrt{N}} \int dxdy\ v_N(x-y)\bar\phi(y) a^\ast_x(c)a^\ast_y(u) a_x(\bar c)\\
&\quad + \frac{1}{\sqrt{N}} \int dxdy\ v_N(x-y)\bar\phi(y) a^\ast_x(u) a^\ast_y(u)a^\ast_x(c)\\
 &\quad + \frac{1}{\sqrt{N}} \int dxdy\ v_N(x-y)\bar\phi(y)a^\ast_x(u) a^\ast_y(u) a_x(\bar u)\\ 
  &\quad + \frac{1}{\sqrt{N}} \int dxdy\ v_N(x-y)\bar\phi(y) a^\ast_x(u)a_x(\bar u) a_y(\bar c)\\ 
   &\quad + \frac{1}{\sqrt{N}} \int dxdy\ v_N(x-y)\bar\phi(y) a^\ast_y(u)a_x(\bar u) a_x(\bar c) \label{Tlast},
\end{align}
\end{subequations}
and $\mathcal{C}_1^\ast$ denotes the adjoint operator of $\mathcal{C}_1$.  
Similarly, the quartic terms \eqref{quartic-terms} contains $2^4+1$ normal ordered terms which we write 
$\eqref{normal-4th} =\mathcal{D}= \mathcal{D}_1+\mathcal{D}_1^\ast + \mathcal{D}_2$ where  
\begin{subequations}
\begin{align}
 \mathcal{D}_1 =&\  \frac{1}{2N}\int dxdy\  v_N(x-y)a^\ast_x(c)a^\ast_y(c)a^\ast_x(u)a^\ast_y(u) \label{D1}\\
 & + \frac{1}{2N}\int dxdy\ v_N(x-y) a^\ast_x(c)a^\ast_y(c)a^\ast_x(u)a_y(\bar c)\label{D2}\\
 &+ \frac{1}{2N}\int dxdy\ v_N(x-y) a^\ast_x(c)a^\ast_y(c)a^\ast_y(u)a_x(\bar c)\label{D3}\\
 &+ \frac{1}{2N}\int dxdy\ v_N(x-y) a^\ast_x(c)a^\ast_x(u)a^\ast_y(u)a_y(\bar u) \label{D4}\\
 &+ \frac{1}{2N}\int dxdy\ v_N(x-y)a^\ast_x(u)a^\ast_y(c)a^\ast_y(u)a_x(\bar u) \label{D5}\\
 &+ \frac{1}{2N}\int dxdy\ v_N(x-y) a^\ast_x(c)a^\ast_x(u)a_y(\bar c)a_y(\bar u) \label{D6}\\ 
  &+ \frac{1}{2N}\int dxdy\ v_N(x-y) a^\ast_x(u)a^\ast_y(c)a_y(\bar c)a_x(\bar u)\label{D7}
 \end{align}
and
\begin{align}
\mathcal{D}_2 =& \frac{1}{2N}\int dxdy\ v_N(x-y) a^\ast_x(u)a^\ast_y(u)a_x(\bar u)a_y(\bar u)\label{D8}\\
&+ \frac{1}{2N}\int dxdy\ v_N(x-y)a^\ast_x(c)a^\ast_y(c)a_x(\bar c)a_y(\bar c)\label{D9}\\ 
&-\frac{1}{2N}\int dxdy\ v_N(x-y)a_x^\ast a_y^\ast a_x a_y \label{D10}.
\end{align}
\end{subequations}

\subsection{Proof of Theorem \ref{main-theorem-2}}Let us first study the solution to 
\begin{align}\label{inhomo}
\vect{S}_D u= f, \ \ \ u(0, \cdot) = 0,
\end{align}
where $u, f$ are Fock vectors. Since $\vect{S}_D$ is a diagonal operator then it follows that \eqref{inhomo} reduces to an inhomogeneous Schr\"odinger equation with potential on each sector. In order to study \eqref{inhomo}, we need to introduce Strichartz and dual Strichartz norm for Fock vectors. 

Let $T>0$ and define the Strichartz norm on the time interval $[0, T]$ by
\begin{align*}
\Lp{u}{S}  =&\ \max\Big\{\Lp{u}{L^\infty(dt)L^2(dx_1 \cdots dx_n)}, \Lp{u}{L^2(dt)L^6(d(x_1-x_2))L^2(d(x_1+x_2) \cdots dx_n)},\\
&\ \Lp{u}{L^2(dt)L^3(d(x_1-x_2)d(x_1+x_2))L^2(dx_3 \cdots dx_n)},  \text{and all other permutations}\Big\}
\end{align*}
and the dual Strichartz norm
\begin{align*}
\Lp{u}{S'}  =&\ \min\Big\{\Lp{u}{L^1(dt)L^2(dx_1 \cdots dx_n)}, \Lp{u}{L^2(dt)L^\frac{6}{5}(d(x_1-x_2))L^2(d(x_1+x_2) \cdots dx_n)},\\
&\ \Lp{u}{L^2(dt)L^\frac{3}{2}(d(x_1-x_2)d(x_1+x_2))L^2(dx_3 \cdots dx_n)}, \text{and all other permutations}\Big\}.
\end{align*}
Also, let $X \in \mathcal{F}$ be a Fock vector such that all but finitely many of the components are zeros, say $X_0, X_1, \ldots, X_k$ are nonzero. Then we define
the Strichartz norm on $\mathcal{F}$ by
\begin{align*}
\Lp{X}{S} = \max\left\{|X_0|, \Lp{X_1}{S}, \ldots, \Lp{X_k}{S} \right\},
\end{align*}
and similarly for the dual Strichartz norm
\begin{align*}
\Lp{X}{S'} = \max\left\{|X_0|, \Lp{X_1}{S'}, \ldots, \Lp{X_k}{S'} \right\}.
\end{align*}

By standard arguments, we have the following lemma.
\begin{lemma}\label{Fock-Strichartz}
Let $f$ be a Fock vector with zero entries past the $k$th sector. Assume $u$ is a solution to 
\begin{align}
\vect{S}_D u = f, \ \ u(0, \cdot) = 0
\end{align}
then we have the estimate
\begin{align}
\Lp{u}{S} \lesssim \Lp{f}{S'}.
\end{align}
In particular, it follows that $\sup_{ t \in [0, T]} \Lp{u}{\mathcal{F}} \lesssim \Lp{u}{S}$.
\end{lemma}

\begin{proof}[Sketch of Proof]
Note for $0<\beta<1$, we see that
\begin{align*}
\frac{1}{N}\Lp{v_N(x_1-x_2) u}{S'} \lesssim&\ \frac{1}{N}\Lp{v_N(x_1-x_2) u}{L^2(dt)L^{\frac{6}{5}}(d(x_1-x_2))L^2(d(x_1+x_2) \cdots dx_n)}\\
\lesssim&\ \frac{1}{N}\Lp{v_N}{L^\frac{3}{2}}\Lp{u}{S} \lesssim N^{\beta-1}\Lp{u}{S}.
\end{align*}
Hence the potential can be treated as a perturbation of the free Schr\"odinger case. 
\end{proof}

The following is the main proposition of this section.
\begin{prop}\label{iteration}
Let $0<\beta<1$ and let $E(t)$ be the solution to \eqref{inhomo}. Assume $(\Lambda(t), \Gamma(t), \phi(t))$ solves the
time-dependent HFB system. Then there exists $\kappa=\kappa(\beta, \varepsilon)$ such that 
\begin{align}
\Lp{E(t)}{\mathcal{F}} \lesssim N^{\frac{\beta-1}{2}}\exp\left(\kappa T\right)
\end{align}
for all $t \in [0, T]$ when $N$ is sufficiently large. 
\end{prop}

\begin{remark}\label{main-remark}
The author believes it is more beneficial to first consider the heuristics behind the main argument of Proposition \ref{iteration}
before proceeding with the proof. 
The main idea behind the proof of the proposition is the fact that $\mathcal{P}$ can be split into a sum $\mathcal{P}_1+ \mathcal{P}_2$ where $\mathcal{P}_2$ 
is a bounded operator on Fock space (for the first $k$ sector) and $\mathcal{P}_1$ satisfies 
$\Lp{\mathcal{P}_1 E}{\mathcal{F}}\lesssim N\exp(\kappa T)\Lp{E}{\mathcal{F}}$, which is the main obstacle we need to overcome, but, fortunately, we also have
$\Lp{\mathcal{P}_1 E}{S'}\lesssim N^{\frac{\beta-1}{2}}\exp(\kappa T)\Lp{E}{S}$. 

However, since $\vect{S}_F$ is not a diagonal operator, it's not immediately obvious 
how one would solve \eqref{inhomo} and estimate the solution. Nevertheless, we begin by considering
\begin{align}\label{diag-inhomo}
\vect{S}_D E_1 = \widetilde X, \ \ E_1(0, \cdot) = 0,
\end{align}
which can be readily solve using Lemma \ref{Fock-Strichartz} and propositions in the following section. Also, note we have the estimate 
$\Lp{E_1}{S}\lesssim \llp{\widetilde X}_{S'}\lesssim N^{\frac{\beta-1}{2}}\exp(\kappa T)$ since $\frac{1}{\sqrt{N}}\Lp{v_N}{\frac{6}{5}} \lesssim N^{\frac{\beta-1}{2}}$.
 With $E_1$, we can construct
\begin{align*}
E_\infty =&\ E_1+\vect{S}_D^{-1}\mathcal{P}_1E_1+(\vect{S}_D^{-1}\mathcal{P}_1)^2E_1 + \ldots +(\vect{S}_D^{-1}\mathcal{P}_1)^n E_1+\ldots\\
 =&\ (\vect{S}_D-\mathcal{P}_1)^{-1}\widetilde X
\end{align*}
which solves
\begin{align}
(\vect{S}_D-\mathcal{P}_1) E_\infty = \widetilde X, \ \ E_\infty(0, \cdot) = 0.
\end{align}
Moreover, we see that
\begin{align*}
\Lp{E_\infty}{\mathcal{F}} \lesssim&\ \Lp{E_1}{\mathcal{F}}+ \Lp{\vect{S}_D^{-1}\mathcal{P}_1E_1}{\mathcal{F}}+\ldots + \Lp{(\vect{S}_D^{-1}\mathcal{P}_1)^nE_1}{\mathcal{F}}+\ldots\\
\lesssim&\ N^{\frac{\beta-1}{2}}\exp(\kappa T)+C(N^{\frac{\beta-1}{2}}\exp(\kappa T))^2+\ldots \\
\lesssim&\ N^{\frac{\beta-1}{2}}\exp(\kappa T)
\end{align*}
for some $C>0$ when $N$ is sufficiently large. Finally, suppose $E(t)$ is a solution to \eqref{inhomo2} then we see that $E-E_\infty$ satisfies 
\begin{align}
(\vect{S}_D-\mathcal{P})(E-E_\infty) = \mathcal{P}_2 E_\infty.
\end{align}
By energy estimates, we see that
\begin{align*}
\Lp{E(t)-E_\infty(t)}{\mathcal{F}} \lesssim&\ \int^t_0 ds\ \Lp{\mathcal{P}_2E_\infty(s)}{\mathcal{F}}\\\
\lesssim&\ \int^t_0 ds\ \llp{E_\infty(s)}_\mathcal{F} \lesssim N^{\frac{\beta-1}{2}}\exp(\alpha T)
\end{align*}
if $\llp{\mathcal{P}_2E_\infty(t)}_\mathcal{F}\lesssim \llp{E_\infty(t)}_\mathcal{F}$.
Then the result follows from the triangle inequality. 
 Thus, to establish Proposition \ref{iteration}, it suffices to prove $\Lp{\mathcal{P}_2E_\infty(t)}{\mathcal{F}}\lesssim \Lp{E_\infty(t)}{\mathcal{F}}$. 
 Unfortunately, the claim is false as stated since we can only show that $\mathcal{P}_2$ is a bounded operator on the first $k$ sectors of $\mathcal{F}$ not the
 the entire $\mathcal{F}$. 
 
 To fix this problem, the proof will only consider the  $M$-truncation (iteration) of the series $E_\infty$, i.e. we consider
\begin{align}
E_M:= E_1+\vect{S}_D^{-1}\mathcal{P}_1E_1+(\vect{S}_D^{-1}\mathcal{P}_1)^2E_1 + \ldots +(\vect{S}_D^{-1}\mathcal{P}_1)^{M-1} E_1,
\end{align}
where $M$ depends on $\beta$. 
\end{remark}

\begin{proof}[Proof of Proposition \ref{iteration}]
Let $0<\beta<1$. Choose $\alpha$ so that Theorem 3.3 in \cite{GM4} holds and $M\ge \lceil\frac{2}{1-\beta}\rceil+1$.

Following Remark \ref{main-remark}, we let $E_1$ be the solution to (\ref{diag-inhomo}) then consider
\begin{align*}
E_M= E_1+\vect{S}_D^{-1}\mathcal{P}_1E_1+(\vect{S}_D^{-1}\mathcal{P}_1)^2E_1 + \ldots +(\vect{S}_D^{-1}\mathcal{P}_1)^{M-1} E_1
\end{align*}
which have at most $k=5+4\times M$ nonzero sectors since $\mathcal{P}$ is a fourth-order operator in $a^\ast, a$. Next, observe $E_M$ satisfies
\begin{align}
\vect{S}_D(I-\vect{S}_D^{-1}\mathcal{P}_1)E_M = \widetilde X -\mathcal{P}_1(\vect{S}_D^{-1}\mathcal{P}_1)^{M-1}E_1
\end{align}
then, by Strichartz estimates, we have that
\begin{equation}\label{EM}
\begin{aligned}
\Lp{(I-\vect{S}_D^{-1}\mathcal{P}_1)E_M(t)}{\mathcal{F}} \lesssim&\ \llp{\widetilde X}_{S'}+\Lp{\mathcal{P}_1(\vect{S}_D^{-1}\mathcal{P}_1)^{M-1}E_1}{S'}\\
  \lesssim&\ N^{\frac{\beta-1}{2}}\exp(\kappa T)
\end{aligned}
\end{equation}
when $N$ is sufficiently large. 
Using reverse triangle inequality on the LHS of \eqref{EM}, it follows
\begin{align}\label{EM-ineq}
\Lp{E_M(t)}{\mathcal{F}} \lesssim&\ N^{\frac{\beta-1}{2}}\exp(\kappa T).
\end{align}
Moreover, let $E(t)$ be a solution to \eqref{inhomo2} then it follows $E-E_M$ satisfies the equation
\begin{align}
\left(\vect{S}_D-\mathcal{P}\right)(E-E_M) = \mathcal{P}_2 E_{M-1}+\mathcal{P}(\vect{S}_D^{-1}\mathcal{P}_1)^{M-1}E_1.
\end{align}
Using energy estimate yields 
\begin{align*}
\Lp{E(t)-E_M(t)}{\mathcal{F}} 
\lesssim&\ \int^t_0ds\ \{\Lp{\mathcal{P}_2E_M(s)}{\mathcal{F}}+\Lp{\mathcal{P}(\vect{S}_D^{-1}\mathcal{P}_1)^{M-1}E_1(s)}{\mathcal{F}}\}\\
\lesssim&\ N^{\frac{\beta-1}{2}}\exp(\kappa T)+N^{1+M\frac{\beta-1}{2}}\exp\left(\kappa'T\right).
\end{align*}
Again, by the reverse triangle inequality, we arrive at the desired conclusion. In particular, we have proven Theorem \ref{main-theorem-2}. 
\end{proof}

\section{Dispersive Estimates for $\mathcal{P}$}
In this section we obtain estimates for the Fock space operator $\mathcal{P}$ which was used in the previous section. We start by splitting
$\mathcal{P}$ into $\mathcal{P}_1+\mathcal{P}_2$ and show that $\mathcal{P}_2$ has uniform in $N$ bounded operator norm 
(but grows as a function of $T$) whereas $\mathcal{P}_1$ can be controlled uniformly in $N$ by Strichartz estimates on Fock space 
but has operator norm that grows in $N$. 
More precisely, we let $\mathcal{P}_1 = \mathcal{C}+\mathcal{D}$ and $\mathcal{P}_2=\mathcal{Q}$.  

Let us consider estimates for $\mathcal{P}_1$. The proof is similar to the prove given in \cite{GM3}. 

\begin{prop}
 Fix $k \in \nn$. If $X$ is a Fock space vector that has nonzero entries only in the first $k$ sectors, then we have the estimates
 \begin{subequations}
 \begin{align}
  \llp{\mathcal{C}X}_{S'} \lesssim_k&\ N^{\frac{\beta-1}{2}}\exp\left(\kappa T\right)\llp{X}_S \label{Strichartz-cubic}\\
   \llp{\mathcal{C}X}_{\mathcal{F}} \lesssim_k&\ N\exp\left(\kappa T\right)\llp{X}_\mathcal{F} \label{f-f-estimate}.
 \end{align}
 \end{subequations}
\end{prop}

\begin{proof}
 Write $\ch(k) = \delta + p$. Then letting \eqref{T1} acts on the Fock vector 
 $X = (0, \ldots, 0, F(x_1, \ldots, x_{n+3}), 0, 0, \ldots)$ gives us the vector with
 \begin{subequations}
 \begin{align}
  (\eqref{T1}X)_n\sim &\int dx dydz\ \frac{v_N(x-y)}{\sqrt{N}}\bar\phi(y)\overline{u}(x, z)F(x, y, z, x_1, \ldots, x_n) \label{term1} \\
  & +\int dxdydz dx'\  \frac{v_N(x-y)}{\sqrt{N}}\bar\phi(y)\overline{u}(x, z)\bar p(x, x')F(y, z, x', \ldots)\label{term2}\\
  & +\int dx dydz dx'dy'\  \frac{v_N(x-y)}{\sqrt{N}}\bar\phi(y)\overline{u}(x, z)\bar p(x, x')
  \bar p(y, y')F(x', y', z, \ldots)\label{term3}\\
  & + \text{similar terms} \nonumber
 \end{align}
 \end{subequations}
 on the $n$th sector and zero elsewhere.

 Let us focus on establishing \eqref{Strichartz-cubic} for the first term. Since $x_1, \ldots, x_n$
 play only passive roles in the estimate, let us temporarily suppress them in the notation. We begin by writing
 $F(x, y, z) = G(x-y, x+y, z)$. Then, by H\"older's inequality and Sobolev embedding, we have the estimate
 \begin{align*}
|\eqref{term1}| \leq&\ \int dzd(x-y)d(x+y)\ \frac{v_N(x-y)}{\sqrt{N}}|\bar\phi(y) \bar u(x, z)G(x-y, x+y, z)| \\
\lesssim&\ \int d(x-y)\ \frac{v_N(x-y)}{\sqrt{N}} \llp{\phi}_{L^6(dx')} \llp{u}_{L^3(dx')L^2(dz)}\llp{G(x-y, \cdot, \cdot)}_{L^2(d(x+y)dz)}\\
\lesssim&\ \frac{1}{\sqrt{N}}\llp{v_N}_{L^{\frac{6}{5}}}\llp{|\grad|^{\frac{1}{2}}\phi}_{L^3(dx)} 
\llp{u}_{L^3(dx)L^2(dy)}\llp{F}_{L^6(d(x-y))L^2(d(x+y)dz)}.
 \end{align*}
Hence, by Corollary \ref{sh-strichartz-corollary}, we have the estimate
\begin{align*}
 &\llp{\eqref{term1}}_{L^1([0, T])L^2(dx_1\cdots dx_n)}\\
&\lesssim\ \frac{1}{\sqrt{N}}\llp{v_N}_{L^\frac{6}{5}}\llp{\grad^{\frac{1}{2}}\phi}_{L^4([0, T])L^3(dx)}
\llp{u}_{L^4([0, T])L^3(dx)L^2(dy)}\llp{X}_{S}\\
&\lesssim\ N^{\frac{\beta-1}{2}}\exp(\kappa T)\llp{X}_{S}.
 \end{align*}
For \eqref{term2}, we see that
\begin{align*}
|\eqref{term2}|&\leq \int dxdydz dx'\  \frac{v_N(x-y)}{\sqrt{N}}|\bar\phi(y)\overline{u}(x, z)\bar p(x, x')F(x', y, z)|\\
&\lesssim\ \int d(x-y)d(x+y)\ \frac{v_N(x-y)}{\sqrt{N}}|\phi(y)|\\
&\quad\times\Lp{u(x, \cdot)p(x, \cdot)}{L^2(dzdx')}\llp{F(y, \cdot, \cdot)}_{L^2(dzdx')}\\
&\lesssim\ N^{-\frac{1}{2}} \llp{\phi}_{L^2(dx)}\Lp{u}{L^\infty(dx)L^2(dy)}\Lp{p}{L^\infty(dy)L^2(dz)}\llp{F}_{L^2(dxdydz)}
\end{align*}
which then, by Corollary \ref{sh-corollary}, it follows
\begin{align*}
 & \llp{\eqref{term2}}_{L^1([0, T])L^2(dx_1\cdots dx_n)}\\
 &\lesssim N^{-\frac{1}{2}}\llp{|\grad|^{\frac{1}{2}+\varepsilon}\phi}_{L^\infty(dt)L^2(dx)}
 \Lp{u}{L^2([0, T])L^6(dx)L^2(dy)}\Lp{p}{L^2([0, T])L^6(dx)L^2(dy)}\llp{X}_{S}\\ 
&\lesssim N^{-\frac{1}{2}}\exp\left(\kappa T\right)\llp{X}_{S}.
 \end{align*}
 Lastly, we see that
 \begin{align*}
  |\eqref{term3}| 
\lesssim&\ N^{-\frac{1}{2}} \llp{\phi}_{L^\infty(dx)}\llp{u}_{L^\infty(dx)L^2(dy)} \llp{p}_{L^\infty(dx)L^2(dy)}^2 \llp{F}_{L^2(dxdydz)}
\end{align*}
which again means
\begin{align*}
 &\llp{\eqref{term3}}_{L^1([0, T])L^2(dx_1\cdots dx_n)}\lesssim N^{-\frac{1}{2}}\exp\left(\kappa T\right)\llp{X}_{S}.
\end{align*}
Thus, in general, the $p$ term in $\ch(k)$ are lower order in $N$ but they contribute to the growth in time in our estimates. 

For the Fock space to Fock space estimate\eqref{f-f-estimate} of \eqref{T1}, let us first estimate \eqref{term1}. By Cauchy-Schwarz inequality, it
suffices to estimate
\begin{align*}
 &\frac{1}{\sqrt{N}}\llp{v_N(x-y)\bar\phi(y)\bar u(x, z)}_{L^2(dxdydz)}\\
 &\lesssim \frac{1}{\sqrt{N}}\llp{v_N}_{L^2}\llp{\phi}_{L^\infty(dx)}\llp{u}_{L^2(dxdy)} \lesssim N\exp(\kappa T)
\end{align*}
which follows from Theorem \ref{sh2-thm}. The other two terms are handled in the same way with the only difference being 
the additional multiplication by $\llp{p}_{L^2(dxdy)}$.

The estimates for the adjoint of \eqref{T1}, contained in $\mathcal{C}_1^\ast$, is obtained by the dual estimate of \eqref{T1}. 

Next, let \eqref{T3} acts on $X=(0, \ldots, 0, F(x_1, \ldots, x_{n+1}), 0, \ldots)$ yields
\begin{subequations}
\begin{align}
&\int dy\ \frac{v_N(x_1-y)}{\sqrt{N}}\bar\phi(y)F( y, x_1, \ldots, x_n)\label{term4}\\
&+ \text{similar or lower order terms} \label{LOT2}
\end{align}
\end{subequations}
on the $n$ sector and zeroes elsewhere. Again, since $x_2,\ldots, x_n$ are passive variable, let us write $F(y, x_1) 
= G(x_1-y, x_1+y)$. Then it follows
\begin{align*}
\llp{\eqref{term4}}_{L^2(dx_1)}&\lesssim \int dy\ \frac{v_N(y)}{\sqrt{N}}\llp{\bar\phi(x_1-y)G(y, 2x_1-y)}_{L^2(dx_1)}\\
&\lesssim \int dy\ \frac{v_N(y)}{\sqrt{N}} \llp{\phi}_{L^\infty(dx)}
\llp{G(y, \cdot)}_{L^2(d x)}\\
&\lesssim N^{\frac{\beta-1}{2}}\llp{\phi}_{L^\infty(dx)}\llp{F}_{L^6(d(x-y))L^2(d(x+y))}
 \end{align*}
which means
\begin{align*}
 &\llp{\eqref{term4}}_{\mathcal{S}'} \lesssim \llp{\eqref{term4}}_{L^1(dt)L^2(dx_1\cdots dx_n)}\\
&\lesssim\ N^{\frac{\beta-1}{2}}\llp{\grad^{\frac{1}{2}+\varepsilon}\phi}_{L^2(dt)L^6(dx)}\llp{F}_{L^2(dt)L^6(d(x_1-y))L^2(d(x_1+y) dx_2\cdots dx_n)}\\
 &\lesssim\ N^{\frac{\beta-1}{2}} T \llp{X}_{S}.
 \end{align*}
The estimate for \eqref{LOT2} is obtained in a similar manner as above and the Fock space to Fock space estimate is clear. Moreover, the adjoint of \eqref{T3} is again obtained from
the dual estimate. 

The estimates for \eqref{T4}-\eqref{Tlast} are obtain in a similar manner. 
\end{proof}

Next, we estimate the quartic terms. 
\begin{prop}
 Fix $k \in \nn$ and let $0<\beta<1$. If $X$ is a Fock space vector that has nonzero entries only in the first $k$ sectors, then we have the estimates
 \begin{subequations}
 \begin{align}
  \llp{\mathcal{D}X}_{S'} \lesssim_k&\ N^{\beta-1}\exp\left(\kappa T\right)\llp{X}_S \label{Strichartz-quartic}\\
   \llp{\mathcal{D}X}_{\mathcal{F}} \lesssim_k&\ \sqrt{N}\exp\left(\kappa T\right)\llp{X}_\mathcal{F} \label{f-f-estimate2}.
 \end{align}
 \end{subequations}
\end{prop}

\begin{proof}
 Letting \eqref{D1} act on $F_{n-4}(x_1, \ldots, x_{n-4})$ yields
 \begin{subequations}
 \begin{align}
  \eqref{D1}F_{n-4}\sim &\frac{1}{2N}v_N(x_1-x_2)u(x_1, x_3)u(x_2, x_4)F_{n-5}(x_5, \ldots, x_n) \label{qterm1}\\
   &\frac{1}{2N}\int dy\ u(x_1, x_3)v_N(x_1-y)p(y, x_2)u(y, x_4)F_{n-5}(x_5, \ldots, x_n) \label{qterm2}\\
  &\frac{1}{2N}\int dxdy\ u(x, x_3)p(x, x_1)v_N(x-y)p(y, x_2)u(y, x_4)F_{n-5}(x_5, \ldots, x_n) \label{qterm3}\\
 &+\text{similar or lower order terms}. \label{LOT3}
 \end{align}
 \end{subequations}
Since we have the estimate
\begin{align*}
&\frac{1}{2N} \llp{v_N(x_1-x_2)u(x_1, x_3)u(x_2, x_4)}_{L^{\frac{3}{2}}(d(x_1-x_2)d(x_1+x_2))L^2(dx_3dx_4)} \\
&\leq \frac{1}{2N}\llp{v_N}_{L^{\frac{3}{2}}}\llp{u}_{L^3(dx)L^2(dy)}^2
\end{align*}
then it follows
\begin{align*}
 \llp{\eqref{qterm1}}_{S'} \lesssim N^{\beta-1}\llp{u}_{L^4([0, T])L^3(dx)L^2(dy)}^2\llp{X}_{S} 
 \lesssim N^{\beta-1}\exp(\kappa T)\llp{X}_{S}.
\end{align*}
To estimate \eqref{qterm2}, observe we have the estimate
\begin{align*}
&\frac{1}{2N} \llp{\int dy\ v_N(x_1-y)u(x_1, x_3)p(y, x_2)u(y, x_4)}_{L^{2}(dx_1 dx_2dx_3dx_4)}\\
&\lesssim \frac{1}{2N} \llp{v_N}_{L^1}\llp{u}_{L^2(dxdy)}\llp{u}_{L^\infty(dx)L^2(dy)}\llp{p}_{L^\infty(dx)L^2(dy)}
\end{align*}
which means
\begin{align*}
  &\llp{\eqref{qterm2}}_{L^1(dt)L^2(dx_1\cdots dx_n)}\lesssim N^{-1}\exp(\kappa T)\llp{X}_S.
\end{align*}
Similar, we see that
\begin{align*}
 &\llp{\eqref{qterm3}}_{L^1(dt)L^2(dx_1\cdots dx_n)} \\
 &\lesssim N^{-1}T\llp{u}_{L^\infty([0, T])L^2(dx)L^2(dy)}^2
 \llp{p}_{L^\infty([0, T])L^\infty(dx)L^2(dy)}^2\llp{X}_S\\
 &\lesssim N^{-1}\exp(\kappa T)\llp{X}_S.
\end{align*}
To obtain the Fock space to Fock space estimate, we see that
\begin{align*}
 &\frac{1}{2N}\llp{v_N(x_1-x_2)u(x_1, x_3)u(x_2, x_4)}_{L^2(dx_1dx_2dx_3dx_4)}\\
 &\lesssim  N^{-1}\llp{v_N}_{L^2}\llp{u}_{L^\infty(dx)L^2(dy)}\llp{u}_{L^2(dxdy)} 
\end{align*}
then it follows
\begin{align*}
 \llp{\eqref{qterm1}}_{L^2(dx_1\cdots dx_n)}\lesssim N^\frac{1}{2}\exp(\kappa T)\llp{X}_\mathcal{F}.
\end{align*}
The Fock space to Fock space estimate for \eqref{qterm2} and \eqref{qterm3} follow a similar calculation. 

Next, let us consider the action of \eqref{D2} and \eqref{D3} on $F_{n-2}(x_1, \ldots, x_{n-2})$. Observe we have
\begin{subequations}
\begin{align}
 \eqref{D2}F_{n-2} \sim&\ \frac{1}{2N}v_N(x_1-x_2)u(x_1, x_3)F_{n-2}(x_2, x_4, \ldots, x_n)\label{qterm6}\\
    &+\text{similar or lower order terms} \label{LOT4}
\end{align}
\end{subequations}
It is similar for \eqref{D3}. Since $x_5, \ldots, x_n$ are passive variables, let us suppress the notation. Then we see that
\begin{align*}
 &\frac{1}{2N} \llp{v_N(x_1-x_2)u(x_1, x_3)F(x_2, x_4)}_{L^{\frac{3}{2}}(d(x_1-x_2)d(x_1+x_2))L^2(dx_3dx_4)}\\
 &\lesssim \frac{1}{2N} \llp{v_N}_{L^\frac{3}{2}}\llp{u}_{L^6(dx)L^2(dy)}\llp{F}_{L^2(dxdy)}
\end{align*}
which leads to the desired estimate. The Fock space to Fock space estimate is similar to the above estimate for \eqref{qterm1}. 

The estimates for \eqref{D4}, \eqref{D5}, and \eqref{D8} follow from the same arguments as \eqref{qterm2} and \eqref{qterm3}.

To estimate \eqref{D6}, observe we have
\begin{subequations}
\begin{align}
 \eqref{D6}F_n \sim& \frac{1}{2N} \int dydz\ v_N(x_1-y)u(x_1, x_2)\bar u(y, z)F_n(y, z, x_3,\ldots, x_n)\label{qterm7}\\
 &+\text{similar or lower order terms} \label{LOT5}. 
\end{align}
\end{subequations}
Then it follows
\begin{align*}
&\llp{\eqref{qterm7}}_{L^2(dx_1dx_2)}\\
 &\lesssim \frac{1}{2N}\int d\eta\ v_N(\eta)\llp{ \int dz\ u(x_1, x_2)\bar u(x_1-\eta, z)F_n(x_1-\eta, z)}_{L^2(dx_1dx_2)}\\
 &\lesssim N^{-1}\llp{u}_{L^\infty(dx)L^2(dy)}^2\llp{F}_{L^2(dx_1dx_2)}
\end{align*}
which yields the desired estimate. Likewise, we can estimate \eqref{D7}. 

Lastly, when we remove \eqref{D10} from \eqref{D9}, the remainder parts can be handle in the same way as in the cases of \eqref{D2} and \eqref{D3}. 
\end{proof}

Let us now consider the estimates for the quadratic terms. 
\begin{prop}
 Fix $k \in \nn$. Then $\mathcal{Q}$, defined by \eqref{quadratic-terms}, is bounded uniformly in $N$ from the first $k$ sectors of Fock space
 to Fock space. More precisely, let $X$ be a Fock vector that has nonzero entries only in the first $k$ sectors, then  we have the 
 following uniform in $N$ estimate
 \begin{align}
  \llp{\mathcal{Q}X}_\mathcal{F} \lesssim_{\varepsilon, k} \exp(\kappa T) \llp{X}_\mathcal{F}
 \end{align}
for some $\kappa>0$.
\end{prop}

\begin{proof}
 Letting \eqref{quadratic-part1} acts on the Fock vector 
 $X = (0, \ldots, 0, F_n(x_1, \ldots, x_n), 0, \ldots 0)$, where $n\leq k$, yields
 \begin{align*}
  (\eqref{quadratic-part1}X)_n =&\ \sum^n_{j=1}(v_N\ast\diag\Gamma)(t, x_j)F_n(x_1, \ldots , x_n)\\
  &\ \sum_{j=1}^n\int dy\ (v_N\Gamma)(t, x_j, y)F_n(y, x_1, \ldots,\hat x_j, \ldots, x_n)
 \end{align*}
and zero elsewhere. It suffices to consider the first summand of the two summations. Also, since $x_3,\ldots, x_n$ are passive variables in the first summand, we will suppress them without further comment. 

Observe, by H\"older, Sobolev, and Proposition \ref{4.9}, we have the estimate
\begin{align*}
 &\llp{(v_N\ast\diag\Gamma)(t, x_1)F(x_1, x_2)}_{L^2(dx_1 dx_2)} \\
 &\lesssim\ \vect{\dot N}_T(\grad_{x+y}^2\Gamma)\llp{F}_{L^2(dx_1 dx_2)}\lesssim \exp\left(\kappa T\right)\llp{F}_{L^2(dx_1 dx_2)}
\end{align*}
and likewise we have
\begin{align*}
 &\llp{\int dy\ v_N(x_1-y)\Gamma(x_1, y)F(y, x_2)}_{L^2(dx_1 dx_2)}\lesssim \exp\left(\kappa T\right)\llp{F}_{L^2(dx_1 dx_2)}.
\end{align*}

Next, letting \eqref{quadratic-part2} acts on $X$ yields
\begin{align*}
(\eqref{quadratic-part2}X)_n = \int dy\ w(x_1, y)F(y, x_2, \ldots, x_n)+\text{similar terms}
\end{align*}
and zero elsewhere. Then, by Cauchy-Schwarz inequality, we see that
\begin{align*}
 \llp{\int dy\ w(x_1, y)F(y, x_2)}_{L^2(dx_1 dx_2)}\lesssim \llp{w}_{L^2(dxdy)}\llp{F}_{L^2(dx_1 dx_2)}.
\end{align*}

To estimate $\llp{w}_{L^2(dxdy)}$ uniformly in $N$, we divide the estimate into two parts. For the first part, we consider
\begin{subequations}
\begin{align}\label{w-part1}
\llp{\sh(k)\circ\overline{(v_N\Lambda)}\circ \overline{\ch(k)}^{-1}}_{L^2(dxdy)}
\end{align}
which can be estimated as follows
\begin{align*}
\eqref{w-part1}\lesssim \vect{N}_T(\grad_{x+y}^2\Lambda)\llp{\sh(k)}_{L^2(dxdy)}\lesssim \exp(\kappa T).
\end{align*}
The estimate for the second part 
\begin{align}
 \llp{ \vect{\tilde W}(\overline{\ch(k)})\circ \overline{\ch(k)}^{-1}}_{L^2(dxdy)}
\end{align}
\end{subequations}
follows from Lemma \ref{w(ch)-est}. 
\end{proof}

\input{HFB_revised.bbl}


\end{document}

%% file: HFB_revised.bbl
\newcommand{\etalchar}[1]{$^{#1}$}
\providecommand{\bysame}{\leavevmode\hbox to3em{\hrulefill}\thinspace}
\providecommand{\MR}{\relax\ifhmode\unskip\space\fi MR }
\providecommand{\MRhref}[2]{%
  \href{http://www.ams.org/mathscinet-getitem?mr=#1}{#2}
}
\providecommand{\href}[2]{#2}